\documentclass[reqno,12pt]{amsart}

\usepackage[all]{xypic}
\usepackage{enumerate}
\usepackage{hyperref}
\usepackage{amssymb}
\usepackage{fullpage}
\usepackage{color}
\usepackage{scalefnt}

\DeclareMathOperator{\pv}{pv}

\DeclareMathOperator{\img}{img}
\DeclareMathOperator{\coker}{coker}
\DeclareMathOperator{\ind}{ind}
\DeclareMathOperator{\tind}{t-ind}
\DeclareMathOperator{\supp}{supp}
\DeclareMathOperator{\gr}{gr}
\DeclareMathOperator{\ev}{ev}
\DeclareMathOperator{\rank}{rank}
\DeclareMathOperator\singsupp{sing-supp}
\DeclareMathOperator{\tr}{tr}
\newcommand\goe{\mathfrak g}
\newcommand\poe{\mathfrak p}
\newcommand\noe{\mathfrak n}
\newcommand\zoe{\mathfrak z}
\newcommand\llangle{\langle\!\langle}
\newcommand\rrangle{\rangle\!\rangle}
\newcommand\bbb{|\mkern-2mu|\mkern-2mu|}
\DeclareMathOperator{\Ad}{Ad}

\DeclareMathOperator{\ad}{ad}
\DeclareMathOperator{\id}{id}
\DeclareMathOperator{\Aut}{Aut}
\DeclareMathOperator{\eend}{end}

\newcommand{\DO}{\mathcal D\mathcal O}
\newcommand{\tDO}{\widetilde{\DO}}

\newcommand\Z{\mathbb Z}
\newcommand\N{\mathbb N}
\newcommand\R{\mathbb R}
\newcommand\C{\mathbb C}

\newcommand\ii{\mathbf i}

\newcommand\op{\textrm{op}}
\newcommand\cl{\textrm{classical}}
\newcommand\prop{\textrm{prop}}
\newcommand\loc{\textrm{loc}}

\newcommand\itemref[1]{(\ref{#1})}

\newcounter{ABC}

\theoremstyle{plain}
  \newtheorem{theorem}{Theorem}[section]
  \newtheorem{corollary}[theorem]{Corollary}
  \newtheorem{lemma}[theorem]{Lemma}
  \newtheorem{proposition}[theorem]{Proposition}
  \newtheorem{thm}[ABC]{Theorem}
\theoremstyle{definition}
  \newtheorem{definition}[theorem]{Definition}
  \newtheorem{example}[theorem]{Example}
  \newtheorem{remark}[theorem]{Remark}

\begin{document}

\title[]{Graded hypoellipticity of BGG sequences}

\author{Shantanu Dave}


\address{Shantanu Dave,
         Faculty of Mathematics,
         University of Vienna,
         Oskar-Morgenstern-Platz 1,
         1090 Vienna,
         Austria.}

\email{shantanu.dave@univie.ac.at}

%
%

\author{Stefan Haller}

\address{Stefan Haller, 
         Department of Mathematics,
         University of Vienna,
         Oskar-Morgenstern-Platz 1,
         1090 Vienna,
         Austria.}

\email{stefan.haller@univie.ac.at}


\begin{abstract}
This article studies hypoellipticity on general filtered manifolds.
We extend the Rockland criterion to a pseudodifferential calculus on filtered manifolds, construct a parametrix and describe its precise analytic structure.
We use this result to study Rockland sequences, a notion generalizing elliptic sequences to filtered manifolds.
The main application that we present is to the analysis of the Bernstein--Gelfand--Gelfand (BGG) sequences over regular parabolic geometries. 
We do this by generalizing the BGG machinery to more general filtered manifolds (in a non-canonical way) and show that the generalized BGG sequences are Rockland in a graded sense.
\end{abstract}



\maketitle

\section{Introduction and main results}\label{S:intro}

Elliptic operators and their analysis touches a large number of problems in geometry, analysis, topology, and physics.
At the heart of their wide applicability are three simple reasons:
\begin{enumerate}[(1)]
\item
There is a big supply of natural, invariant elliptic operators such as the Laplace and Dirac operators, which encode Riemannian and other geometry.
\item
Elliptic operators are defined by invertibility of their ``highest order'' term.
However, this suffices to guarantee that they are invertible up to a smoothing error.
This approximate inverse is known as a parametrix and it ensures many of the usual analytic properties of elliptic operators, such as their Fredholmness on compact manifolds.
\item
The heat kernel asymptotic for positive elliptic operators also encodes several important invariants of geometry and topology.
\end{enumerate}

The purpose of this article is to show that in many other geometric situations there is a large supply of natural hypoelliptic operators which admit a nice parametrix in a suitable calculus \cite{M82,CGGP92,EY15v5}.
These operators include the curved Bernstein--Gelfand--Gelfand (BGG) operators on regular parabolic geometries \cite{CSS01} and Rumin's complexes \cite{R90,R94,R99,R01} on Carnot--Carath\'eodory (C-C) manifolds.
The heat kernel asymptotics for positive differential operators in this class resembles the elliptic case, as has been shown in another article \cite{DH17b}.
This class of hypoelliptic operators therefore enjoys the same properties as formulated above for elliptic operators.

Historically, finding geometric hypoelliptic operators has been a dream.
In the few cases where they are known they lead to striking results, for instance the Connes--Moscovici \cite{CM95} index formula for the transverse signature operators on foliations, or Julg--Kasparov's \cite{JK95} proof of the Baum--Connes conjecture for discrete subgroups of $\operatorname{SU}(n,1)$.
In both the examples above, the hypoellipticity was known through the Heisenberg calculus, see \cite{BG88,T84,P08}.
Another substantial source of hypoelliptic operators is H\"ormander's celebrated sum of squares theorem \cite{H67}.

There are many problems in geometry and operator theory, for example the extension of the results by Connes--Moscovici and Julg--Kasparov mentioned above, where the underlying manifold is a filtered manifold and the natural representative of a K-homology class is expected to be a geometric hypoelliptic operator, see \cite{E10a, E10b, Y11, BE14, CCH16} for instance. 
The K-homology class represented by the classical BGG sequences on generalized flag varieties are expected to play a crucial role in extending Julg and Kasparov's approach to higher rank Lie groups.
We will construct similar hypoelliptic sequences for a large class of filtered manifolds.

To obtain precise analytic properties of a hypoelliptic operator, we want its parametrix to be in a suitable pseudodifferential calculus.
This will for instance provide precise maximal hypoellipticity estimates in the corresponding Sobolev spaces.
Such a pseudodifferential calculus for filtered manifold was first described in Melin \cite{M82}.
Melin's original preprint remains unpublished.
A new geometric approach to the calculus was developed in van~Erp and Yuncken \cite{EY15v5}.
We shall follow \cite{EY15v5} for definiteness, although arguably both methods \cite{M82,EY15v5} very likely produce the same calculus.
Van Erp and Yuncken's approach is based on the construction of a Heisenberg tangent groupoid for filtered manifolds \cite{EY16,CP15,HH18,M21}. 
We will refer to the calculus as Heisenberg calculus.
The construction of the Heisenberg calculus follows the geometrical insights of Debord and Skandalis \cite{DS14} into classical pseudodifferential calculus using tangent groupoids \cite{C94,H10,RLM14}.
The operators in the Heisenberg  calculus are classical in the sense that they admit local homogeneous expansions.
Our main analytic result shows that operators satisfying a pointwise Rockland condition admit a parametrix in the Heisenberg calculus.

In the remaining part of this introductory section we will outline the main results and hint at their proofs.

\subsection{Filtered manifolds and their osculating groups}\label{SS:filtered}

A natural structure available on every smooth manifold is its Lie algebra of vector fields.
Various geometric structures on smooth manifolds can be described in terms of a filtration on the tangent bundle which is compatible with the Lie bracket of vector fields.
This compatibility is subsumed in the concept of a filtered manifold.
A filtered manifold is a smooth manifold $M$ together with a filtration of its tangent bundle by smooth subbundles,
\begin{equation}\label{E:filtM}
TM=T^{-r}M\supseteq\cdots\supseteq T^{-2}M\supseteq T^{-1}M\supseteq T^0M=0,
\end{equation}
which is compatible with the Lie bracket of vector fields in the sense that $[X,Y]\in\Gamma^\infty(T^{p+q}M)$ for all $X\in\Gamma^\infty(T^pM)$ and $Y\in\Gamma^\infty(T^qM)$.

To each point $x$ in a filtered manifold $M$ one can assign a simply connected nilpotent Lie group $\mathcal T_xM$, called the osculating group at $x$, which can be regarded as a non-commutative analogue of the tangent space at $x$.
Its Lie algebra is
$$
\mathfrak t_xM:=\gr(T_xM):=\bigoplus_pT_x^pM/T_x^{p+1}M
$$
with the (Levi) bracket induced from the Lie bracket of vector fields.
The osculating algebras combine to form a smooth bundle of graded nilpotent Lie algebras $\mathfrak tM$ over $M$, called the bundle of osculating algebras.
Correspondingly, the osculating groups combine to form a smooth bundle of simply connected nilpotent Lie groups $\mathcal TM$ over $M$, called the bundle of osculating groups. 
These play a crucial role in the analysis on filtered manifolds.

\begin{remark}
There are various conventions on choosing orders/degrees on filtrations and gradings.
Van~Erp and Yuncken \cite{EY15v5,EY16}, for instance, assign positive values to gradings.
We follow the convention prevalent in parabolic geometry, see for example \cite{CS09,M93}.
In this context the choice of negative degree/order is a well established convention.
\end{remark}

To describe simple examples, suppose $\noe=\noe_{-r}\oplus\cdots\oplus\noe_{-1}$ is a graded nilpotent Lie algebra, and let $N$ be a Lie group with Lie algebra $\noe$.
Then the filtration $\noe=\noe^{-r}\supseteq\cdots\supseteq\noe^{-1}\supseteq\noe^0=0$, with $\noe^p:=\bigoplus_{p\leq q}\noe_q$, determines a left invariant filtration of $TN$ which turns $N$ into a filtered manifold with (locally) trivial bundle of osculating algebras and typical fiber $\noe$.

Let us mention a few geometric structures that can be described as filtered manifolds with special osculating algebras.
By Frobenius' theorem, foliated manifolds can equivalently be described as filtered manifolds with abelian, but non-trivially graded osculating algebras.
A contact manifold is just a filtered manifold with osculating algebras isomorphic to the Heisenberg algebra.
According to Darboux's theorem, the filtration on a contact manifold is even locally diffeomorphic to the left invariant filtration on the Heisenberg group.
Engel structures on 4-manifolds provide further examples of filtered manifolds that admit local normal forms, cf.~\cite{P16,V09} and Example~\ref{Ex:Engel}.
A generic rank two distribution in dimension five \cite{C10,BH93,S08,CS09,DH16} is a filtered manifold with osculating algebras isomorphic to a particular (generic) $5$-dimensional graded nilpotent Lie algebra, see Example~\ref{Ex:BGG235}.
Most regular normal parabolic geometries can equivalently be described as filtered manifolds with prescribed osculating algebras, see \cite[Proposition~4.3.1]{CS09}.
These include generic rank two distributions in dimension five, generic rank three distributions in dimension six \cite{B06}, and quaternionic contact structures \cite{B00}.

Heisenberg manifolds \cite{P08} constitute a well studied class of filtered manifolds for which the bundle of osculating algebras need not be locally trivial.
They occur naturally as boundaries of complex manifolds.
Equiregular Carnot--Carath\'eodory (C-C) manifolds \cite{G96} give rise to filtered manifolds which are often assumed to be bracket generating in the sense that $T^{-p}M$ is spanned by iterated Lie brackets of sections in $T^{-1}M$ which are of length at most $p$.

\subsection{Analytic results}\label{SS:analysis}

In this paper we will study operators on filtered manifolds and provide a criterion for their hypoellipticity.
Our main analytic results are based on the Heisenberg calculus for filtered manifolds.
This calculus can be obtained using the Heisenberg tangent groupoid \cite{EY15v5,EY16,CP15,HH18,M21} and the idea of essential homogeneity introduced in \cite{DS14}.
Although the calculus in \cite{EY15v5} is described for scalar operators it can easily be generalized to operators acting between sections of vector bundles.
The goal is to obtain a general Rockland type theorem in such a pseudodifferential calculus.
In addition to using the Heisenberg calculus, the proof of this theorem also builds upon harmonic analysis by Christ, Geller, G{\l}owacki, and Polin \cite{CGGP92} and arguments due to Ponge \cite{P08}.

Let us briefly describe the Heisenberg calculus and the related setup.
For two vector bundles $E$ and $F$ over a filtered manifold $M$, and any complex number $s$, we let $\Psi^s(E,F)$ denote the class of all pseudodifferential operators of Heisenberg order at most $s$.
These are continuous operators $\Gamma^\infty_c(E)\to\Gamma^\infty(F)$ which extend continuously to pseudolocal operators on distributional sections, $\Gamma^{-\infty}_c(E)\to\Gamma^{-\infty}(F)$.
They can be characterized as operators with a Schwartz kernel that admits an extension to the Heisenberg tangent groupoid which is essentially homogeneous of order $s$.
This extends the Heisenberg filtration on differential operators.
More precisely, a differential operator has Heisenberg order at most $k\in\N_0$ if and only if it is contained in $\Psi^k(E,F)$.

An operator $A\in\Psi^s(E,F)$ has a Heisenberg principal cosymbol $\sigma_x^s(A)\in\Sigma^s_x(E,F)$ at every point $x\in M$.
Here $\Sigma^s_x(E,F)$ denotes the space of regular distributional volume densities on the osculating group $\mathcal T_xM$ with values in $\hom(E_x,F_x)$ which are essentially homogeneous of order $s$, modulo smooth volume densities.
This cosymbol extends the Heisenberg principal (co)symbol of differential operators on filtered manifolds.
The basic properties of this operator class and the Heisenberg principal cosymbol are summarized in Proposition~\ref{P:Psi}.

Let $\pi\colon\mathcal T_xM\to U(\mathcal H)$ be a non-trivial irreducible unitary representation of the osculating group on a Hilbert space $\mathcal H$, and let $\mathcal H_\infty$ denote the subspace of smooth vectors.
If $A\in\Psi^s(E,F)$, then $\bar\pi(\sigma_x^s(A))$ is a well defined closed unbounded operator on $\mathcal H$, restricting to a map $\bar\pi(\sigma^s_x(A))\colon\mathcal H_\infty\otimes E_x\to\mathcal H_\infty\otimes F_x$.
The operator $A$ is said to satisfy the Rockland \cite{R78} condition if $\bar\pi(\sigma^s_x(A))$ is injective on $\mathcal H_\infty\otimes E_x$ for all non-trivial irreducible unitary representations $\pi$ of $\mathcal T_xM$ and every $x\in M$.

For left invariant differential operators on graded nilpotent Lie groups, Helffer and Nourrigat \cite{HN79} proved that the harmonic analytic Rockland condition implies hypoellipticity, thus confirming a conjecture of Rockland's \cite{R78} for the Heisenberg group.
Christ, Geller, G{\l}owacki, and Polin \cite{CGGP92} constructed a pseudodifferential operator calculus on graded nilpotent Lie groups and proved that the pointwise Rockland condition implies the existence of a parametrix in their calculus.
For Heisenberg manifolds with varying osculating algebras such a result has been obtained by Ponge \cite{P08}. 
As mentioned already, Melin \cite{M82} constructed a pseudodifferential calculus on general filtered manifolds and used it construct a parametrix for scalar Rockland differential operators. 
We will prove the following general Rockland type result for systems of pseudodifferential operators on general filtered manifolds.


\begin{thm}\label{thmA}
Let $E$ and $F$ be two vector bundles over a filtered manifold $M$ and suppose $A\in\Psi^s(E,F)$ satisfies the Rockland condition.
Then there exists a properly supported left parametrix $B\in\Psi^{-s}_\prop(F,E)$, that is, $BA-\id$ is a smoothing operator.
\end{thm}

As a consequence, every operator $A\in\Psi^s(E,F)$ satisfying the Rockland condition is hypoelliptic, that is, if $\psi$ is a compactly supported distributional section of $E$ such that $A\psi$ is smooth on an open subset $U$ of $M$, then $\psi$ was smooth on $U$.
Over closed manifolds this implies that $\ker(A)$ is a finite dimensional subspace of $\Gamma^\infty(E)$, see Theorem~\ref{T:Rockland} below.

Combining Theorem~\ref{thmA} with a result of Christ, Geller, G{\l}owacki, and Polin, see \cite[Theorem~6.1]{CGGP92}, we construct, for each complex number $s$, an operator $\Lambda_s\in\Psi^s(E)$ which is invertible mod smoothing operators, see Lemma~\ref{L:Lambda}.
This permits us to introduce a Heisenberg Sobolev scale, see Proposition~\ref{P:Hs}, and allows us to formulate more refined regularity statements, including maximal hypoelliptic estimates, for operators satisfying the Rockland condition, see Corollary~\ref{C:reg}.

We will use Theorem~\ref{thmA} to analyze Rockland sequences.
A sequence of operators,
$$
\cdots\to\Gamma^\infty(E_{i-1})\xrightarrow{A_{i-1}}\Gamma^\infty(E_i)\xrightarrow{A_i}\Gamma^\infty(E_{i+1})\to\cdots,
$$
where $A_i\in\Psi^{s_i}(E_i,E_{i+1})$ will be called Rockland sequence if the corresponding principal symbol sequence is exact in every non-trivial irreducible unitary representation $\pi\colon\mathcal T_xM\to U(\mathcal H)$ at each $x\in M$, that is, the sequence
$$
\cdots\to\mathcal H_\infty\otimes E_{x,i-1}\xrightarrow{\bar\pi(\sigma^{s_{i-1}}_x(A_{i-1}))}\mathcal H_\infty\otimes E_{x,i}\xrightarrow{\bar\pi(\sigma^{s_i}_x(A_i))}\mathcal H_\infty\otimes E_{x,i+1}\to\cdots
$$
is exact.
On trivially filtered manifolds this definition reduces to the well known concept of elliptic sequences.
To study these sequences we consider formal adjoints $A_i^*\in\Psi^{\bar s}(E_{i+1},E_i)$ with respect to standard $L^2$ inner products on $\Gamma^\infty(E_i)$.
Theorem~\ref{thmA} implies that $(A_{i-1}^*,A_i)$ is hypoelliptic, and more refined regularity statements, including maximal hypoelliptic estimates, can be formulated using the Heisenberg Sobolev scale.

On closed manifolds, in case the Rockland sequence forms a complex, i.e.\ $A_iA_{i-1}=0$, it is convenient to use suitable Sobolev adjoints, $A_i^\sharp$ of $A_i$.
These adjoints are constructed such that $A_{i-1}A_{i-1}^\sharp$ and $A_i^\sharp A_i$ have the same Heisenberg order and thus $B_i:=A_{i-1}A_{i-1}^\sharp+A_i^\sharp A_i$ is a Rockland operator.
We obtain a Hodge decomposition
$$
\Gamma^\infty(E_i)=\img(A_{i-1})\oplus\ker(B_i)\oplus\img(A_i^\sharp)
$$
where $\ker(B_i)=\ker(A_{i-1}^\sharp)\cap\ker(A_i)$.
In particular, each cohomology class has a unique harmonic representative, that is, $\ker(A_i)/\img(A_{i-1})=\ker(B_i)$.

For Rockland complexes of differential operators of positive order, one can alternatively follow the Rumin--Seshadri approach \cite{RS12} and consider 
\begin{equation}\label{E:IRS}
\Delta_i:=(A_{i-1}A_{i-1}^*)^{a_{i-1}}+(A_i^*A_i)^{a_i}
\end{equation}
where the numbers $a_i\in\N$ are chosen such that $\kappa:=s_{i-1}a_{i-1}=s_ia_i$.
Then $\Delta_i$ is a differential operator of order at most $2\kappa$ which satisfies the Rockland condition.
We obtain a similar Hodge decomposition, namely,
$$
\Gamma^\infty(E_i)=\img(A_{i-1})\oplus\ker(\Delta_i)\oplus\img(A_i^*),
$$
where $\ker(A_i)/\img(A_{i-1})=\ker(\Delta_i)=\ker(A_{i-1}^*)\cap\ker(A_i)$.

In order to study the generalized BGG sequences we will construct below, the analysis needs to be adapted to a setup where the operators act on sections of filtered vector bundles, cf.\ Rumin's concept of Carnot--Carath\'eodory (C-C) ellipticity in \cite{R99,R01}.
For any two filtered vector bundles $E$ and $F$ we consider a class of operators, $\tilde\Psi^s(E,F)$, which will be called pseudodifferential operators of graded Heisenberg order $s$.
If we identify $E$ and $F$ with the associated graded using splittings of the filtrations, then an operator $A\in\tilde\Psi^s(E,F)$ can be considered as a matrix with entries $A_{qp}\in\Psi^{s+q-p}(\gr_p(E),\gr_q(F))$.
The graded Heisenberg principal cosymbol, $\tilde\sigma^s_x(A)$, can be defined as the matrix obtained by taking the (ordinary) Heisenberg principal cosymbol of each entry, that is, 
$$
\tilde\sigma^s_x(A)=\sum_{p,q}\sigma^{s+q-p}_x(A_{qp}).
$$
Neither the operator class $\tilde\Psi^s(E,F)$, nor the graded Heisenberg principal cosymbol $\tilde\sigma^s_x(A)$ depend on the choice of splittings.

An operator $A\in\tilde\Psi^s(E,F)$ is called graded Rockland operator if $\bar\pi(\tilde\sigma^s_x(A))\colon\mathcal H_\infty\otimes\gr(E_x)\to\mathcal H_\infty\otimes\gr(F_x)$ is injective for all non-trivial irreducible unitary representations $\pi\colon\mathcal T_xM\to U(\mathcal H)$ and all $x\in M$.
Similarly, a graded Rockland sequence is defined to be a sequence of operators such that its graded Heisenberg principal symbol sequence is exact in each non-trivial irreducible unitary representation of $\mathcal T_xM$.
Theorem~\ref{thmA} has a graded analogue, and all the analysis mentioned above generalizes to this graded setup, see Section~\ref{S:grRockland}.
Even if a graded Rockland sequence is made of differential operators, its analysis requires conjugation by a pseudodifferential operator in the calculus, cf.~\cite{R99,R01}.
This is another reason why we need the generality of Theorem~\ref{thmA} to analyze generalized BGG sequences.

\subsection{Construction of Rockland sequences}\label{SS:construction}

In this paper will shall construct several examples of Rockland sequences.
The most basic sequence we will consider is the de~Rham sequence associated with a linear connection $\nabla$ on a filtered vector bundle $E$ over a filtered manifold $M$.
This can be characterized as the unique extension of $\nabla$,
\begin{equation}\label{E:nablaE}
\cdots\to\Omega^{k-1}(M;E)\xrightarrow{d_{k-1}^\nabla}\Omega^k(M;E)\xrightarrow{d_k^\nabla}\Omega^{k+1}(M;E)\to\cdots,
\end{equation}
such that the Leibniz rule $d^\nabla(\alpha\wedge\psi)=d\alpha\wedge\psi+(-1)^k\alpha\wedge d^\nabla\psi$ holds for all $\alpha\in\Omega^k(M)$ and $\psi\in\Omega^*(M;E)$, cf.~\cite[Section~7.14]{GHV2}.
Here we use the notation $\Omega^k(M;E)=\Gamma^\infty(\Lambda^kT^*M\otimes E)$ for the space of $E$-valued differential forms.

We assume that $\nabla$ is filtration-preserving, that is to say, we assume $\nabla_X\psi\in\Gamma^\infty(E^{p+q}M)$ for all $X\in\Gamma^\infty(T^pM)$ and $\psi\in\Gamma^\infty(E^q)$.
Then all operators in the sequence \eqref{E:nablaE} are of graded Heisenberg order at most zero with respect to the induced filtration on the vector bundles $\Lambda^kT^*M\otimes E$.
We will, furthermore, assume that the curvature \cite[Section~7.15]{GHV2} of $\nabla$ is contained in filtration degree one, that is, we assume $F^\nabla_x(X_1,X_2)\psi\in E_x^{p_1+p_2+p+1}$ for all $X_i\in T_x^{p_i}M$ and $\psi\in E^p_x$.
Linear connections of this kind exist on every filtered vector bundle.
If $E$ is trivially filtered, then all linear connections on $E$ satisfy the two assumptions.
In general, using a splitting of the filtration to identify $E$ with its associated graded, $\gr(E)=\bigoplus_pE^p/E^{p+1}$, each linear connection preserving the grading on $\gr(E)$ will satisfy the two assumptions.
Moreover, all tractor bundles associated with regular parabolic geometries come equipped with a natural linear connection satisfying these assumptions, see Section~\ref{SS:BGG} below.

We have the following generalization of a result of Rumin's for the de~Rham complex on C-C manifolds, cf.~\cite[Theorem~5.2]{R01} and \cite[Theorem~3]{R99}.

\begin{thm}\label{thmB}
Let $E$ be a filtered vector bundle over a filtered manifold $M$ and suppose $\nabla$ is a filtration-preserving linear connection on $E$ such that its curvature is contained in filtration degree one.
Then the de~Rham sequence in \eqref{E:nablaE} is a graded Rockland sequence.
\end{thm}

Essentially, Theorem~\ref{thmB} follows from the fact that the Lie algebra cohomology $H^*(\goe;\mathcal H_\infty)$ vanishes for every finite dimensional nilpotent Lie algebra $\goe$ and its representation on the space of smooth vectors $\mathcal H_\infty$ associated with any non-trivial irreducible unitary representation of the corresponding simply connected nilpotent Lie group on a Hilbert space $\mathcal H$.
This will be established in Section~\ref{SS:linearconnections}, see Proposition~\ref{P:hypo} below.

To construct new sequences, we follow {\v{C}}ap, Slov{\'a}k, and Sou{\v{c}}ek, see \cite{CSS01}, and consider a Kostant type codifferential.
By this we mean a sequence of filtration-preserving vector bundle homomorphisms,
$$
\cdots\leftarrow\Omega^{k-1}(M;E)\xleftarrow{\delta_k}\Omega^k(M;E)\xleftarrow{\delta_{k+1}}\Omega^{k+1}(M;E)\leftarrow\cdots,
$$
satisfying  $\delta_k\delta_{k+1}=0$ and two more conditions formulated in Definition~\ref{D:Kdelta} below.
Assuming $\delta_k$ to have locally constant rank, we obtain smooth vector bundles $\ker(\delta_k)$, $\img(\delta_{k+1})$, and $\mathcal H_k:=\ker(\delta_k)/\img(\delta_{k+1})$, which are filtered in a natural way.
We let $\bar\pi_k\colon\ker(\delta_k)\to\mathcal H_k$ denote the natural vector bundle projection.

Using the BGG machinery \cite{CSS01,CD01,CS12,CS15} we will see that there exist operators analogous to the splitting operators in parabolic geometry, see \cite[Theorem~2.4]{CS12}.
More precisely, there exists a unique differential operator $\bar L_k\colon\Gamma^\infty(\mathcal H_k)\to\Omega^k(M;E)$ such that $\delta_k\bar L_k=0$, $\bar\pi_k\bar L_k=\id$, and $\delta_{k+1}d^\nabla_k\bar L_k=0$.
These operators $\bar L_k$ are of graded Heisenberg order zero and permit defining a sequence of differential operators of graded Heisenberg order zero,
\begin{equation}\label{E:Dbar}
\cdots\to\Gamma^\infty(\mathcal H_{k-1})\xrightarrow{\bar D_{k-1}}\Gamma^\infty(\mathcal H_k)\xrightarrow{\bar D_k}\Gamma^\infty(\mathcal H_{k+1})\to\cdots,
\end{equation}
by setting $\bar D_k:=\bar\pi_{k+1}d^\nabla_k\bar L_k$.

In Section~\ref{SS:subdeRham} we will establish the following result, see Corollary~\ref{C:D}.

\begin{thm}\label{thmC}
The operators in \eqref{E:Dbar} form a graded Rockland sequence.
\end{thm}

A codifferential $\delta$ of maximal rank exists, provided the dimension of the Lie algebra cohomology $H^*(\mathfrak t_xM;\gr(E_x))$ is locally constant in $x$.
Note that the curvature assumption on $\nabla$ implies that $\gr(E_x)$ becomes a graded representation of the graded nilpotent Lie algebra $\mathfrak t_xM$, see Lemma~\ref{L:curv}.
In this case the codifferential $\delta_k$ can be constructed using splittings of the filtrations on the bundles $\Lambda^kT^*M\otimes E\cong\gr(\Lambda^kT^*M\otimes E)=\Lambda^k\mathfrak t^*M\otimes\gr(E)$ and the adjoint of the fiber-wise Chevalley--Eilenberg differential, $\partial_{k-1}\colon\Lambda^{k-1}\mathfrak t^*M\otimes\gr(E)\to\Lambda^k\mathfrak t^*M\otimes\gr(E)$, see Remark~\ref{R:Kdeltaexi} for details.
For this codifferential there exists a (non-canonical) isomorphism of smooth vector bundles $\mathcal H_k\cong H^k(\mathfrak tM;\gr(E))=\ker(\partial_k)/\img(\partial_{k-1})$ where the latter denotes the vector bundle with fibers $H^k(\mathfrak t_xM;\gr(E_x))$.

For tractor bundles associated with regular parabolic geometries, however, there exists a natural choice for $\delta$ which is called \emph{Kostant codifferential} and often denoted by $\partial^*$.
In this case the construction above reduces to the construction of the curved BGG sequences, and the operators $\bar L_k$ coincide with the well known splitting operators, see \cite[Theorem~2.4]{CS12} for instance.
As an immediate corollary of Theorem~\ref{thmC} we thus obtain, cf.\ Corollary~\ref{C:BGG}:

\begin{thm}\label{thmD}
All (curved, torsion free) BGG sequences associated with a regular parabolic geometry are graded Rockland sequences.
\end{thm}

To prove Theorem~\ref{thmC}, we shall construct another sequence that, at the principal symbol level, can be combined with the sequence \eqref{E:Dbar} to obtain the de~Rham sequence of Theorem~\ref{thmB}, up to conjugation. 
The Rockland condition for both components then follows from Theorem~\ref{thmB}.
This construction is closely related to the standard BGG machinery and the approach by Rumin \cite{R99,R01}.

More precisely, we consider $\Box_k:=d^\nabla_{k-1}\delta_k+\delta_{k+1}d^\nabla_k$, a differential operator of graded Heisenberg order at most zero on $\Omega^k(M;E)$.
The associated graded vector bundle endomorphism $\tilde\Box_k:=\gr(\Box_k)$ on $\gr(\Lambda^kT^*M\otimes E)$ is analogous to Kostant's box operator.
Using the fiber-wise projection onto the generalized zero eigenspace of $\tilde\Box_k$, we obtain a vector bundle projector $\tilde P_k$ on $\gr(\Lambda^kT^*M\otimes E)$, providing a decomposition of smooth vector bundles
$$
\gr(\Lambda^kT^*M\otimes E)=\img(\tilde P_k)\oplus\ker(\tilde P_k)
$$
such that $\tilde\Box_k$ is nilpotent on $\img(\tilde P_k)$ and invertible on $\ker(\tilde P_k)$.
We will construct two sequences of differential operators of graded Heisenberg order at most zero,
\begin{equation}\label{E:DDD}
\cdots\to\Gamma^\infty(\img(\tilde P_{k-1}))\xrightarrow{D_{k-1}}\Gamma^\infty(\img(\tilde P_k))\xrightarrow{D_k}\Gamma^\infty(\img(\tilde P_{k+1}))\to\cdots
\end{equation}
and
\begin{equation}\label{E:BBB}
\cdots\to\Gamma^\infty(\ker(\tilde P_{k-1}))\xrightarrow{B_{k-1}}\Gamma^\infty(\ker(\tilde P_k))\xrightarrow{B_k}\Gamma^\infty(\ker(\tilde P_{k+1}))\to\cdots,
\end{equation}
as well as invertible differential operators 
$$
L_k\colon\Gamma^\infty(\gr(\Lambda^kT^*M\otimes E))\to\Omega^k(M;E),
$$
such that the graded Heisenberg principal symbols are related by 
$$
\tilde\sigma^0_x(L^{-1}_{k+1}d_k^\nabla L_k)=\tilde\sigma^0_x(D_k)\oplus\tilde\sigma^0_x(B_k)
$$
at each point $x\in M$.
Theorem~\ref{thmB} readily implies that \eqref{E:DDD} and \eqref{E:BBB} are both graded Rockland sequences.
Moreover, we will construct an invertible differential operator of graded Heisenberg order zero, $V_k\colon\Gamma^\infty(\img(\tilde P_k))\to\Gamma^\infty(\mathcal H_k)$, such that $V_{k+1}^{-1}\bar D_kV_k=D_k$, whence Theorem~\ref{thmC}.

The construction of the operators announced in the preceding paragraph is based on the observation that there exists a unique filtration-preserving differential operator
$$
P_k\colon\Omega^k(M;E)\to\Omega^k(M;E)
$$ 
characterized by $P_k\Box_k=\Box_kP_k$, $P_k^2=P_k$ and $\gr(P_k)=\tilde P_k$.
This operator has graded Heisenberg order zero.
Using splittings of the filtrations, $S_k\colon\gr(\Lambda^kT^*M\otimes E)\to\Lambda^kT^*M\otimes E$, we define differential operators of graded Heisenberg order zero, 
$$
L_k:=P_kS_k\tilde P_k+(\id-P_k)S_k(\id-\tilde P_k).
$$
Since $\gr(L_k)=\id$, this differential operator is invertible and its inverse $L^{-1}_k$ is a differential operator of graded Heisenberg order zero too.
Moreover, it conjugates the differential projectors into vector bundle projectors, $L^{-1}_kP_kL_k=\tilde P_k$.
We will verify that the operators $D_k:=\tilde P_{k+1}L_{k+1}^{-1}d^\nabla_kL_k|_{\Gamma^\infty(\img(\tilde P_k))}$, $B_k:=(\id-\tilde P_{k+1})L_{k+1}^{-1}d^\nabla_kL_k|_{\Gamma^\infty(\ker(\tilde P_k))}$, and $V_k:=\bar\pi_kL_k|_{\Gamma^\infty(\img(\tilde P_k))}$ with inverse $V_k^{-1}=L_k^{-1}\bar L_k$ have all the desired properties.
The operator $P_k$ is related to the splitting operator $\bar L_k$ considered above by $\bar L_k\bar\pi_k=P_k|_{\ker(\delta_k)}$.
On regular parabolic geometries $P_k$ coincides with the composition of (5.1) and (5.2) in \cite{CD01}.

Let us now suppose that the linear connection $\nabla$ is flat.
In this case the sequence \eqref{E:nablaE} is known as de~Rham complex, $d^\nabla_kd^\nabla_{k-1}=0$, and computes the cohomology of $M$ with coefficients in the locally constant sheave provided by the flat connection on $E$.
In this case the sequence of operators $L_{k+1}^{-1}d^\nabla_kL_k$ decouples into a Rumin complex and an acyclic subcomplex, cf.\ \cite[Theorem~2.6]{R01} or \cite[Theorem~1]{R99}.
More precisely, we have
$$
L_{k+1}^{-1}d^\nabla_kL_k=D_k\oplus B_k,
$$
and, in particular, $D_kD_{k-1}=0$ as well as $B_kB_{k-1}=0$.
In this situation the sequences in \eqref{E:Dbar} and \eqref{E:DDD} will be called Rumin complexes, for they essentially coincide with complexes on contact \cite{R90,R94,R00} and more general Carnot--Carath\'eodory \cite{R99, R01} manifolds which have been introduced by Rumin.
We will show that the sequence $B_k$ is conjugate to an acyclic tensorial complex.
More precisely, we will see that there exist invertible differential operators of graded Heisenberg order at most zero, $G_k$ acting on $\Gamma^\infty(\ker(\tilde P_k))$, such that 
$$
G_{k+1}^{-1}B_kG_k=\partial_k|_{\Gamma^\infty(\ker(\tilde P_k))}
$$ 
where the right hand side denotes the restriction of the Chevalley--Eilenberg differential on $\gr(\Lambda^kT^*M\otimes E)=\Lambda^k\mathfrak t^kM\otimes\gr(E)$ to the invariant acyclic subbundle $\ker(\tilde P_k)$, see Theorem~\ref{T:D}.
Summarizing, we obtain:

\begin{thm}\label{thmE}
If the linear connection $\nabla$ is a flat, then there exist invertible differential operators of graded Heisenberg order zero, $W_k\colon\Gamma^\infty(\mathcal H_k\oplus\ker(\tilde P_k))\to\Omega^k(M;E)$, such that
$$
W^{-1}_{k+1}d^\nabla_kW_k=\bar D_k\oplus(\partial_k|_{\Gamma^\infty(\ker(\tilde P_k))}).
$$
\end{thm}

On a contact manifold, the Rumin complex \eqref{E:Dbar} is Rockland in the ungraded sense.
Hypoellipticity of this complex has been established by Rumin in \cite[Section~3]{R94} using results of Helffer and Nourrigat \cite{HN85}.
For generic rank two distributions in dimension five, the Rumin complex is Rockland in the ungraded sense too, see Example~\ref{Ex:BGG235}.
In general, the Rumin complex will only be Rockland in the graded sense, and the graded analysis in Section~\ref{S:grRockland} may be used to study them.
For instance, the Rumin complex associated with an Engel structure will only be Rockland in the graded sense, see Example~\ref{Ex:Engel}.
On C-C manifolds, Rumin has used the concept of C-C ellipticity, see \cite[Definition~5.1]{R01} or \cite[Section~2]{R99}, to show that the Rumin complexes are hypoelliptic, see \cite[Theorem~5.2]{R01} or \cite[Theorem~3]{R99}.

\subsection{Motivation and outlook}\label{SS:motivation}

The work in this paper provides a framework to study filtered manifolds by exploring the analogies to the elliptic case.
Classically, the relation between geometry and topology has been successfully studied by analyzing elliptic operators that arise naturally.
We hope that the hypoellipticity of the operators considered in this paper will allow relating the geometry of filtered manifolds to global topological properties.
We will now mention some directions which have been motivating our investigations.

By hypoellipticity, Rockland operators on closed filtered manifolds are Fredholm and there is a clear candidate for the index formula.
To be more specific, suppose $E$ and $F$ are two vector bundles over a closed filtered manifold, and consider $A\in\Psi^s(E,F)$ such that $A$ and $A^t$ both satisfy the Rockland condition.
In this situation, the analysis mentioned above implies that $A$ induces a Fredholm operator between appropriate Heisenberg Sobolev spaces, see Corollary~\ref{C:Fredholm}.
We expect that the index of this operator can be computed by an index formula similar to van~Erp's in the contact case, see \cite{E10a} and \cite{C94}.
More precisely, the Rockland condition should guarantee that the Heisenberg principal symbol of $A$ represents a $K$-theory class on the non-commutative cotangent bundle, $[\sigma^s(A)]\in K_0(C^*(\mathcal TM))$, and we expect the index formula $\ind(A)=\tind(\psi([\sigma^s(A)]))$ where $\psi\colon K_0(C^*(\mathcal TM))\to K^0(T^*M)$ denotes the abstract Connes--Thom isomorphism \cite{C81} and $\tind\colon K^0(T^*M)\to\Z$ is the topological index map of Atiyah and Singer \cite{AS68}.
More generally, the Heisenberg principal symbol sequence of every Rockland complex should, in a natural way, represent an element in $K_0(C^*(\mathcal TM))$ which is mapped to the Euler characteristics of the Rockland complex via $\tind\circ\psi\colon K_0(C^*(\mathcal TM))\to\Z$.
We expect explicit index formulas for various parabolic geometries, similar to van~Erp's formula on contact manifolds, see \cite{E10b}.

It seems promising to try to extend the Weitzenb\"ock formula for the Rumin complex on contact manifolds \cite{R90,R94} to other filtered manifolds $M$ and combine them with the Hodge decomposition, cf.\ Corollaries~\ref{C:Hodge} and \ref{C:grHodge}, to obtain analogues of Bochner's vanishing result.
Assuming non-negative curvature, a Weitzenb\"ock formula should imply that every harmonic section of $\mathcal H_k$ is parallel.
Over closed connected manifolds, the Hodge decomposition would thus yield a bound on the $k$-th Betti number, $b_k(M)\leq\rank(\mathcal H_k)$.
If, moreover, the curvature is strictly positive at one point, one would expect $b_k(M)=0$.

Let us specialize the above remarks to a particular $5$-dimensional Cartan geometry and formulate a precise conjecture.
To this end, consider a $5$-manifold $M$ equipped with a generic rank two distribution \cite{C10,BH93,S08,DH16}.
More precisely, suppose $T^{-1}M\subseteq TM$ is a distribution of rank two with growth vector $(2,3,5)$, that is, Lie brackets of sections of $T^{-1}M$ span a rank three subbundle $T^{-2}M$ of $TM$ and triple brackets of sections of $T^{-1}M$ span all of $TM$.
Such a filtered manifold can equivalently be described as a regular normal parabolic geometry of type $(G,P)$ where $G$ is the split real form of the exceptional Lie group $G_2$ and $P$ is a particular parabolic subgroup.
Cartan \cite{C10} constructed a curvature tensor $\kappa\in\Gamma^\infty(S^4(T^{-1}M)^*)$ which is a complete obstruction to local flatness.
More precisely, $\kappa$ vanishes if and only if the filtration is locally diffeomorphic to the flat model $G/P$.
Regarding the curvature $\kappa_x$ as a fourth order polynomial on $T^{-1}_xM$, we call $\kappa_x$ non-negative and write $\kappa_x\geq0$, if $\kappa_x(X,Y,X,Y)\geq0$ for all $X,Y\in T_x^{-1}M$.
Since the corresponding Rumin complex for the trivial flat line bundle has $\rank(\mathcal H_1)=2$, see Example~\ref{Ex:BGG235} below, we conjecture the following to hold true:
If $M$ is closed, connected, and $\kappa\geq0$, then the first Betti number is bounded by $b_1(M)\leq2$. 
If, moreover, $\kappa_x>0$ in at least one point $x$, then $b_1(M)=0$.

Another application we have in mind concerns the extension of Ponge's \cite{P08} spectral analysis on Heisenberg manifolds to more general filtered manifolds.
In \cite{DH17b}, building on the analytic results presented here, the heat kernel expansion has been established for formally selfadjoint, non-negative Rockland differential operators on general closed filtered manifolds.
As an immediate application of this result, one obtains the detailed structure of complex powers of these operators, and a Weyl's law for the growth of their eigenvalues.
Other applications include a McKean--Singer index formula for Rockland differential operators, and the construction of a non-commutative residue on the algebra of Heisenberg pseudodifferential operators, see \cite{DH17b}.
Moreover, this analysis permits to generalize \cite{H21} the Rumin--Seshadri analytic torsion \cite{RS12} to closed filtered manifolds which give rise to ungraded Rumin complexes.
Due to the rich structure available on regular parabolic geometries it appears feasible to work out explicit anomaly formulas for the Rumin--Seshadri analytic torsion, expressing to what extent this analytic torsion depends on the $L^2$ inner product used to define the formal adjoints, see~\eqref{E:IRS}.
We hope that the decomposition of the de~Rham complex in Theorem~\ref{thmE} will prove helpful in establishing a comparison result relating the Rumin--Seshadri analytic torsion of the Rumin complex with the Ray--Singer torsion \cite{RS71} of the full de~Rham complex.

The existence of a regular parabolic geometry of a particular type on a given manifold can often be described equivalently in terms of a differential relation.
Formally, these can be solved in terms of homotopy theory.
The subtle question is to what extent Gromov's h-principle \cite{G86} holds true for regular parabolic geometries.
We anticipate that the proposed generalization of the Rumin--Seshadri analytic torsion has the potential to detect a possible failure of the h-principle.
In particular, this might lead to topological obstructions to the existence of regular parabolic geometries on closed manifolds \cite{DH16}, and it might provide a sufficiently strong tool to show that formally homotopic regular parabolic geometries need not be homotopic in general, see \cite{P16}.

\subsection{Organization of the remaining part of the paper}\label{SS:outline}

In Section~\ref{S:hesequences} we begin by considering differential operators on filtered manifolds.
We present a result which states the existence of a parametrix for Rockland differential operators, and discuss several immediate consequences.
The Rockland type theorem will be proved in Section~\ref{S:PDO} which contains a more general form of this result.
There we recall the Heisenberg pseudodifferential calculus and complement it with a general Rockland type theorem asserting that Rockland pseudodifferential operators admit a parametrix in this calculus.
As an application, we introduce a Heisenberg Sobolev scale, and formulate maximal hypoelliptic estimates.
We extend the BGG machinery to filtered manifolds in Section~\ref{S:grhesequences} and show that the BGG sequences are Rockland in a graded sense.
Section~\ref{S:grRockland} then provides hypoellipticity for graded Rockland operators by reducing it to the results in Section~\ref{S:PDO}.
This completes the proof of the claim that BGG sequences are hypoelliptic.

\section{Hypoelliptic sequences of differential operators}\label{S:hesequences}

In the present section we are considering differential operators on filtered manifolds.
We present a result which states the existence of a parametrix for Rockland differential operators.
Since our main goal is to provide analysis for general BGG type operators, we will note here that the result of this section is inadequate for this purpose in spite of the fact that these sequences are made of differential operators.
This is because the vector bundles underlying the BGG sequences are graded and hence they only satisfy a graded version of the pointwise Rockland condition, as shown in Section~\ref{S:grhesequences}.
The purpose of this section is to set up notation and provide background for readers not familiar with pseudodifferential operators.

\subsection{Differential operators on filtered manifolds}\label{SS:DO}

Let $M$ be a filtered manifold, cf.~\eqref{E:filtM}, and consider the quotient bundle $\mathfrak t^pM:=T^pM/T^{p+1}M$ with fibers $\mathfrak t_x^pM=T^p_xM/T^{p+1}_xM$.
Recall that the Lie bracket of vector fields induces a tensorial (Levi) bracket $\mathfrak t^pM\otimes\mathfrak t^qM\to\mathfrak t^{p+q}M$ which turns the associated graded $\mathfrak tM:=\bigoplus_p\mathfrak t^pM$ into a bundle of graded nilpotent Lie algebras called the \emph{bundle of osculating algebras}.
Each fiber $\mathfrak t_xM=\bigoplus_p\mathfrak t^p_xM$ is a graded nilpotent Lie algebra which will be referred to as the \emph{osculating algebra at $x$}.
The Lie algebra structure depends smoothly on $x$ but is not assumed to be locally trivial, that is, different fibers may be non-isomorphic as Lie algebras.
In the literature $\mathfrak t_xM$ is also known as the symbol algebra of $M$ at $x$, see \cite{M93,M02,N10}.

A filtration on $M$ induces a (Heisenberg) filtration on differential operators.
If $E$ and $F$ are smooth vector bundles over $M$, we let $\DO(E,F)$ denote the class of all differential operators mapping section of $E$ to sections of $F$.
A differential operator in $\DO(E,F)$ is said to be of \emph{Heisenberg order at most $k$} if, locally, it can be written as a finite linear combination of operators of the form $\Phi\nabla_{X_m}\cdots\nabla_{X_1}$ where $\Phi\in\Gamma^\infty(\hom(E,F))$, $\nabla$ is a linear connection on $E$, and $X_i\in\Gamma^\infty(T^{p_i}M)$ are vector fields such that $-k\leq p_m+\cdots+p_1$.
Denoting the space of these differential operators by $\DO^k(E,F)$, we obtain a filtration on $\DO(E,F)$,
$$
\Gamma^\infty(\hom(E,F))=\DO^0(E,F)\subseteq\DO^1(E,F)\subseteq\DO^2(E,F)\subseteq\cdots,
$$
which is compatible with composition and transposition.
More explicitly, if $G$ is another vector bundle over $M$, $A\in\DO^k(E,F)$ and $B\in\DO^l(F,G)$, then $BA\in\DO^{l+k}(E,G)$ and $A^t\in\DO^k(F',E')$.
Here $A^t$ denotes the transpose (differential) operator characterized by
\begin{equation}\label{E:At}
\langle\phi,A\psi\rangle=\langle A^t\phi,\psi\rangle,
\end{equation} 
for all $\psi\in\Gamma^\infty_c(E)$ and $\phi\in\Gamma^\infty_c(F')$, with respect to the canonical pairings $\Gamma^\infty_c(E')\times\Gamma^\infty(E)\to\C$ and $\Gamma^\infty_c(F')\times\Gamma^\infty(F)\to\C$.
Here we are using the notation $E':=E^*\otimes|\Lambda|_M$ where $E^*$ denotes the dual bundle and $|\Lambda|_M$ denotes the bundle of $1$-densities on $M$.
Note that $(A^t)^t=A$, up to the canonical isomorphism of vector bundles $E''=E$.

\begin{remark}[The spaces $\Gamma^r(E)$]\label{R:GammarE}
For $r\in\mathbb N_0$ we let $\Gamma^r(E)$ denotes the Heisenberg analogue of the space of $r$ times continuously differentiable sections of $E$.
More precisely, $\Gamma^r(E)$ denotes the space of all $\psi\in\Gamma^{-\infty}(E)$ such that $A\psi\in\Gamma(F)$ for all differential operators $A\in\DO^r(E,F)$ of Heisenberg order at most $r$ and all vector bundles $F$.
Here $\Gamma(F)\subseteq\Gamma^{-\infty}(F)$ denotes the space of continuous sections equipped with the topology of uniform convergence on compact subsets.
We equip $\Gamma^r(E)$ with the coarsest topology such that the maps $A\colon\Gamma^r(E)\to\Gamma(F)$ are continuous for all $A\in\DO^r(E,F)$.
If $r-k\geq0$, then each $A\in\DO^k(E,F)$ induces a continuous operator, $A\colon\Gamma^r(E)\to\Gamma^{r-k}(F)$.
Note that we have continuous inclusions $\cdots\subseteq\Gamma^2(E)\subseteq\Gamma^1(E)\subseteq\Gamma^0(E)$ and topological isomorphisms $\Gamma^0(E)=\Gamma(E)$ as well as $\bigcap_r\Gamma^r(E)=\Gamma^\infty(E)$.
We will denote the compactly supported analogue by $\Gamma^r_c(E)$.
\end{remark}

\begin{remark}[Universal differential operators]\label{R:jk}
Consider a vector bundle $E$ over $M$ and let $J^kE\to E$ denote the bundle of Heisenberg $k$-jets of sections of $E$.
This is a smooth vector bundle whose fiber over $x\in M$ coincides with the vector space of Heisenberg $k$-jets at $x$ of sections of $E$.
Recall that two sections $\psi_1,\psi_2\in\Gamma^\infty(E)$ are said to represent the same Heisenberg $k$-jet at $x$ if $A(\psi_2-\psi_1)(x)=0$ for all differential operators $A\in\DO^k(E,F)$.
Assigning to a section of $E$ its Heisenberg $k$-jet, we obtain a differential operator $j^k\colon\Gamma^\infty(E)\to\Gamma^\infty(J^kE)$.
In fact, $j^k\in\DO^k(E,J^kE)$, and this differential operator is universal in the following sense: For every $A\in\DO^k(E,F)$ there exists a unique smooth vector bundle homomorphism $\alpha\colon J^kE\to F$ such that $A=\alpha\circ j^k$.
We refer to \cite[Section~3.1]{M02}, \cite[Section~1.2.6]{N10} or \cite{N09} for details.
\end{remark}

A differential operator $A\in\DO^k(E,F)$ has a \emph{Heisenberg principal cosymbol} at each $x\in M$,
$$
\sigma_x^k(A)\in\mathcal U_{-k}(\mathfrak t_xM)\otimes\hom(E_x,F_x),
$$
where $\mathcal U_{-k}(\mathfrak t_xM)$ denotes the degree $-k$ part of the universal enveloping algebra of the graded Lie algebra $\mathfrak t_xM=\bigoplus_p\mathfrak t_x^pM$.
More explicitly, $\mathcal U_{-k}(\mathfrak t_xM)$ can be described as the linear subspace of $\mathcal U(\mathfrak t_xM)$ spanned by all elements of the form $X_m\cdots X_1$ where $X_i\in\mathfrak t_x^{p_i}M$ and $-k=p_m+\cdots+p_1$.
The Heisenberg principal cosymbol provides a short exact sequence
$$
0\to\DO^{k-1}(E,F)\to\DO^k(E,F)\xrightarrow{\sigma^k}\Gamma^\infty\bigl(\mathcal U_{-k}(\mathfrak tM)\otimes\hom(E,F)\bigr)\to0
$$
where $\mathcal U_{-k}(\mathfrak tM):=\bigsqcup_{x\in M}\mathcal U_{-k}(\mathfrak t_xM)$ is a smooth vector bundle of finite rank according to the Poincar\'e--Birkhoff--Witt theorem.
Details may be found in \cite[Section~1.2.5]{N10}.

If $A\in\DO^k(E,F)$ and $B\in\DO^l(F,G)$ where $G$ is another vector bundle over $M$, then
\begin{equation}\label{E:sABAt}
\sigma_x^{l+k}(BA)=\sigma_x^l(B)\sigma^k_x(A)
\qquad\text{and}\qquad
\sigma_x^k(A^t)=\sigma_x^k(A)^t.
\end{equation}
To explain the second equation in \eqref{E:sABAt}, we extend $-\id\colon\mathfrak t_xM\to\mathfrak t_xM$ to an anti-automorphism of $\mathcal U(\mathfrak t_xM)$, $\mathbf X\mapsto\mathbf X^t$.
Hence, $(\mathbf X^t)^t=\mathbf X$ and $(\mathbf X\mathbf Y)^t=\mathbf Y^t\mathbf X^t$ for all $\mathbf X,\mathbf Y\in\mathcal U(\mathfrak t_xM)$.
This \emph{antipode} preserves the grading components $\mathcal U_{-k}(\mathfrak t_xM)$.
We extend this further to a transposition $\mathcal U(\mathfrak t_xM)\otimes\hom(E_x,F_x)\to\mathcal U(\mathfrak t_xM)\otimes\hom(F_x',E_x')$ characterized by $(\mathbf X\otimes\Phi)^t:=\mathbf X^t\otimes\Phi^t\otimes\id_{|\Lambda|_{M,x}}$ for all $\Phi\in\hom(E_x,F_x)$ and $\mathbf X\in\mathcal U(\mathfrak t_xM)$ where $\Phi^t\in\hom(F_x^*,E_x^*)$ denotes the linear map dual to $\Phi$.
This is the transposition used in $\sigma_x^k(A)^t$, see~\eqref{E:sABAt}.

If $\nabla$ is a linear connection on $E$ and $X\in\Gamma^\infty(T^{-k}M)$ then $\nabla_X\in\DO^k(E)$ and
\begin{equation}\label{E:snablaX}
\sigma^k(\nabla_X)=[X]\otimes\id_E\in\Gamma^\infty\bigl(\mathcal U_{-k}(\mathfrak tM)\otimes\eend(E)\bigr)
\end{equation}
where $[X]$ denotes the section of $\mathfrak t^{-k}M=T^{-k}M/T^{-k+1}M$ represented by $X$.
This property, together with the multiplicativity in \eqref{E:sABAt} and the requirement $\sigma^0(A)=A$ for all $A\in\DO^0(E,F)=\Gamma^\infty(\hom(E,F))=\Gamma^\infty(\mathcal U_0(\mathfrak tM)\otimes\hom(E,F))$, characterizes the Heisenberg principal symbol uniquely.

\begin{remark}[Formal adjoints]\label{R:sA*}
Suppose $A\in\DO^k(E,F)$ and let $A^*$ denote the formal adjoint with respect to $L^2$ inner products associated with a smooth volume density $dx$ on $M$ and smooth fiber-wise Hermitian inner products $h_E$ and $h_F$ on the vector bundles $E$ and $F$, respectively.
Hence, $A^*$ is characterized by 
\begin{equation}\label{E:A*}
\llangle A^*\phi,\psi\rrangle_{L^2(E)}=\llangle\phi,A\psi\rrangle_{L^2(F)}
\end{equation}
for all $\phi\in\Gamma^\infty_c(F)$ and $\psi\in\Gamma^\infty_c(E)$, where
\begin{equation}\label{E:llrr}
\llangle\psi_1,\psi_2\rrangle_{L^2(E)}=\int_Mh_E(\psi_1(x),\psi_2(x))dx
=\langle(h_E\otimes dx)\psi_1,\psi_2\rangle
\end{equation}
for $\psi_1,\psi_2\in\Gamma_c^\infty(E)$ and similarly for $F$.
Here we consider $h_E\otimes dx\colon \bar E\to E'$ as a vector bundle isomorphism.
In terms of the transpose we have
\begin{equation}\label{E:A*At}
A^*=(h_E\otimes dx)^{-1}\circ A^t\circ(h_F\otimes dx).
\end{equation}
In particular, $A^*\in\DO^k(F,E)$ and
\begin{equation}\label{E:sA*}
\sigma_x^k(A^*)=\sigma^k_x(A)^*.
\end{equation}
The involution $\mathcal U(\mathfrak t_xM)\otimes\hom(E_x,F_x)\to\mathcal U(\mathfrak t_xM)\otimes\hom(F_x,E_x)$ used on the right hand side can be characterized by $(\mathbf X\otimes\Phi)^*=\mathbf X^t\otimes\Phi^*$ for all $\Phi\in\hom(E_x,F_x)$ and $\mathbf X\in\mathcal U(\mathfrak t_xM)$ where $\Phi^*\in\hom(F_x,E_x)$ denotes the adjoint of $\Phi$ with respect to the inner products $h_{E,x}$ and $h_{F,x}$.
Equation \eqref{E:sA*} follows from \eqref{E:A*At} and \eqref{E:sABAt}.
\end{remark}

A graded Lie algebra has a natural group of \emph{dilation automorphisms.}
Thus, for $\lambda>0$ we let $\dot\delta_\lambda\in\Aut(\mathfrak tM)$ denote the bundle automorphism given by multiplication with $\lambda^{-p}$ on the grading component $\mathfrak t^pM$.
For each $x\in M$, this restricts to an automorphism $\dot\delta_{\lambda,x}\in\Aut(\mathfrak t_xM)$ of the osculating algebra such that $\lim_{\lambda\to0}\dot\delta_{\lambda,x}=0$.
Clearly, $\dot\delta_{\lambda_1\lambda_2}=\dot\delta_{\lambda_1}\dot\delta_{\lambda_2}$ for all $\lambda_1,\lambda_2>0$.
Extending $\dot\delta_{\lambda,x}$ to an automorphism of $\mathcal U(\mathfrak t_xM)$, we can characterize the grading by
\begin{equation}\label{E:Uk}
\mathcal U_{-k}(\mathfrak t_xM)=\bigl\{\mathbf X\in\mathcal U(\mathfrak t_xM):\textrm{$\dot\delta_{\lambda,x}(\mathbf X)=\lambda^k\mathbf X$ for all $\lambda>0$}\bigr\}.
\end{equation}

We let $\mathcal TM\to M$ denote the \emph{bundle of osculating groups.}
For each $x\in M$, the fiber $\mathcal T_xM$ is a simply connected nilpotent Lie group, called the \emph{osculating group at $x$}, with Lie algebra $\mathfrak t_xM$.
The fiber-wise exponential map, $\exp\colon\mathfrak tM\to\mathcal TM$, provides an isomorphism of smooth fiber bundles.
The Lie algebra automorphisms $\dot\delta_{\lambda,x}$ integrate to group automorphisms $\delta_{\lambda,x}\in\Aut(\mathcal T_xM)$ which assemble to a smooth bundle automorphism $\delta_\lambda\in\Aut(\mathcal TM)$ such that $\exp\circ\dot\delta_\lambda=\delta_\lambda\circ\exp$.
Clearly, $\delta_{\lambda_1\lambda_2}=\delta_{\lambda_1}\delta_{\lambda_2}$, for all $\lambda_1,\lambda_2>0$.

Since the universal enveloping algebra of $\mathfrak t_xM$ can be identified with the algebra of left invariant differential operators on $\mathcal T_xM$, and in view of \eqref{E:Uk}, the Heisenberg principal symbol of $A\in\DO^k(E,F)$ can equivalently be regarded as a \emph{left invariant differential operator,}
\begin{equation}\label{E:skA}
\sigma_x^k(A)\colon C^\infty(\mathcal T_xM,E_x)\to C^\infty(\mathcal T_xM,F_x),
\end{equation}
which is \emph{homogeneous} of degree $k$, that is,
\begin{equation}\label{E:skAlg}
\sigma^k_x(A)\circ l_g^*=l_g^*\circ\sigma^k_x(A)
\qquad\text{and}\qquad
\sigma^k_x(A)\circ\delta_{\lambda,x}^*=\lambda^k\cdot\delta_{\lambda,x}^*\circ\sigma^k_x(A)
\end{equation}
for all $g\in\mathcal T_xM$ and $\lambda>0$.
Here $l_g^*$ denotes pull back along the left translation, $l_g\colon\mathcal T_xM\to\mathcal T_xM$, $l_g(h):=gh$, and $\delta_{\lambda,x}^*$ denotes pull back along the dilation discussed above.

\begin{remark}
If the filtration on $M$ is trivial, that is to say, if $T^{-1}M=TM$, then the filtration on differential operators is the usual one.
In this case $\mathcal T_xM=T_xM$ is an Abelian Lie group and the principal cosymbol $\sigma^k_x(A)$ of a differential operator $A$ is a translation invariant (constant coefficient) differential operator on $T_xM$.
\end{remark}

\subsection{Parametrices}\label{SS:paraDO}

As we have seen above, the Heisenberg principal symbols of a differential operator can be described by homogeneous left invariant operators on the osculating Lie groups.
This is the primary reason why the osculating groups and their representation theory, and particularly the Rockland condition, become relevant to the analysis of these operators.
We shall now briefly recall some facts from representation theory necessary to formulate the Rockland condition for differential operators, see Definition~\ref{D:rockland} below, and state the corresponding Rockland type theorem, see Theorem~\ref{T:para}.

Let $G$ be a Lie group with Lie algebra $\goe$. 
Suppose $\pi\colon G\to U(\mathcal H)$ is a \emph{unitary representation} of $G$ on a Hilbert space $\mathcal H$.
These representations will always be assumed to be \emph{strongly continuous}, that is, the map $G\to\mathcal H$, $g\mapsto\pi(g)v$, is assumed to be continuous for every vector $v\in\mathcal H$.
For unitary representations, this is actually equivalent to \emph{weak continuity} which only asserts that the function $G\to\C$, $g\mapsto\llangle\pi(g)v,w\rrangle_{\mathcal H}$, is continuous for any two vectors $v,w\in\mathcal H$, see \cite[Theorem~1 in Appendix~V]{K04}.
Rarely will the representations we shall encounter be continuous with respect to the norm topology on $U(\mathcal H)$.

Recall that $v\in\mathcal H$ is called \emph{smooth vector} if the map $G\to\mathcal H$, $g\mapsto\pi(g)v$, is (strongly) smooth.
According to \cite[Theorem~3 in Appendix~V]{K04} this is equivalent to the weak assumption: the function $G\to\C$, $g\mapsto\llangle\pi(g)v,w\rrangle_{\mathcal H}$, is smooth for all vectors $w\in\mathcal H$.
We will denote the subspace of smooth vectors by $\mathcal H_\infty$.
This is a dense subspace in $\mathcal H$ which is invariant under the operators $\pi(g)$ for all $g\in G$, see \cite[Theorem~4(1) in Appendix~V]{K04}.
For each $X\in\goe$ we may define, see \cite[Theorem~4(2) in Appendix~V]{K04}, 
$$
\pi(X)\colon\mathcal H_\infty\to\mathcal H_\infty,\qquad\pi(X)v:=\tfrac\partial{\partial t}\big|_{t=0}\pi(\exp(tX))v,
$$
where $v\in\mathcal H_\infty$.
By unitarity, $\llangle\pi(X)v,w\rrangle_{\mathcal H}=\llangle v,\pi(-X)w\rrangle_{\mathcal H}$ for all $v,w\in\mathcal H_\infty$.
Hence, $\pi(X)$ has a densely defined adjoint, $\pi(X)^*=\pi(-X)$, and, in particular, $\pi(X)$ is closeable, see \cite[Theorem~4(2) in Appendix~V]{K04}.
Clearly, $\pi([X,Y])=\pi(X)\pi(Y)-\pi(Y)\pi(X)$ for any two $X,Y\in\goe$.
Extending the definition of $\pi$ to the universal enveloping algebra of $\goe$, we obtain $\pi(\mathbf X)\colon\mathcal H_\infty\to\mathcal H_\infty$ for $\mathbf X\in\mathcal U(\goe)$ such that 
\begin{equation}\label{E:piXY}
\pi(\mathbf X)\pi(\mathbf Y)=\pi(\mathbf X\mathbf Y)
\end{equation}
for all $\mathbf X,\mathbf Y\in\mathcal U(\goe)$.
We let $\mathbf X\mapsto\mathbf X^t$ denote the antipode of $\mathcal U(\goe)$ obtained by extending $-\id\colon\goe\to\goe$ to the universal enveloping algebra.
Hence, $(\mathbf X^t)^t=\mathbf X$ and $(\mathbf X\mathbf Y)^t=\mathbf Y^t\mathbf X^t$ for all $\mathbf X,\mathbf Y\in\mathcal U(\goe)$.
For each $\mathbf X\in\mathcal U(\goe)$ we thus have 
\begin{equation}\label{E:piX*}
\pi(\mathbf X)^*=\pi(\mathbf X^t)
\end{equation} 
as operators on $\mathcal H_\infty$.

If $E_0$ and $F_0$ are two finite dimensional vector spaces and $a\in\mathcal U(\goe)\otimes\hom(E_0,F_0)$ we let 
$$
\pi(a)\colon\mathcal H_\infty\otimes E_0\to\mathcal H_\infty\otimes F_0
$$ 
denote the linear operator obtained by linearly extending the definition $\pi(\mathbf X\otimes\Phi):=\pi(\mathbf X)\otimes\Phi$ for all $\Phi\in\hom(E_0,F_0)$ and $\mathbf X\in\mathcal U(\goe)$.
Equivalently, using bases of $E_0$ and $F_0$ to identify $a$ with a matrix with entries in $\mathcal U(\goe)$, the operator $\pi(a)$ corresponds to a matrix of the same size whose entries are operators on $\mathcal H_\infty$ obtained by applying $\pi$ to the corresponding entry of $a$.
The multiplicativity in \eqref{E:piXY} immediately implies
\begin{equation}\label{E:pBA}
\pi(ba)=\pi(b)\pi(a)
\end{equation}
for all $a\in\mathcal U(\goe)\otimes\hom(E_0,F_0)$ and $b\in\mathcal U(\goe)\otimes\hom(F_0,G_0)$ where $G_0$ is another finite dimensional vector space.
If, moreover, $E_0$ and $F_0$ are equipped with Hermitian inner products, then \eqref{E:piX*} leads to 
\begin{equation}\label{E:pA*}
\pi(a)^*=\pi(a^*)
\end{equation}
as operators $\mathcal H_\infty\otimes F_0\to\mathcal H_\infty\otimes E_0$.
Here the adjoint on the left hand side of \eqref{E:pA*} is with respect to the inner products on $\mathcal H_\infty\otimes E_0$ and $\mathcal H_\infty\otimes F_0$ induced by inner products on $E_0$ and $F_0$ and the restriction of the inner product on $\mathcal H$.
On the right hand side of \eqref{E:pA*}, $a^*\in\mathcal U(\goe)\otimes\hom(F_0,E_0)$ is defined by linear extension of $(\mathbf X\otimes\Phi)^*:=\mathbf X^t\otimes\Phi^*$ for all $\mathbf X\in\mathcal U(\goe)$ and $\Phi\in\hom(E_0,F_0)$ where $\Phi^*\in\hom(F_0,E_0)$ denotes the adjoint of $\Phi$.
Equivalently, using orthogonal bases of $E_0$ and $F_0$ to identify $a$ with a matrix with entries in $\mathcal U(\goe)$, $a^*$ corresponds to the matrix obtained by taking the transpose conjugate of $a$ and applying the antipode $\mathbf X\mapsto\mathbf X^t$ to each entry.

\begin{definition}[Rockland condition]\label{D:rockland}
Let $E$ and $F$ be vector bundles over a filtered manifold $M$.
A differential operator $A\in\DO^k(E,F)$ of Heisenberg order at most $k$ is said to satisfy the \emph{Rockland condition} if $\pi(\sigma_x^k(A))\colon\mathcal H_\infty\otimes E_x\to\mathcal H_\infty\otimes F_x$ is injective for every point $x\in M$ and every non-trivial irreducible unitary representation $\pi\colon\mathcal T_xM\to U(\mathcal H)$ of the osculating group $\mathcal T_xM$ on a Hilbert space $\mathcal H$.
Here $\mathcal H_\infty$ denotes the subspace of smooth vectors in $\mathcal H$.
\end{definition}

We let $\mathcal O(E,F)$ denote the space of operators $\Gamma^\infty_c(E)\to\Gamma^{-\infty}(F)$ corresponding to Schwartz kernels with wave front set contained in the conormal bundle of the diagonal.
These are precisely the operators whose kernel is smooth away from the diagonal and which map $\Gamma^\infty_c(E)$ continuously into $\Gamma^\infty(F)$.
If $A\in\mathcal O(E,F)$, then $A^t\in\mathcal O(F',E')$, cf.~\eqref{E:At}.
The transpose permits extending $A$ continuously to distributional sections, $A\colon\Gamma^{-\infty}_c(E)\to\Gamma^{-\infty}(F)$, such that $\langle A^t\phi,\psi\rangle=\langle\phi,A\psi\rangle$ for all $\psi\in\Gamma^{-\infty}_c(E)$ and $\phi\in\mathcal D(F):=\Gamma^\infty_c(F')$, and this extension is pseudolocal, i.e.\ $\singsupp(A\psi)\subseteq\singsupp(\psi)$
for all $\psi\in\Gamma^{-\infty}_c(E)$.
Recall that an operator with Schwartz kernel $k$ is called \emph{properly supported} if the two projections $M\times M\to M$ both restrict to proper maps on the support of $k$.
If $B\in\mathcal O(F,G)$ and at least one of $A$ or $B$ is properly supported, then $BA\in\mathcal O(E,G)$ and $(BA)^t=A^tB^t$.
We refer to \cite{T67,D96} for details.

We have the following vector valued analogue of a result due to Melin \cite[Theorem~7.2]{M82}:

\begin{theorem}[Left parametrix]\label{T:para}
Let $E$ and $F$ be vector bundles over a filtered manifold $M$ and suppose $A\in\DO^k(E,F)$ is a differential operator of Heisenberg order at most $k$ which satisfies the Rockland condition, see Definition~\ref{D:rockland}.
Then there exists a properly supported left parametrix $B\in\mathcal O_\prop(F,E)$ such that $BA-\id$ is a smoothing operator.
\end{theorem}

Melin \cite{M82} considers the scalar case and shows that the parametrix may be chosen to be a pseudodifferential operator of Heisenberg order $-k$ in the calculus constructed in said paper.
In Section~\ref{S:PDO} we will formulate and prove a generalization of this result for pseudodifferential operators of any order, see Theorem~\ref{T:Rockland} below.
This will allow us to refine the subsequent hypoellipticity statements, see Section~\ref{SS:sobolev}, and to extend them to the graded setup required for the analysis of (generalized) BGG sequences, see Section~\ref{S:grRockland}.

\begin{remark}
Let us point out that several special cases of this result are well know.
To begin with, for trivially filtered manifolds, i.e., $TM=T^{-1}M$, this reduces to the classical, elliptic case.
In this situation all irreducible unitary representations of the (abelian) osculating group are one dimensional, and the scalar Rockland condition at $x\in M$ becomes the familiar condition that the principal symbol of the operator is invertible at every $0\neq\xi\in T_x^*M$.

Another well studied class are the contact and (more generally) Heisenberg manifolds.
For Heisenberg manifolds, a pseudodifferential calculus has been developed independently by Beals--Greiner \cite{BG88} and Taylor \cite{T84}, see also \cite{P08}.
Special cases of Theorem~\ref{T:para} for Heisenberg manifolds can be found in \cite[Theorem~8.4]{BG88} or \cite[Theorem~5.4.1]{P08}.
These investigations can be traced back to the work of Kohn \cite{K65}, Boutet de Monvel \cite{B74}, and Folland--Stein \cite{FS74} on CR manifolds.
For more historical comments we refer to the introduction of \cite{BG88}.

If the filtration on $M$ is locally diffeomorphic to that on a graded nilpotent Lie group, then the scalar version of Theorem~\ref{T:para} can be found in \cite[Theorem~2.5(d)]{CGGP92}. 
This suffices to study the flat models in parabolic geometry given by the homogeneous spaces $G/P$, as well as topologically stable \cite{P16} structures like contact and Engel manifolds.
\end{remark}

\begin{corollary}[Hypoellipticity]\label{C:hypo}
Let $E$ and $F$ be vector bundles over a filtered manifold $M$ and suppose $A\in\DO^k(E,F)$ is a differential operator of Heisenberg order at most $k$ which satisfies the Rockland condition.
Then $A$ is hypoelliptic, that is, if $\psi$ is a compactly supported distributional section of $E$ and $A\psi$ is smooth on an open subset $U$ of $M$, then $\psi$ was smooth on $U$.
If, moreover, $M$ is closed, then $\ker(A)$ is a finite dimensional subspace of\/ $\Gamma^\infty(E)$.
\end{corollary}

\begin{proof}
We recall a standard argument.
Let $B\in\mathcal O_\prop(F,E)$ be a left parametrix as in The\-o\-rem~\ref{T:para}.
Hence, $BA-\id$ is a smoothing operator and $BA\psi-\psi$ is a smooth section of $E$.
Moreover, $BA\psi$ is a smooth on $U$, for $B$ is pseudolocal.
Consequently, $\psi$ is smooth on $U$.

Assume $M$ to be closed.
By hypoellipticity, $\ker(A)\subseteq\Gamma^\infty(E)\subseteq L^2(E)$.
Since $BA-\id$ is a smoothing operator, the identical map on $\ker(A)$ coincides with the restriction of a smoothing operator.
The latter induces a compact operator on $L^2(E)$ according to the theorem of Arzel\`a--Ascoli.
Hence, every bounded subset of $\ker(A)$ is precompact in $L^2(E)$.
Consequently, $\ker(A)$ has to be finite dimensional.
\end{proof}

\begin{corollary}[Hodge decomposition]\label{C:smooth-Hodge-decomposition}
Let $E$ be a vector bundle over a closed filtered manifold $M$.
Suppose $A\in\DO^k(E)$ satisfies the Rockland condition and is formally selfadjoint, $A^*=A$, with respect to an $L^2$ inner product of the form \eqref{E:llrr}.
Moreover, let $Q$ denote the orthogonal projection onto the (finite dimensional) subspace $\ker(A)\subseteq\Gamma^\infty(E)$.
Then $A+Q$ is invertible with inverse $(A+Q)^{-1}\in\mathcal O(E)$.
Consequently, we have topological isomorphisms and Hodge type decompositions:
\begin{align*}
A+Q\colon\Gamma^\infty(E)&\xrightarrow\cong\Gamma^\infty(E),&\Gamma^\infty(E)&=\ker(A)\oplus A(\Gamma^\infty(E)),\\
A+Q\colon\Gamma^{-\infty}(E)&\xrightarrow\cong\Gamma^{-\infty}(E),&\Gamma^{-\infty}(E)&=\ker(A)\oplus A(\Gamma^{-\infty}(E)).
\end{align*}
\end{corollary}

\begin{proof}
We recall the classical argument \cite{G84} which will be referred to in the proof of Corollary~\ref{C:PsiinvA}.
According to Theorem~\ref{T:para} there exists $B\in\mathcal O(E)$ such that $BA-\id$ is a smoothing operator.
Since $A$ is formally selfadjoint, $AB^*-\id$ is a smoothing operator too.
We conclude that $B$ and $B^*$ differ by a smoothing operator.
Hence $P^*=P:=\tfrac12(B+B^*)\in\mathcal O(E)$ is a formally selfadjoint parametrix such that $PA-\id$ and $AP-\id$ are both smoothing operators.

Note that $Q=Q^*$, the orthogonal projection onto $\ker(A)$, is a smoothing operator.
In particular, $A+Q$ is hypoelliptic.
Moreover, $\ker(A+Q)=0$ in view of $A^*=A$.
Since $AP-\id$ is a smoothing operator, arguing as in the proof of Corollary~\ref{C:hypo} shows that $P$ is hypoelliptic and $\ker(P)$ is a finite dimensional subspace of $\Gamma^\infty(E)$.
Adding the orthogonal projection onto $\ker(P)$ to $P=P^*$, we may furthermore assume $\ker(P)=0$.

Consider $G:=(A+Q)P\in\mathcal O(E)$.
Since $G-\id$ is a smoothing operator, it induces a compact operator on every classical Sobolev space $H^s_\cl(E)$.
Hence, $G$ induces a Fredholm operator with vanishing index on $H^s_\cl(E)$ for all real numbers $s$.
By construction, $G$ is injective, whence invertible with bounded inverse on $H^s_\cl(E)$.
Using the classical Sobolev embedding theorem, we conclude that $G$ is invertible on $\Gamma^\infty(E)$ with continuous inverse, $G^{-1}\colon\Gamma^\infty(E)\to\Gamma^\infty(E)$.
Using $G^*=P(A+Q)$, the same argument shows that $G^*$ is invertible on $\Gamma^\infty(E)$ with continuous inverse, $(G^*)^{-1}\colon\Gamma^\infty(E)\to\Gamma^\infty(E)$.
Since $(G^*)^{-1}$ is the formal adjoint of $G^{-1}$, we conclude that $G^{-1}$ extends continuously to distributional sections, $G^{-1}\colon\Gamma^{-\infty}(E)\to\Gamma^{-\infty}(E)$.
Thus, according to the Schwartz kernel theorem, $G^{-1}$ is given by a (distributional) kernel we will denote by $G^{-1}$ too.
The obvious relation $G^{-1}-\id=-(G-\id)G^{-1}$ implies that $G^{-1}-\id$ is a smoothing operator and, consequently, $G^{-1}\in\mathcal O(E)$.
We conclude $(A+Q)^{-1}=PG^{-1}\in\mathcal O(E)$.
The remaining assertions follow at once.
\end{proof}

\subsection{Rockland sequences}\label{SS:hesDO}

Let us now generalize the concept of elliptic sequences of differential operators to filtered manifolds.
In subsequent sections we will generalize further to the graded setup, see Definition~\ref{D:graded_hypoelliptic_seq}, and pseudodifferential operators, see Definition~\ref{D:gRs}.

\begin{definition}[Rockland sequences of differential operators]\label{def.Hypo-seq}
Let $E_i$ be smooth vector bundles over a filtered manifold $M$.
A sequence of differential operators,
\begin{equation}\label{E:hes}
\cdots\to\Gamma^\infty(E_{i-1})\xrightarrow{A_{i-1}}\Gamma^\infty(E_i)\xrightarrow{A_i}\Gamma^\infty(E_{i+1})\to\cdots,
\end{equation}
with $A_i\in\DO^{k_i}(E_i,E_{i+1})$, is called \emph{Rockland sequence} if
for every $x\in M$ and all non-trivial irreducible unitary representation of the osculating group, $\pi\colon\mathcal T_xM\to U(\mathcal H)$, the sequence
\begin{equation}\label{E:Hisigma}
\cdots\to
\mathcal H_\infty\otimes E_{i-1,x}\xrightarrow{\pi(\sigma^{k_{i-1}}_x(A_{i-1}))}
\mathcal H_\infty\otimes E_{i,x}\xrightarrow{\pi(\sigma^{k_i}_x(A_i))}
\mathcal H_\infty\otimes E_{i+1,x}\to\cdots
\end{equation}
is weakly exact, i.e., the image of the left arrow is contained and dense in the kernel of the right arrow.
Here $\mathcal H_\infty$ denotes the subspace of smooth vectors in the Hilbert space $\mathcal H$.
\end{definition}

\begin{remark}\label{R:weakexactness}
For every Rockland sequence of differential operators, the sequence~\eqref{E:Hisigma} is actually exact in the strict (algebraic) sense.
This follows from Lemma~\ref{L:rockseq} and Remark~\ref{R:AtRock}.
\end{remark}

\begin{remark}\label{R:transposedseq}
If a sequence of differential operators as in \eqref{E:hes} is a Rockland sequence, then so are the transposed sequence,
$$
\cdots\leftarrow\Gamma^\infty(E_{i-1}')\xleftarrow{A_{i-1}^t}\Gamma^\infty(E_i')\xleftarrow{A_i^t}\Gamma^\infty(E_{i+1}')\leftarrow\cdots,
$$
and the sequence of formal adjoints with respect to $L^2$ inner products of the form \eqref{E:llrr}.
\end{remark}

\begin{remark}
Consider the case $E_i=0$ for all $i\neq1,2$.
In other words, consider a sequence with a single differential operator of Heisenberg order at most $k_1$,
$$
0\to\Gamma^\infty(E_1)\xrightarrow{A_1}\Gamma^\infty(E_2)\to0.
$$
Such a sequence is hypoelliptic in the sense of Definition~\ref{def.Hypo-seq} iff, for all points $x\in M$,
$$
\pi(\sigma^{k_1}_x(A_1))\colon\mathcal H_\infty\otimes E_{1,x}\to\mathcal H_\infty\otimes E_{2,x}
$$ 
is injective with dense image for all non-trivial irreducible unitary representations $\pi\colon\mathcal T_xM\to U(\mathcal H)$.
Equivalently, $A_1$ and $A_1^t$ both satisfy the Rockland condition, see Definition~\ref{D:rockland}.
\end{remark}

Consider a Rockland sequence of differential operators as in \eqref{E:hes}.
To study this sequence we shall introduce certain additional structures and operators in analogy with the standard elliptic sequences.
Fix a smooth volume density $dx$ on $M$, let $h_i$ be smooth fiber-wise Hermitian inner products on $E_i$, and consider the associated $L^2$ inner products on $\Gamma^\infty_c(E_i)$,
\begin{equation}\label{E:llrrEi}
\llangle\psi_1,\psi_2\rrangle_{L^2(E_i)}
=\int_Mh_i(\psi_1(x),\psi_2(x))dx
=\langle(h_i\otimes dx)\psi_1,\psi_2\rangle
\end{equation}
where $\psi_1,\psi_2\in\Gamma_c^\infty(E_i)$.
Moreover, let $A_i^*\in\DO^{k_i}(E_{i+1},E_i)$ denote the corresponding formal adjoint, that is, $\llangle A_i^*\phi,\psi\rrangle_{L^2(E_i)}=\llangle\phi,A_i\psi\rrangle_{L^2(E_{i+1})}$ for $\psi\in\Gamma^\infty_c(E_i)$ and $\phi\in\Gamma_c^\infty(E_{i+1})$.

Assume $k_i\geq1$, choose positive integers $a_i$ such that 
\begin{equation}\label{E:RSai}
k_{i-1}a_{i-1}=k_ia_i=:\kappa,
\end{equation}
and consider the differential operator $\Delta_i\in\DO^{2\kappa}(E_i)$,
\begin{equation}\label{E:RSDelta}
\Delta_i:=(A_{i-1}A_{i-1}^*)^{a_{i-1}}+(A_i^*A_i)^{a_i}.
\end{equation}
We will refer to these operators as \emph{Rumin--Seshadri operators} since they generalize the fourth order Laplacians associated with the Rumin complex in \cite[Section~2.3]{RS12}.

\begin{lemma}\label{L:rockseq}
The Rumin--Seshadri operators $\Delta_i$ satisfy the Rockland condition.
\end{lemma}

\begin{proof}
Consider $x\in M$ and let $\pi\colon\mathcal T_xM\to U(\mathcal H)$ be a non-trivial irreducible unitary representation.
We equip $\mathcal H_\infty\otimes E_{i,x}$ with the Hermitian inner product provided by the restriction of the scalar product of the Hilbert space $\mathcal H$ and the inner product $h_{i,x}$ on $E_{i,x}$.
Using \eqref{E:sABAt}, \eqref{E:sA*}, \eqref{E:pBA} and \eqref{E:pA*}, we obtain:
$$
\pi(\sigma^\kappa_x(\Delta_i)))
=(B_{i-1}B_{i-1}^*)^{a_{i-1}}+(B_i^*B_i)^{a_i}
$$
where we abbreviate $B_i:=\pi(\sigma_x^{k_i}(A_i))\colon\mathcal H_\infty\otimes E_{i,x}\to\mathcal H_\infty\otimes E_{i+1,x}$ and we consider $B_i^*=\pi(\sigma_x^{k_i}(A_i^*))\colon\mathcal H_\infty\otimes E_{i+1,x}\to\mathcal H_\infty\otimes E_{i,x}$.
Due to positivity,
\begin{align*}
\ker(\pi(\sigma_x^{2\kappa}(\Delta_i)))
&=\ker\bigl((B_{i-1}B_{i-1}^*)^{a_{i-1}}\bigr)
\cap\ker\bigl((B_i^*B_i)^{a_i}\bigr),
\\
\ker\bigl((B_{i-1}B_{i-1}^*)^{a_{i-1}}\bigr)
&=\ker(B_{i-1}^*)\textrm{, and}
\\
\ker\bigl((B_i^*B_i)^{a_i}\bigr)
&=\ker(B_i).
\end{align*}
Since $A_i$ is a Rockland sequence, we also have $\ker(B_i)\subseteq\overline{\img(B_{i-1})}\perp\ker(B_{i-1}^*)$ and thus $\ker(B_i)\cap\ker(B_{i-1}^*)=0$.
Combining this with the preceding equalities, we obtain $\ker(\pi(\sigma_x^{2\kappa}(\Delta_i)))=0$, i.e., $\Delta_i$ satisfies the Rockland condition.
\end{proof}

Combining Corollary~\ref{C:hypo} and Lemma~\ref{L:rockseq} we see that each Rumin--Seshadri operator is hypoelliptic.
For Rockland sequences this immediately implies:

\begin{corollary}\label{C:RShypo}
The differential operator $(A_{i-1}^*,A_i)\colon\Gamma^\infty(E_i)\to\Gamma^\infty(E_{i-1}\oplus E_{i+1})$ is hypoelliptic, that is, if $\psi$ is a distributional section of $E_i$ such that $A_{i-1}^*\psi$ and $A_i\psi$ are both smooth on an open subset $U$ of $M$, then $\psi$ was smooth on $U$.
Moreover,
\begin{multline}\label{E:kerDAA}
\ker(\Delta_i|_{\Gamma^{-\infty}_c(E_i)})
=\ker(A_{i-1}^*|_{\Gamma^{-\infty}_c(E_i)})\cap\ker(A_i|_{\Gamma^{-\infty}_c(E_i)})
\\=\ker(A_{i-1}^*|_{\Gamma^\infty_c(E_i)})\cap\ker(A_i|_{\Gamma^\infty_c(E_i)}).
\end{multline}
\end{corollary}

Over closed manifolds Corollary~\ref{C:smooth-Hodge-decomposition} applies to $\Delta_i=\Delta_i^*$, hence $\ker(\Delta_i)$ is a finite dimensional subspace of $\Gamma^\infty(E_i)$, and we get Hodge type decompositions as in Corollary~\ref{C:smooth-Hodge-decomposition} for the Rumin--Seshadri operators.
For Rockland complexes over closed manifolds this implies:

\begin{corollary}\label{C:Hodge}
If $M$ is closed and $A_iA_{i-1}=0$, then we have Hodge type decompositions
\begin{align*}
\Gamma^\infty(E_i)&=A_{i-1}(\Gamma^\infty(E_{i-1}))\oplus\ker(\Delta_i)\oplus A_i^*(\Gamma^\infty(E_{i+1})),
\\
\Gamma^{-\infty}(E_i)&=A_{i-1}(\Gamma^{-\infty}(E_{i-1}))\oplus\ker(\Delta_i)\oplus A_i^*(\Gamma^{-\infty}(E_{i+1})),
\end{align*}
and
\begin{align*}
\ker(A_i|_{\Gamma^\infty(E_i)})&=A_{i-1}(\Gamma^\infty(E_{i-1}))\oplus\ker(\Delta_i),
\\
\ker(A_i|_{\Gamma^{-\infty}(E_i)})&=A_{i-1}(\Gamma^{-\infty}(E_{i-1}))\oplus\ker(\Delta_i).
\end{align*}
In particular, every cohomology class admits a unique harmonic representative:
$$
\frac{\ker(A_i|_{\Gamma^{-\infty}(E_i)})}{\img(A_{i-1}|_{\Gamma^{-\infty}(E_{i-1})})}
=\frac{\ker(A_i|_{\Gamma^\infty(E_i)})}{\img(A_{i-1}|_{\Gamma^\infty(E_{i-1})})}
=\ker(\Delta_i)
=\ker(A_{i-1}^*)\cap\ker(A_i).
$$
\end{corollary}

In the subsequent section these regularity statements will be refined by maximal hypoelliptic estimates.
We postpone these more elaborate results because their formulation requires a pseudodifferential calculus for filtered manifolds.

\section{A Rockland theorem for the Heisenberg calculus}\label{S:PDO}

The aim of this section is to prove Theorem~\ref{T:para} above which lies at the core of the hypoellipticity results discussed in Sections~\ref{SS:paraDO} and \ref{SS:hesDO}.
Combining the Heisenberg calculus \cite{M82,EY15v5} with harmonic analysis due to Christ, Geller, G{\l}owacki and Polin \cite{CGGP92}, and using arguments of Ponge \cite{P08}, we obtain a more general Rockland type theorem for pseudodifferential operators on filtered manifolds, see Theorem~\ref{T:Rockland} below.
In Section~\ref{SS:sobolev} we use this to introduce a Heisenberg Sobolev scale, see Proposition~\ref{P:Hs}, and improve upon Corollaries~\ref{C:RShypo} and \ref{C:Hodge} by establishing maximal hypoelliptic estimates, see Corollaries~\ref{C:regrockseq} and \ref{C:HsHodge-seq} below.

\subsection{The Heisenberg pseudodifferential calculus}\label{SS:calculus}

The pseudodifferential calculus on a filtered manifold can be approached via the Heisenberg tangent groupoid \cite{EY16,CP15,HH18,M21}.
A key feature of the Heisenberg tangent groupoid is an action of $\R_+$, the so-called zoom action.
The kernels in the Heisenberg calculus can be characterized as the Schwartz kernels that extend across the Heisenberg tangent groupoid in an essentially homogeneous fashion with respect to the zoom action.
As already mentioned, this is inspired by a characterization of classical pseudodifferential operators due to Debord and Skandalis \cite{DS14}.
Our first task is to describe the cosymbol space of the Heisenberg calculus and relate it to the harmonic analysis. 
We will outline the results needed from the Heisenberg calculus and refer the reader to \cite{EY15v5} for details. 
A more self-contained exposition of this part can be found in \cite[Section~3]{DH17}.

For two vector bundles $E$ and $F$ over a filtered manifold $M$, and any complex number $s$, we let $\Psi^s(E,F)$ denote the class of \emph{pseudodifferential operators} of Heisenberg order at most $s$, mapping sections of $E$ to sections of $F$, cf.~\cite[Definition~19]{EY15v5}.
There is a \emph{principal cosymbol map} 
$$
\sigma^s\colon\Psi^s(E,F)\to\Sigma^s(E,F),
$$
cf.~\cite[Definition~35]{EY15v5}, where $\Sigma^s(E,F)$ denotes the space \emph{principal cosymbols of order $s$.}

To describe the space of cosymbols, we let $\Omega_\pi$ denote the (trivializable) line bundle over $\mathcal TM$ obtained by applying the representation $|\det|^{-1}$ of the general linear group to the frame bundle of the vertical bundle $\ker(T\pi)$ of the submersion $\pi\colon\mathcal TM\to M$.
Hence, the restriction of $\Omega_\pi$ to the fiber $\mathcal T_xM$ identifies canonically with the bundle of volume densities on $\mathcal T_xM$, that is, $\Omega_\pi|_{\mathcal T_xM}=|\Lambda|_{\mathcal T_xM}$.

We let $\mathcal K(\mathcal TM;E,F)$ denote the space of all distributions $k\in\Gamma^{-\infty}(\hom(\pi^*E,\pi^*F)\otimes\Omega_\pi)$ whose wave front set is contained in the conormal of the identical section $M\subseteq\mathcal TM$. 
In particular, these $k$ are assumed to be smooth on $\mathcal TM\setminus M$. 
Equivalently, these can be characterized as families of regular distributional volume densities on the fibers $\mathcal T_xM$ with values in $\hom(E_x,F_x)$, smoothly parametrized by $x\in M$.
We also introduce the notation $\mathcal K^\infty(\mathcal TM;E,F):=\Gamma^\infty(\hom(\pi^*E,\pi^*F)\otimes\Omega_\pi)$ for the subspace of smooth cosymbols.
A cosymbol $k$ is called \emph{properly supported} if $\pi$ restricts to proper map on the support of $k$.

The regular representation of $k\in\mathcal K(\mathcal TM;E,F)$ at $x\in M$ provides a right invariant operator $C^\infty_c(\mathcal T_xM,E_x)\to C^\infty(\mathcal T_xM,F_x)$ on the group $\mathcal T_xM$ with matrix valued convolution kernel $k^x\in\Gamma^{-\infty}(|\Lambda|_{\mathcal T_xM})\otimes\hom(E_x,F_x)$.
The kernels in $\mathcal K(\mathcal TM;E,F)$ which are supported on the space of units $M\subseteq\mathcal TM$ correspond precisely to differential operators $\Gamma^\infty(\pi^*E)\to\Gamma^\infty(\pi^*F)$ which are vertical, i.e., commute with functions in the image of the homomorphism $\pi^*\colon C^\infty(M)\to C^\infty(\mathcal TM)$, and restrict to right invariant operators on each fiber $\mathcal T_xM$.

We denote the space of \emph{complete cosymbols} by
$$
\Sigma(E,F):=\frac{\mathcal K_\prop(\mathcal TM;E,F)}{\mathcal K_\prop^\infty(\mathcal TM;E,F)}
=\frac{\mathcal K(\mathcal TM;E,F)}{\mathcal K^\infty(\mathcal TM;E,F)},
$$
where the subscript indicates properly supported kernels.
The fiber-wise convolution product and inversion induce an associative multiplication and a compatible transposition,
$$
\Sigma(F,G)\times\Sigma(E,F)\xrightarrow*\Sigma(E,G),\qquad\Sigma(E,F)\xrightarrow{t}\Sigma(F',E').
$$
More explicitly, we have $(l*k)^t=k^t*l^t$ and $(k^t)^t=k$, for $k\in\Sigma(E,F)$ and $l\in\Sigma(F,G)$.

The scaling automorphism $\delta_\lambda$ acts on $\Sigma(E,F)$ in a way compatible with multiplication and transposition.
A cosymbol $k\in\mathcal K(\mathcal TM;E,F)$ is called \emph{essentially homogeneous of order $s$} if $(\delta_\lambda)_*k-k\in\mathcal K^\infty(\mathcal TM;E,F)$, for $\lambda>0$.
The space of \emph{principal cosymbols of order $s$} is
$$
\Sigma^s(E,F)
:=\left\{k\in\frac{\mathcal K(\mathcal TM;E,F)}{\mathcal K^\infty(\mathcal TM;E,F)}:\textrm{$(\delta_\lambda)_*k=\lambda^sk$ for all $\lambda>0$}\right\},
$$
cf.\ \cite[Definition~34]{EY15v5}.
Convolution and transposition are compatible with homogeneity, i.e.\
$$
\Sigma^{s_2}(F,G)\times\Sigma^{s_1}(E,F)\xrightarrow*\Sigma^{s_1+s_2}(E,G),\qquad\Sigma^s(E,F)\xrightarrow{t}\Sigma^s(F',E').
$$
For $k\in\N_0$ there is a canonical inclusion,
\begin{equation}\label{E:incUS}
\Gamma^\infty(\mathcal U_{-k}(\mathfrak tM)\otimes\hom(E,F))\subseteq\Sigma^k(E,F)
\end{equation}
provided by regarding both sides as right invariant vertical operators on $\mathcal T^\op M$.
\footnote{The opposite groupoid $\mathcal T^\op M$ mediates between two conflicting, yet common, conventions we are following: The Lie algebra of a Lie group is usually defined by restricting the Lie bracket to \emph{left invariant} vector fields, while the Lie algebroid of a smooth groupoid is defined using \emph{right invariant} vector fields.}

The following basic properties of the Heisenberg calculus have been established in \cite{EY15v5} for scalar valued operators and integral $s$. 
It is straight forward to extend this to the slightly more general setup we are considering here, see also \cite{M82,M83}.
Hence, we have:

\begin{proposition}\label{P:Psi}
Let $E$, $F$ and $G$ be vector bundles over a filtered manifold $M$ and let $s$ be any complex number.
Then the following hold true:
\begin{enumerate}[(a)]
\item\label{P:Psi:O}
We have $\Psi^s(E,F)\subseteq\mathcal O(E,F)$, the operators with conormal kernels.
\item\label{P:Psi:symbsequ}
We have $\Psi^{s-1}(E,F)\subseteq\Psi^s(E,F)$, and the following sequence is exact:
$$
0\to\Psi^{s-1}(E,F)\to\Psi^s(E,F)\xrightarrow{\sigma^s}\Sigma^s(E,F)\to0
$$
\item\label{P:Psi:smooth}
$\bigcap_{k\in\N}\Psi^{s-k}(E,F)=\mathcal O^{-\infty}(E,F)$, the smoothing operators.
\item\label{P:Psi:mult}
If $A\in\Psi^{s_1}(E,F)$, $B\in\Psi^{s_2}(F,G)$, and at least one of the two is properly supported, then $BA\in\Psi^{s_2+s_1}(E,G)$ and $\sigma^{s_2+s_1}(BA)=\sigma^{s_2}(B)\sigma^{s_1}(A)$.
\item\label{P:Psi:trans}
If $A\in\Psi^s(E,F)$, then $A^t\in\Psi^s(F',E')$ and $\sigma^s(A^t)=\sigma^s(A)^t$.
\item\label{P:Psi:DO}
$\DO^k(E,F)=\DO(E,F)\cap\Psi^k(E,F)$ for all $k\in\N_0$, and the principal symbol considered here extends the one for differential operators via the canonical inclusion \eqref{E:incUS}.
\item\label{P:Psi:parametrix}
Let $A\in\Psi^s(E,F)$ and assume that there exists $b\in\Sigma^{-s}(F,E)$ such that $b\,\sigma^s(A)=1$.
Then there exists a left parametrix $B\in\Psi^{-s}_\prop(F,E)$ such that $\sigma^{-s}(B)=b$ and $BA-\id$ is a smoothing operator. 
An analogous statement involving right parametrices holds true.
\end{enumerate}
\end{proposition}

Strictly speaking, the statement about the transposed in item \eqref{P:Psi:trans} above has not been addressed in \cite{EY15v5}. 
However, given the characterization of the caluclus in terms of the tangent groupoid, this is a trivial consequence, see \cite[Proposition~3.4(e)]{DH17}.
An equivalent statement for formal adjoints can be found in \cite[Theorem~5.10]{M82}.

\begin{remark}
For Heisenberg manifolds, a statement similar to Proposition~\ref{P:Psi}\itemref{P:Psi:parametrix} can be found in \cite[Proposition~3.3.1]{P08}.
\end{remark}

\begin{remark}[Formal adjoints]\label{R:Psi:adjoint}
If $A\in\Psi^s(E,F)$ and $A^*$ denotes the formal adjoint with respect to inner products of the form \eqref{E:llrr}, then $A^*\in\Psi^{\bar s}(F,E)$ and $\sigma^s(A^*)=\sigma^{\bar s}(A)^*$.
Indeed, in view of \eqref{E:A*At} this follows immediately from the assertions \itemref{P:Psi:mult}, \itemref{P:Psi:trans}, and \itemref{P:Psi:DO} of Proposition~\ref{P:Psi}.
Here the adjoint of the cosymbol, $\sigma^s(A)^*$, is understood as follows:
If $k\in\mathcal K(\mathcal TM;E,F)$, then $k^*\in\mathcal K(\mathcal TM;F,E)$ is defined by $k^*(g)=k(g^{-1})^*$ where $g
\in\mathcal T_xM$, and the right hand side denotes the adjoint of $k(g^{-1})$ with respect to the inner products $h_{E,x}$ and $h_{F,x}$ on $E_x$ and $F_x$, respectively.
Since this star preserves the subspace $\mathcal K^\infty(\mathcal TM;E,F)$ and commutes with $\delta_\lambda$, it induces an involution $\Sigma^{\bar s}(F,E)\xrightarrow{*}\Sigma^s(E,F)$, for each complex $s$.
For $s\in\N_0$ this extends the involution in Remark~\ref{R:sA*} via the inclusion \eqref{E:incUS}.
\end{remark}

\begin{remark}[Asymptotic expansion in exponential coordinates]\label{R:expcoor}
Let us use exponential coordinates as in \cite{EY15v5} to identify an open neighborhood $U$ of the zero section in $\mathcal TM$ with an open neighborhood $V$ of the diagonal in $M\times M$,
$$
\mathcal TM\supseteq U\xrightarrow\varphi V\subseteq M\times M.
$$
This diffeomorphism is obtained by restricting the composition
$$
\mathcal TM\xleftarrow{\exp}\mathfrak tM\xrightarrow{-S}TM\supseteq U'\xrightarrow{(p,\exp^\nabla)}M\times M.
$$
Here $\exp$ denotes the fiber-wise exponential map;
$S$ is a splitting of the filtration;
$\exp^\nabla$ denotes the exponential map associated with a linear connection on the tangent bundle which preserves the grading $TM=\bigoplus_pS(\mathfrak t^pM)$; 
$U'$ is an open neighborhood of the zero section in $TM$ on which $\exp^\nabla$ is defined and a diffeomorphism onto its image; and
$p\colon TM\to M$ denotes the canonical projection, cf.\ \cite[Section~3.2]{EY15v5}.
Let $\pi\colon\mathcal TM\to M$ denote the canonical projection and, after possibly shrinking $U$, fix an isomorphism of vector bundles $\tilde\varphi$ over $\varphi$, 
$$
\xymatrix{
\bigl(\pi^*\hom(E,F)\otimes\Omega_\pi\bigr)|_U\ar[d]\ar[r]^-{\tilde\varphi}_-\cong&(F\boxtimes E')|_V\ar[d]
\\
U\ar[r]^-\varphi_\cong&V,
}
$$ 
which restricts to the tautological identification over the zero section/diagonal. 
For every Schwartz kernel $k\in\Gamma^{-\infty}(F\boxtimes E')$ we obtain $\tilde\varphi^*(k|_V)\in\Gamma^{-\infty}((\pi^*\hom(E,F)\otimes\Omega_\pi)|_U)$.
If $k$ is the kernel of an operator $A\in\Psi^s(E,F)$, then we have an asymptotic expansion of the form
\begin{equation}\label{E:asexp}
\tilde\varphi^*(k|_V)\sim k_0+k_1+k_2+\cdots,
\end{equation}
where $k_j\in\Sigma^{s-j}(E,F)$ and $\sigma^s(A)=k_0$.
More precisely, for every $r\in\N$ there exists $N\in\N$ such that $\tilde\varphi^*(k|_V)-\sum_{j=0}^Nk_j$ is of class $C^r$ on $U$.
Conversely, if a Schwartz kernel $k$ is smooth away from the diagonal and admits an asymptotic expansion of the form \eqref{E:asexp}, then the corresponding operator is in $\Psi^s(E,F)$, see \cite[Theorem~59]{EY15v5}.
The calculus is asymptotically complete in the sense that any sequence $k_j$ can be realized by an operator in $\Psi^s(E,F)$.
\end{remark}

\begin{remark}\label{R:CGGP}
In the flat case, that is to say, if the filtration on $M$ is locally diffeomorphic to the left invariant filtration on a graded nilpotent Lie group, the calculus described above coincides with the calculus of Christ, Geller, G{\l}owacki, and Polin \cite{CGGP92}.
On Heisenberg manifolds it specializes to the classical Heisenberg calculus, see \cite{BG88,P08,T84}, which builds upon work of Boutet de Monvel \cite{B74}, Folland--Stein \cite{FS74} and Dynin \cite{D75,D76}, see also \cite{BGH75}, \cite{EM00}, \cite{G76}, \cite{H67}, and \cite{RS76}.
The equivalence with the Heisenberg calculus follows from \cite[Theorems~15.39 and 15.49]{BG88}, see also \cite[Proposition~3.1.15]{P08}, for the exponential coordinates used in Remark~\ref{R:expcoor} are clearly privileged coordinates in the sense of \cite[Definition~2.1.10]{P08}.
For trivially filtered manifolds, that is $TM=T^{-1}M$, we recover classical pseudodifferential operators, see \cite[Section~11]{EY15v5} and \cite{DS14}.
Since the coordinates used by Melin, obtained by integrating \cite[Proposition~2.9]{M82} connection maps \cite[Definition~2.7]{M82}, include the exponential coordinates used in Remark~\ref{R:expcoor}, his calculus should coincide with the one sketched above.
\end{remark}

Let us now link the principal cosymbols considered above with the principal cosymbols used by Christ, Geller, G{\l}owacki, and Polin in \cite{CGGP92}.
For every complex number $s$, put 
$$
\mathcal P^s(\mathcal TM;E,F):=\bigl\{k\in\mathcal K^\infty(\mathcal TM;E,F):\textrm{$(\delta_\lambda)_*k=\lambda^sk$ for all $\lambda>0$}\bigr\}.
$$ 
Clearly, $\mathcal P^s(\mathcal TM;E,F)=0$ if $-s-n\notin\mathbb N_0$ where 
\begin{equation}\label{E:hdim}
n:=-\sum_pp\cdot\rank(\mathfrak t^pM)=-\sum_pp\cdot\rank(T^pM/T^{p+1}M)
\end{equation} 
denotes the \emph{homogeneous dimension} of $M$.
For $-s-n\in\mathbb N_0$, by Taylor's theorem, this is the space of smooth kernels which restrict to (matrices of) polynomial volume densities of homogeneous degree $s$ on each fiber $\mathcal T_xM$, using the exponential map to canonically identify $\mathcal T_xM$ with the graded vector space $\mathfrak t_xM$.
Equivalently, these can be characterized as smooth families of polynomial volume densities of homogeneous degree $s$ on the fibers $\mathcal T_xM=\mathfrak t_xM$, smoothly parametrized by $x\in M$.
We have the following classical fact, cf.~\cite[Lemma~3.8]{DH17}:

\begin{lemma}\label{L:EYvsCGGP}
The identical map on $\mathcal K(\mathcal TM;E,F)$ induces a canonical identification
$$
\Sigma^s(E,F)=\left\{k\in\frac{\mathcal K(\mathcal TM;E,F)}{\mathcal P^s(\mathcal TM;E,F)}:\text{$(\delta_\lambda)_*k=\lambda^sk$ for all $\lambda>0$}\right\}.
$$
Moreover, with respect to a homogeneous norm $|-|$ on $\mathfrak tM$, and using the fiber-wise exponential map, $\exp\colon\mathfrak tM\to\mathcal TM$,
every kernel $k\in\mathcal K(\mathcal TM;E,F)$ which is essentially homogeneous of order $s$ can be written in the form
\begin{equation}\label{E:EYvsCGGP}
k=k_\infty+k_s+p_s\log|\exp^{-1}(-)|
\end{equation}
where $k_\infty\in\mathcal K^\infty(\mathcal TM;E,F)$, $k_s\in\mathcal K(\mathcal TM;E,F)$ homogeneous of order $s$, that is, $(\delta_\lambda)_*k_s=\lambda^sk_s$ for all $\lambda>0$, and $p_s\in\mathcal P^s(\mathcal TM;E,F)$.
If $-s-n\notin\mathbb N_0$, then $p_s=0$ and the decomposition in \eqref{E:EYvsCGGP} is unique.
If $-s-n\in\mathbb N_0$, then the decomposition in \eqref{E:EYvsCGGP} is unique up to adding a kernel in $\mathcal P^s(\mathcal TM;E,F)$ to $k_s$ and subtracting it from $k_\infty$ in turn.
\end{lemma}

\begin{proposition}\label{P:Psimap}
Let $E$ and $F$ be vector bundles over a filtered manifold $M$ of homogeneous dimension $n$, see \eqref{E:hdim}.
Consider $A\in\Psi^s(E,F)$ where $s$ is some complex number, and let $k\in\Gamma^{-\infty}(F\boxtimes E')$ denote the corresponding Schwartz kernel.
Then the following hold true:
\begin{enumerate}[(a)]
\item\label{P:Psimap:L2cont}
If $\Re s\leq0$, then $A$ induces a continuous operator $A\colon L^2_c(E)\to L^2_\loc(F)$.
\item\label{P:Psimap:compact}
If $\Re s<0$, then $A$ induces a compact operator $A\colon L^2_c(E)\to L^2_\loc(F)$.
\item\label{P:Psimap:kL2}
If $\Re s<-n/2$, then $A$ induces a continuous operator $A\colon L^2_c(E)\to\Gamma(F)$, cf.~Remark~\ref{R:GammarE}.
\item\label{P:Psimap:cont}
If $\Re s<-n$, then the kernel $k$ is continuous.
If, moreover, $M$ is closed and $E=F$, then $A\colon L^2(E)\to L^2(E)$ is trace class and
$$
\tr_{L^2(E)}(A)=\int_M\tr_E(\iota^*k)=\int_{x\in M}\tr_{E_x}(k(x,x))
$$
where $\iota^*k\in\Gamma^\infty(\eend(E)\otimes|\Lambda|_M)$ denotes the restriction of the kernel to the diagonal.
\end{enumerate}
\end{proposition}

\begin{proof}
To show \itemref{P:Psimap:L2cont}, suppose $\Re s\leq0$.
The symbol estimate in \cite[Corollary~45]{EY15v5} implies that the full symbol of $A$, i.e.\ the fiber-wise Fourier transform of the full cosymbol, with respect to exponential coordinates as in Remark~\ref{R:expcoor}, is in the standard class $\mathcal S^{s/m}_{1/m,0}$ where $m$ is such that $\mathfrak tM=\bigoplus_{j=1}^m\mathfrak t^{-j}M$, cf.\ \cite[Proposition~10.22]{BG88}.
This symbol class is not invariant under general coordinate change.
Nevertheless, boundedness on $L^2$ is a well known consequence, see for instance \cite[Theorem~6.1]{S01}.

To show \itemref{P:Psimap:compact} suppose $\Re s<0$.
Using Remark~\ref{R:expcoor} and Lemma~\ref{L:EYvsCGGP} we see that the kernel provides a family $k(x,-)\in L^1_\loc(F_x\otimes E')$, smoothly parametrized by $x\in M$; and the same holds true for the transposed kernel, see Proposition~\ref{P:Psi}\itemref{P:Psi:trans}.
In particular, given two compact subsets $K$ and $L$ of $M$, there exists a constant $C\geq0$ such that $\sup_{x\in L}\int_K|k(x,y)|dy\leq C$ and $\sup_{y\in K}\int_L|k(x,y)|dx\leq C$.
Hence, according to Schur's lemma, see \cite[Lemma~9.1]{S01} or \cite[Lemma~15.2]{FS74} for instance, the operator norm of the composition
\begin{equation}\label{E:AKL}
L^2_K(E)\subseteq L^2_c(E)\xrightarrow{A}L^2_\loc(F)\to L^2_L(F)
\end{equation}
is bounded by $C$, i.e., $\|A\psi\|_{L^2_L(F)}\leq C\|\psi\|_{L^2_K(E)}$ for all $\psi\in L^2_K(E)$.
Writing $k=\chi k+(1-\chi)k$, where $\chi\in C^\infty(M\times M,[0,1])$ and $\chi\equiv1$ in a neighborhood of the diagonal, we obtain a decomposition $A=A'+R$ where $R$ is a smoothing operator with kernel $(1-\chi)k$ and $A'\in\Psi^s(E,F)$ has kernel $\chi k$.
Given $\varepsilon>0$, we may choose $\chi$ such that $\sup_{x\in L}\int_K|\chi k(x,y)|dy\leq\varepsilon$ and $\sup_{y\in K}\int_L|\chi k(x,y)|dx\leq\varepsilon$ and, consequently, $\|A'\psi\|_{L^2_L(F)}\leq\varepsilon\|\psi\|_{L^2_K(E)}$.
We conclude that the composition in \eqref{E:AKL} can be approximated by smoothing operators.
Since the latter are compact, we conclude that the composition in \eqref{E:AKL} is  compact too.

To show \itemref{P:Psimap:kL2} we suppose $\Re s<-n/2$.
Using Remark~\ref{R:expcoor} and Lemma~\ref{L:EYvsCGGP} we see that the kernel provides a smooth family $k(x,-)\in L^2_\loc(F_x\otimes E')$, parametrized by $x\in M$.
In particular, given two compact subsets $K$ and $L$ of $M$, there exists a constant $C\geq0$ such that $\sup_{x\in L}\int_K|k(x,y)|^2dy\leq C^2$.
Using the Cauchy--Schwarz inequality, we obtain $\sup_{x\in L}|(A\psi)(x)|\leq C\|\psi\|_{L^2_K(E)}$ for all $\psi\in L^2_K(E)$.
Hence, $A$ maps $L^2_c(E)$ continuously into $\Gamma(F)$.

To show \itemref{P:Psimap:cont}, we suppose $\Re s<-n$.
Using Remark~\ref{R:expcoor} and Lemma~\ref{L:EYvsCGGP} we see that the kernel provides a family $k(x,-)\in\Gamma(F_x\otimes E')$, smoothly parametrized by $x\in M$.
Clearly, this implies that $k$ is continuous.
The remaining assertions are now obvious.
\end{proof}

For Melin's calculus, Proposition~\ref{P:Psimap}\itemref{P:Psimap:L2cont} can be found in \cite[Corollary~6.14]{M82}.

\subsection{Parametrices and Rockland condition}\label{SS:para}

In this section we establish a Rockland type result characterizing (left) invertible cosymbols and the existence of (left) parametrices in terms of irreducible unitary representations of the osculating groups, see Theorem~\ref{T:Rockland}.
As an application, we construct operators of arbitrary (complex) order which are invertible in the Heisenberg calculus, up to smoothing operators.
The latter play a crucial role in the construction of the Heisenberg Sobolev scale.

For every $x\in M$ we let
$$
\Sigma_x^s(E,F):=
\left\{k\in\frac{\mathcal K(\mathcal T_xM;E_x,F_x)}{\mathcal K^\infty(\mathcal T_xM;E_x,F_x)}:\textrm{$(\delta_\lambda)_*k=\lambda^sk$ for all $\lambda>0$}\right\}
$$
denote the space principal cosymbols of order $s$ at $x$.
Restriction provides a linear map $\ev_x\colon\Sigma^s(E,F)\to\Sigma^s_x(E,F)$ which is compatible with convolution and transposition.
Composing this with the principal symbol map, we obtain 
$$
\sigma^s_x\colon\Psi^s(E,F)\to\Sigma^s_x(E,F).
$$
We will refer to $\sigma^s_x(A)$ as the principal cosymbol of $A\in\Psi^s(E,F)$ at $x$.

Following \cite{CGGP92} and \cite[Section~3.3.2]{P08}, we will now formulate a (matrix) Rockland type condition for cosymbols in $\Sigma_x^s(E,F)$.
Let $\mathcal P(\mathcal T_xM)=\bigoplus_{j=0}^\infty\mathcal P^{-n-j}(\mathcal T_xM)$ denote the space of polynomial volume densities on $\mathcal T_xM$, where $n$ denotes the homogeneous dimension of $M$, see~\eqref{E:hdim}.
Recall that $\mathcal P(\mathcal T_xM)$ is invariant under translation and inversion.
For a finite dimensional vector space $E_0$, we let $\mathcal S_0(\mathcal T_xM;E_0)$ denote the subspace of all $f$ in the $E_0$-valued Schwartz space $\mathcal S(\mathcal T_xM;E_0)$ such that $\int_{\mathcal T_xM}pf=0$ for all $p\in\mathcal P(\mathcal T_xM)$.

In view of Lemma~\ref{L:EYvsCGGP}, every cosymbol $a\in\Sigma_x^s(E,F)$ can be represented in the form
\begin{equation}\label{E:akplog}
a=k+p\log|\exp^{-1}(-)|
\end{equation}
where $k\in\mathcal K(\mathcal T_xM;E_x,F_x)$ is homogeneous of order $s$, that is $(\delta_\lambda)_*k=\lambda^sk$ for all $\lambda>0$, and $p\in\mathcal P^s(\mathcal T_xM;E_x,F_x)$.
If $-s-n\notin\mathbb N_0$, then $p=0$ and $k$ is unique.
If $-s-n\in\mathbb N_0$, then $p$ is unique but $k$ comes with an ambiguity in $\mathcal P^s(\mathcal T_xM;E_x,F_x)$.
Hence, the restriction of the left regular representation,
\begin{equation}\label{E:aS0}
\mathcal S_0(\mathcal T_xM;E_x)\to\mathcal S_0(\mathcal T_xM;F_x),\qquad f\mapsto a*f,
\end{equation}
does not depend on the representative for $a\in\Sigma^s_x(E,F)$, provided it is of the form \eqref{E:akplog}.

Suppose $\pi\colon\mathcal T_xM\to U(\mathcal H)$ is a non-trivial irreducible unitary representation of the osculating group $\mathcal T_xM$ on a Hilbert space $\mathcal H$.
Let $\mathcal H_0$ denote the subspace of $\mathcal H$ spanned by elements of the form $\pi(fdg)v$ where $v\in\mathcal H$, $f\in\mathcal S_0(\mathcal T_xM)$, $dg$ denotes a (left) invariant volume density on $\mathcal T_xM$, and $\pi(fdg):=\int_{\mathcal T_xM}\pi(g)f(g)dg\in\mathcal K(\mathcal H)$.
Since $\mathcal H_0$ non-trivial and invariant under $\pi(g)$ for all $g\in\mathcal T_xM$, the subspace $\mathcal H_0$ is dense in $\mathcal H$.
Note that $\mathcal H_0\otimes E_x$ is spanned by vectors of the form $\pi(fdg)v$ where $v\in\mathcal H$, $f\in\mathcal S_0(\mathcal T_xM;E_x)$, and $\pi(fdg):=\int_{\mathcal T_xM}\pi(g)f(g)dg\in\mathcal K(\mathcal H,\mathcal H\otimes E_x)$.

Still assuming a representative of the form \eqref{E:akplog}, we define an unbounded operator $\pi(a)$ from $\mathcal H\otimes E_x$ to $\mathcal H\otimes F_x$ by
$$
\pi(a)\colon\mathcal H_0\otimes E_x\to\mathcal H_0\otimes F_x,\qquad\pi(a)\pi(fdg)v:=\pi(a*fdg)v,
$$ 
for all $f\in\mathcal S_0(\mathcal T_xM;E_x)$ and $v\in\mathcal H$, cf.~\eqref{E:aS0}.
As explained in \cite[Section~2]{CGGP92}, this definition of $\pi(a)$ is unambiguous.
Moreover, $\pi(a)$ is closeable, for $\pi(a^*)$ is a densely defined adjoint.
We denote its closure by $\bar\pi(a)$.
It is well know, see \cite{CGGP92}, \cite{G91}, or \cite[Proposition~3.3.6]{P08}, that the domain of definition of $\bar\pi(a)$ contains the space of smooth vectors, $\mathcal H_\infty\otimes E_x$, and this subspace is mapped into $\mathcal H_\infty\otimes F_x$ by $\bar\pi(a)$.
Furthermore, on $\mathcal H_\infty\otimes E_x$, we have
\begin{equation}\label{E:piab}
\bar\pi(ba)=\bar\pi(b)\bar\pi(a)
\end{equation}
whenever $a\in\Sigma^s_x(E,F)$ and $b\in\Sigma^{\tilde s}_x(F,G)$.
Moreover, with respect to inner products on $E_x$ and $F_x$ and the associated inner products on $\mathcal H\otimes E_x$ and $\mathcal H\otimes F_x$, we have 
\begin{equation}\label{E:pia*}
\bar\pi(a)^*=\bar\pi(a^*)
\end{equation}
on $\mathcal H_\infty\otimes F_x$.
If $a$ is the cosymbol of a differential operator then, on $\mathcal H_\infty\otimes E_x$, the operator $\bar\pi(a)$ coincides with $\pi(a)$ considered in Section~\ref{SS:paraDO}, see \cite[Remark~3.3.7]{P08}, and thus the following definition is consistent with Definition~\ref{D:rockland}.

\begin{definition}[Matrix Rockland condition]\label{D:Rockland}
A principal cosymbol $a\in\Sigma_x^s(E,F)$ at $x\in M$ is said to satisfy the \emph{Rockland condition} if, for every non-trivial irreducible unitary representation $\pi\colon\mathcal T_xM\to U(\mathcal H)$, the unbounded operator $\bar\pi(a)$ is injective on $\mathcal H_\infty\otimes E_x$. 
An operator $A\in\Psi^s(E,F)$ is said to satisfy the Rockland condition if its principal cosymbol, $\sigma^s_x(A)\in\Sigma^s_x(E,F)$, satisfies the Rockland condition at each point $x\in M$.
\end{definition}

We will need the following matrix version of a Rockland \cite{R78} type theorem due to Christ, Geller, G{\l}owacki, and Polin, see \cite[Theorem~6.2]{CGGP92} or \cite{G91} for the order zero case.
For differential operators such a statement can be found in van~Erp's thesis \cite{E05}.

\begin{lemma}[Christ--Geller--G{\l}owacki--Polin]\label{L:Rockland}
A principal cosymbol $a\in\Sigma_x^s(E,F)$ at $x\in M$ satisfies the Rockland condition iff it admits a left inverse $b\in\Sigma_x^{-s}(F,E)$, that is, $ba=1$.
\end{lemma}

\begin{proof}
The necessity of the Rockland condition is obvious, see \eqref{E:piab}.
To prove the non-trivial implication, we will present an elementary argument, reducing the statement to the well known scalar case due to Christ--Geller--G{\l}owacki--Polin, see \cite[Theorem~6.2]{CGGP92}.
For $s=0$ this has been proved in \cite{G91}.

With respect to bases of $E_x$ and $F_x$, the cosymbol $a\in\Sigma^s_x(E,F)$ corresponds to a matrix $A$ with entries in $\Sigma^s_x$, the space of principal cosymbols of order $s$ at $x$ for scalar operators.
Moreover, $\bar\pi(a)$ corresponds to the matrix $\bar\pi(A)$ obtained by applying the representation $\pi$ to each entry of $A$, that is, $(\bar\pi(A))_{ij}=\bar\pi(A_{ij})$.
Hence, $a$ satisfies the Rockland condition in Definition~\ref{D:Rockland} iff the matrix $\bar\pi(A)$ acts injectively on $(\mathcal H_\infty)^m$.
Here $m:=\dim(E_x)$ is the number of columns of $A$.
Using induction on $m$, we will show that there exists a matrix $B$ with entries in $\Sigma^{-s}_x$ such that $BA=I_m$ where $I_m$ denotes the $m\times m$ unit matrix.

Clearly, $A^*A$ is an $(m\times m)$-matrix with entries in $\Sigma^{2s}_x$ which satisfies the Rockland condition, see \eqref{E:piab} and \eqref{E:pia*}.
Let $y:=\sum_j(A_{j1})^*A_{j1}\in\Sigma^{2s}_x$ denote the entry in the upper left corner of $A^*A$.
Clearly, $y$ satisfies the (scalar) Rockland condition.
Hence, according to \cite[Theorem~6.2]{CGGP92}, there exists $z\in\Sigma^{-2s}_x$ such that $zy=1$.
Since $y^*=y$, we also have $yz=1$, whence $z$ is invertible.
Hence the diagonal matrix $D:=zI_m$ is invertible and, thus, $DA^*A$ satisfies the Rockland condition. 
The matrix $DA^*A$ has entries in $\Sigma^0_x$ and, by construction, its entry in the upper left corner is 1. 
Performing elementary row operations, we find an invertible matrix $L$ with entries in $\Sigma^0_x$ such that
$$
LDA^*A=\begin{pmatrix}1&*\\0&Y\end{pmatrix}
$$
where $Y$ is an $(m-1)\times(m-1)$-matrix with entries in $\Sigma^0_x$. 
Since $L$ is invertible, $LDA^*A$ satisfies the Rockland condition and, thus, $Y$ satisfies the Rockland condition too. 
By induction, there exists an $(m-1)\times(m-1)$-matrix $Z$ with entries in $\Sigma^0_x$ such that $ZY=I_{m-1}$. 
The matrix $C:=\begin{pmatrix}1&0\\0&Z\end{pmatrix}$ has entries in $\Sigma^0_x$ and, by construction, 
$$
CLDA^*A=\begin{pmatrix}1&*\\0&I_{m-1}\end{pmatrix}.
$$
Performing further elementary row operations, we find an invertible matrix $U$ with entries in $\Sigma^0_x$ such that $UCLDA^*A=I_m$. 
Hence, the matrix $B:=UCLDA^*$ has entries in $\Sigma^{-s}_x$ and satisfies $BA=I_m$.
\end{proof}

\begin{lemma}\label{L:pointwiseinv}
Let $E$ and $F$ be two vector bundles over a filtered manifold $M$.
Consider $a\in\Sigma^s(E,F)$ and suppose $\ev_{x_0}(a)$ is left invertible at some point $x_0\in M$, that is, there exists $b_{x_0}\in\Sigma_{x_0}^{-s}(F,E)$ such that $b_{x_0}\ev_{x_0}(a)=1$.
Then there exists an open neighborhood $U$ of $x_0$ and $b\in\Sigma^{-s}(F|_U,E|_U)$ such that $ba|_U=1$.
Moreover, a similar statement involving right inverses holds true.
\end{lemma}

\begin{proof}
If the bundle of osculating algebras is locally trivial, then this follows from \cite{CGGP92}, at least in the scalar case.
To handle general bundles of osculating algebras, we will proceed as in \cite[Section~3.3.3]{P08} where Ponge considers (in general non-contact) Heisenberg manifolds with varying osculating algebras, see also \cite[Section~6]{M82}.

Choose $\tilde b\in\Sigma^{-s}(F,E)$ such that $\ev_{x_0}(\tilde b)=b_{x_0}$.
Putting $c:=\tilde ba$, we have $c\in\Sigma^0(E)$ and $\ev_{x_0}(c)=1$.
It suffices to find an open neighborhood $U$ of $x_0$ in $M$ such that $c|_U$ is invertible in $\Sigma^0(E|_U)$, for then $b:=c|_U^{-1}\tilde b$ is the desired local left inverse of $a$.

We will identify $\Sigma^0(E)=\hat\Sigma^0(E)$ where $\hat\Sigma^0(E)$ denotes the space of all $k\in\mathcal K(\mathcal TM;E,E)$ which are strictly homogeneous of order zero, that is, $(\delta_\lambda)_*k=k$ for all $\lambda>0$, see Lemma~\ref{L:EYvsCGGP}.
Convolution with $k\in\hat\Sigma^0_x(E)$ gives rise to a bounded operator on $L^2(\mathcal T_xM)\otimes E_x$, see \cite[Theorem~6.19]{FS82}.
We fix an auxiliary fiber-wise Hermitian inner product on $E$, as well as a smooth family of invariant volume densities on the osculating groups $\mathcal T_xM$, and let $\bbb k\bbb_x$ denote the operator norm with respect to the associated Hermitian inner product on $L^2(\mathcal T_xM)\otimes E_x$.

We fix a fiber-wise homogeneous norm $|-|$ on $\mathcal TM$ which is smooth on $\mathcal TM\setminus M$.
This permits decomposing each $k\in\hat\Sigma^0_x(E)$ uniquely in the form $k=c_x(k)\delta_x+\pv_x(k)$ where $c_x(k)\in\eend(E_x)$, and $\pv_x(k)\in\hat\Sigma^0_x(E)$ is the principal value distribution
$$
\langle\pv_x(k),\psi\rangle:=\lim_{\varepsilon\to0}\int_{\{g\in\mathcal T_xM:|g|\geq\varepsilon\}}k\psi,
$$
see \cite[Proposition~6.13]{FS82}.
For each $r\in\N_0$ we consider the norm
$$
\|k\|_{x,r}:=|c_x(k)|+\left(\int_{\mathcal S_xM}|j^r(k|_{\mathcal S_xM})|^2\mu_x\right)^{1/2},
$$
where $k\in\hat\Sigma^0_x(E)$.
Here $j^r(k|_{\mathcal S_xM})$ denotes the $r$-jet of the restriction of $k$ to the sphere $\mathcal S_xM:=\{g\in\mathcal T_xM:|g|=1\}$, and we use smooth fiber-wise Hermitian metric on the $r$-jet bundle $J^r(\mathcal S_xM,E_x)$ which depends smoothly on $x$, and we are using a smooth volume density $\mu_x$ on the sphere $\mathcal S_xM$ which depends smoothly on $x$.

There exists $r_0\in\N_0$ and constants $C_x\geq0$ such that $\bbb k\bbb_x\leq C_x\|k\|_{x,r_0}$ for all $k\in\hat\Sigma^0_x(E)$.
This follows from a result due to Folland and Stein, see \cite[Theorem~6.19]{FS82}, and the Sobolev embedding theorem.
Ponge observed, see \cite[Lemma~3.3.13]{P08}, that the proof of Folland and Stein allows choosing the constants $C_x$ uniformly over compact subsets of $M$.
To make this more precise, we put, for every compact subset $L$ of $M$ and each $k\in\hat\Sigma^0(E)$,
$$
\bbb k\bbb_L:=\sup_{x\in L}\bbb\ev_x(k)\bbb_x\qquad\text{and}\qquad \|k\|_{L,r}:=\sup_{x\in L}\|\ev_x(k)\|_{x,r}.
$$
Then there exists a constant $C_L\geq0$ such that
\begin{equation}\label{E:bbb}
\bbb k\bbb_L\leq C_L\|k\|_{L,r_0}
\end{equation}
holds for all $k\in\hat\Sigma^0(E)$.\footnote{Actually, $r_0=0$ appears to be sufficient, see \cite{C88} and \cite[Remark~3.3.14]{P08}, but we won't need that.}
Moreover, see \cite[Lemma~3.3.15]{P08} and \cite[Lemma~5.7]{CGGP92}, for every $r\geq r_0$ there exists a constant $C_{L,r}\geq0$ such that 
$$
\|k_2*k_1\|_{L,r}\leq C_{L,r}\Bigl(\|k_2\|_{L,r}\cdot\bbb k_1\bbb_L+\bbb k_2\bbb_L\cdot\|k_1\|_{L,r}\Bigr)
$$
holds for all $k_1,k_2\in\hat\Sigma^0(E)$.
As in \cite[Lemma~4]{C88}, this gives 
$$
\|k^{2i}\|^{1/2i}_{L,r}\leq(2C_{L,r})^{1/2i}\sqrt{\bbb k^i\bbb^{1/i}_L}\sqrt{\|k^i\|_{L,r}^{1/i}}\leq(2C_{L,r})^{1/2i}\sqrt{\bbb k\bbb}_L\sqrt{\|k^i\|_{L,r}^{1/i}}
$$ 
and passing to the limit, we obtain
\begin{equation}\label{E:geom}
\lim_{i\to\infty}\|k^i\|^{1/i}_{L,r}\leq\bbb k\bbb_L
\end{equation}
for all $r\geq r_0$ and all $k\in\hat\Sigma^0(E)$.

To invert $c$ we write $c=1-k$ where $k\in\hat\Sigma^0(E)$.
Then $\ev_{x_0}(k)=0$, and there exists a compact neighborhood $L$ of $x_0$ such that $\bbb k\bbb_L<1$, see \eqref{E:bbb}.
In view of \eqref{E:geom}, the Neumann series $\sum_{i=0}^\infty k^i$ converges with respect to the norm $\|-\|_{L,r}$ for all $r\geq r_0$.
For each $x\in L$ we conclude that $\ev_x(c)$ is invertible in $\hat\Sigma^0_x(E)$ with inverse $\ev_x(c)^{-1}=\sum_{i=0}^\infty\ev_x(k)^i$.
This also shows that $\ev_x(c)^{-1}$ depends continuously on $x\in L$.
Proceeding as in the proof of \cite[Proposition~3.3.11]{P08}, we see that $\ev_x(c)^{-1}$ actually depends smoothly on $x$ in the interior $U$ of $L$, see also \cite[Proposition~5.10]{CGGP92}.
Hence, $c|_U$ is invertible in $\hat\Sigma^0(E|_U)$ with inverse $c|_U^{-1}=\sum_{i=0}^\infty k|_U^i$.
\end{proof}

Combining these results, we obtain the following Rockland \cite{R78} type theorem generalizing Melin's result for scalar differential operators \cite[Theorem~7.1]{M82}, see also \cite[Theorem~2.5(d)]{CGGP92}, \cite[Theorem~0.1]{HN79} and \cite[Theorem~3.3.10 and 5.4.1]{P08}.

\begin{theorem}\label{T:Rockland}
Let $E$ and $F$ be vector bundles over a filtered manifold $M$. 
Let $s$ be a complex number, and suppose $A\in\Psi^s(E,F)$ satisfies the Rockland condition.
Then there exists a left parametrix $B\in\Psi^{-s}_\prop(F,E)$ such that $BA-\id$ is a smoothing operator.
In particular, $A$ is hypoelliptic. 
If, moreover, $M$ is closed, then $\ker(A)$ is a finite dimensional subspace of\/ $\Gamma^\infty(E)$.
\end{theorem}

\begin{proof}
According to Lemma~\ref{L:Rockland} the principal cosymbol $\sigma^s_x(A)$ admits a left inverse at each point $x\in M$.
Hence, in view of Lemma~\ref{L:pointwiseinv}, we see that the principal cosymbol $\sigma^s(A)$ locally admits left inverses.
Using a smooth partition of unity on $M$, we obtain a global left inverse $b\in\Sigma^{-s}(F,E)$ such that $b\sigma^s(A)=1$.
Applying Proposition~\ref{P:Psi}\itemref{P:Psi:parametrix}, we obtain $B\in\Psi^{-s}_\prop(F,E)$ such that $BA-\id$ is a smoothing operator.
The remaining assertions are now obvious, cf.\ the proof of Corollary~\ref{C:hypo} above.
\end{proof}

Note that Theorem~\ref{T:para} follows immediately from Theorem~\ref{T:Rockland}.

\begin{corollary}\label{C:PsiinvA}
Let $E$ be a vector bundle over a closed filtered manifold $M$.
Suppose $A\in\Psi^s(E)$ satisfies the Rockland condition and is formally selfadjoint, $A^*=A$, with respect to an $L^2$ inner product of the form \eqref{E:llrr}, and let $Q$ denote the orthogonal projection onto the (finite dimensional) subspace $\ker(A)\subseteq\Gamma^\infty(E)$.
Then $A+Q$ is invertible with inverse $(A+Q)^{-1}\in\Psi^{-s}(E)$.
\end{corollary}

\begin{proof}
Proceeding exactly as in the proof of Corollary~\ref{C:smooth-Hodge-decomposition}, we start with a left parametrix $B\in\Psi^{-s}(E)$, see Theorem~\ref{T:Rockland}, and observe that the injective and formally selfadjoint parametrix $P$ constructed there is contained in $\Psi^{-s}(E)$.
As explained in the proof of Corollary~\ref{C:smooth-Hodge-decomposition}, $G:=(A+Q)P\in\Psi^0(E)$ is invertible, and $G^{-1}-\id$ is a smoothing operator.
Clearly, this implies $G^{-1}\in\Psi^0(E)$, whence $(A+Q)^{-1}=PG^{-1}\in\Psi^{-s}(E)$.
\end{proof}

Combining these observations with a result from \cite{CGGP92} permits constructing operators of arbitrary order which are invertible up to smoothing operators.
More precisely, we have:

\begin{lemma}\label{L:Lambda}
Let $E$ be a vector bundle over a filtered manifold $M$.
Then, for every complex number $s$, there exist $\Lambda\in\Psi^s_\prop(E)$ and $\Lambda'\in\Psi^{-s}_\prop(E)$ 
such that $\Lambda\Lambda'-\id$ and $\Lambda'\Lambda-\id$ are both smoothing operators.
Moreover, $\Lambda$ and $\Lambda'$ may be chosen so that they act injectively on $\Gamma^{-\infty}_c(E)$.
For closed $M$, there exist $\Lambda\in\Psi^s(E)$ and $\Lambda'\in\Psi^{-s}(E)$ such that $\Lambda\Lambda'=\id=\Lambda'\Lambda$.
\end{lemma}

\begin{proof}
For each $x\in M$ we fix an invertible cosymbol $a_x\in\Sigma^{s/4}_x(E)$ at $x$, see \cite[Theorem~6.1]{CGGP92}.
We extend these to cosymbols $\tilde a_x\in\Sigma^{s/4}(E)$ such that $\ev_x(\tilde a_x)=a_x$.
By Lemma~\ref{L:pointwiseinv}, for each $x\in M$, there exists an open neighborhood $U_x$ of $x$ such that $\tilde a_x|_{U_x}$ is invertible in $\Sigma^{s/4}(E|_{U_x})$.
Fix a smooth partition of unity $\lambda_x$, $x\in M$, such that $\supp(\lambda_x)\subseteq U_x$ for each $x\in M$.
With respect to a fiber-wise Hermitian metric on $E$, we consider 
$$
a:=\sum_{x\in M}\lambda_x\tilde a_x^*\tilde a_x\in\Sigma^{s/2}(E).
$$
We claim that $a$ is an invertible cosymbol.
To see this note first that, in view of Lemma~\ref{L:pointwiseinv}, it suffices to show that $\ev_y(a)$ admits an inverse in $\Sigma^{-s/2}_y(E)$ for each $y\in M$.
For fixed $y\in M$, there exists $x\in M$ such that $\lambda_x(y)>0$ and, by construction, $\ev_y(\tilde a_x)$ is invertible.
Using \eqref{E:piab} and \eqref{E:pia*} we conclude that $\ev_y(a)$ satisfies the Rockland condition and, thus, admits a left inverse in $\Sigma^{-s/2}_y(E)$, see Lemma~\ref{L:Rockland}.
Actually, $\ev_y(a)$ is invertible since $\ev_y(a)=\ev_y(a)^*$.
This shows that $a$ is invertible, hence there exists $a'\in\Sigma^{-s/2}(E)$ such that $aa'=1=a'a$.

According to Proposition~\ref{P:Psi}\itemref{P:Psi:symbsequ} there exist $A\in\Psi^{s/2}_\prop(E)$ such that $\sigma^{s/2}(A)=a$.
Using Proposition~\ref{P:Psi}\itemref{P:Psi:parametrix} we obtain $A'\in\Psi^{-s/2}_\prop(E)$ such that $R:=A'A-\id$ is a smoothing operator.
Proposition~\ref{P:Psi}\itemref{P:Psi:parametrix} also gives $A''\in\Psi_\prop^{-s/2}(E)$ such that $AA''-\id$ is a smoothing operator.
Since $A''$ differs from $A'$ by a smoothing operator, $AA'-\id$ is a smoothing operator too.

Consider $\Lambda\in\Psi^s_\prop(E)$ and $\Lambda'\in\Psi^{-s}_\prop(E)$ defined by 
$$
\Lambda:=A^*A+R^*R\qquad\text{and}\qquad\Lambda':=A'(A')^*+RR^*,
$$ 
where the adjoints are with respect to an $L^2$ inner product of the form \eqref{E:llrr}.
One readily verifies that $\Lambda'\Lambda-\id$ and $\Lambda\Lambda'-\id$ are both smoothing operators.
In particular, $\Lambda$ is hypoelliptic and, thus, every distributional section in the kernel of $\Lambda$ has to be smooth.
Using~\eqref{E:A*}, we conclude, $\ker(\Lambda|_{\Gamma^{-\infty}_c(E)})\subseteq\ker(A)\cap\ker(R)\subseteq\ker(\id)=0$.
Analogously, $\ker(\Lambda'|_{\Gamma^{-\infty}_c(E)})\subseteq\ker((A')^*)\cap\ker(R^*)\subseteq\ker(\id)=0$ in view of $R^*=A^*(A')^*-\id$.
If the underlying manifold is closed, then $\Lambda^{-1}\in\Psi^{-s}(E)$ according to Corollary~\ref{C:PsiinvA}.
\end{proof}

\begin{remark}[Right parametrix]\label{R:rightpara}
If $A\in\Psi^s(E,F)$ and $A^t$ satisfies the Rockland condition, then there exists a right parametrix $B\in\Psi_\prop^{-s}(F,E)$ such that $AB-\id$ is a smoothing operator.
Indeed, by Theorem~\ref{T:Rockland} there exists $B'\in\Psi_\prop^{-s}(E',F')$ such that $B'A^t-\id$ is a smoothing operator.
Hence, $B:=(B')^t$ is the desired right parametrix for $A$.
\end{remark}

\begin{corollary}\label{R:AAtRock}
If $A\in\Psi^s(E,F)$ is such that $A$ and $A^t$ both satisfy the Rockland condition, then there exists a parametrix $B\in\Psi_\prop^{-s}(F,E)$ such that $AB-\id$ and $BA-\id$ are both smoothing operators.
In particular, the principal cosymbol $\sigma^s(A)$ is invertible.
\end{corollary}

\begin{proof}
This follows immediately from Theorem~\ref{T:Rockland} and Remark~\ref{R:rightpara} since every left pa\-ra\-me\-trix differs from any right parametrix by a smoothing operator.
\end{proof}

\begin{remark}\label{R:AtRock}
For $A\in\Psi^s(E,F)$ and $x\in M$ the following are equivalent:
\begin{enumerate}[(a)]
\item $A^t$ satisfies the Rockland condition at $x$.
\item $A^*$ satisfies the Rockland condition at $x$.
\item $\bar\pi(\sigma_x^s(A))\colon\mathcal H_\infty\otimes E_x\to\mathcal H_\infty\otimes F_x$ has dense image, for all non-trivial irreducible unitary representations $\pi\colon\mathcal T_xM\to U(\mathcal H)$.
\item $\bar\pi(\sigma_x^s(A))\colon\mathcal H_\infty\otimes E_x\to\mathcal H_\infty\otimes F_x$ is onto, for all non-trivial irreducible unitary representations $\pi\colon\mathcal
 T_xM\to U(\mathcal H)$.
\end{enumerate}
Indeed, the equivalence (a)$\Leftrightarrow$(b) is clear in view of \eqref{E:A*At}.
The equivalence (b)$\Leftrightarrow$(c) follows from $\bar\pi(\sigma^s_x(A))^*=\bar\pi(\sigma^{\bar s}_x(A^*))$, see \eqref{E:pia*} and Remark~\ref{R:Psi:adjoint}.
To see the implication (a)$\Rightarrow$(d) suppose $A^t$ satisfies the Rockland condition at $x$.
According to Lemma~\ref{L:Rockland}, there exists $b\in\Sigma^{-s}_x(E',F')$ such that $b\sigma^s_x(A^t)=1$.
Transposing this equation, we obtain $\sigma^s_x(A)b^t=1$, see Proposition~\ref{P:Psi}\itemref{P:Psi:trans}.
In view of \eqref{E:piab}, this implies (d), for $\bar\pi(b^t)$ maps $\mathcal H_\infty\otimes F_x$ into $\mathcal H_\infty\otimes E_x$.
\end{remark}

\subsection{The Heisenberg Sobolev scale}\label{SS:sobolev}

The properties of the operator class $\Psi^s(E,F)$ discussed above permit introducing a Heisenberg Sobolev scale on filtered manifolds which can be used to refine the hypoellipticity results in Section~\ref{SS:hesDO}, see Corollaries~\ref{C:regrockseq} and \ref{C:HsHodge-seq} at the end of this section.
The main properties of this Sobolev scale are summarized in Proposition~\ref{P:Hs} below, a refined regularity statement including maximal hypoelliptic estimates can be found in Corollary~\ref{C:reg}.

For non-degenerate CR manifolds the origins of this Sobolev scale can be traced back to a paper of Folland and Stein \cite{FS74} where $L^p$ Sobolev spaces for integral $s$ are constructed using differential operators, see \cite[Section~15]{FS74}, and maximal hypoelliptic estimates for Kohn's Laplacian are established, see \cite[Theorem~16.6]{FS74} and also \cite[Theorem~16.7]{FS74}.
For Heisenberg manifolds satisfying only the bracket generating condition $H+[H,H]=TM$, a full Sobolev scale has been constructed by Ponge using complex powers of subLaplacians, see \cite[Section~5.5]{P08} and \cite[Propositions~5.5.9 ad 5.5.14]{P08}.
Maximal hypoelliptic estimates can also be found in \cite{HN85} and the work of Beals and Greiner, see \cite[Theorem~18.31]{BG88}.
A full Sobolev scale for stratified Lie groups has been constructed in \cite{F75}.

For every real $s$ we let $H^s_\loc(E)$ denote the space of all distributional sections $\psi\in\Gamma^{-\infty}(E)$ such that $A\psi\in L^2_\loc(F)$ for all $A\in\Psi^s_\prop(E,F)$ and all vector bundles $F$.
We equip $H^s_\loc(E)$ with the coarsest topology such that the maps $A\colon H^s_\loc(E)\to L^2_\loc(F)$ are continuous, for all $A\in\Psi_\prop^s(E,F)$.
Analogously, let $H^s_c(E)$ denote the space of compactly supported distributional sections $\psi\in\Gamma_c^{-\infty}(E)$ such that $A\psi\in L^2_\loc(F)$ for all $A\in\Psi^s(E,F)$, and equip $H^s_c(E)$ with the coarsest topology such that the corresponding maps $A\colon H^s_c(E)\to L^2_\loc(F)$ are continuous, for all $A\in\Psi^s(E,F)$ and all vector bundles $F$.

\begin{proposition}\label{P:Hs}
Let $E$ and $F$ be vector bundles over a filtered manifold $M$.
\begin{enumerate}[(a)]
\item\label{P:Hs:locfilt}
$H^s_\loc(E)$ is a complete locally convex vector space, and we have continuous inclusions
$$
\Gamma^\infty(E)\subseteq H^{s_2}_\loc(E)\subseteq H^{s_1}_\loc(E)\subseteq\Gamma^{-\infty}(E)
$$
whenever $s_1\leq s_2$.
Moreover, $H^0_\loc(E)=L^2_\loc(E)$ as locally convex spaces.
\item\label{P:Hs:cfilt}
$H^s_c(E)$ is a complete locally convex vector space, and we have continuous inclusions
$$
\Gamma^\infty_c(E)\subseteq H^{s_2}_c(E)\subseteq H^{s_1}_c(E)\subseteq\Gamma^{-\infty}_c(E)
$$
whenever $s_1\leq s_2$.
Moreover, $H^0_c(E)=L^2_c(E)$ as locally convex spaces.
\item\label{P:Hs:pairing}
The canonical pairing $\Gamma^\infty_c(E')\times\Gamma^\infty(E)\to\C$ extends to a pairing 
$$
H^{-s}_c(E')\times H^s_\loc(E)\to\C
$$ 
inducing linear bijections $H^s_\loc(E)^*=H^{-s}_c(E')$ and $H^{-s}_c(E')^*=H^s_\loc(E)$.
If, moreover, $M$ is closed, then $H^s_c(E)=H^s_\loc(E)$ is a Hilbert space we denote by $H^s(E)$, and the pairing induces an isomorphism of Hilbert spaces, $H^s(E)^*=H^{-s}(E')$.
\item\label{P:Hs:operators}
Each $A\in\Psi^r(E,F)$ restricts to continuous operator $A\colon H^s_c(E)\to H_\loc^{s-\Re(r)}(F)$.
On a closed manifold we obtain a bounded operator $A\colon H^s(E)\to H^{s-\Re(r)}(F)$.
\item\label{P:Hs:innerproduct}
Assume $M$ closed and let $\Lambda_s\in\Psi^s(E)$ be invertible with $\Lambda_s^{-1}\in\Psi^{-s}(E)$, see Lemma~\ref{L:Lambda}.
If $\llangle-,-\rrangle_{L^2(E)}$ is any Hermitian inner product generating the topology on $L^2(E)$, then
\begin{equation}\label{E:llrrSobolev}
\llangle\psi_1,\psi_2\rrangle_{H^s(E)}:=\llangle\Lambda_s\psi_1,\Lambda_s\psi_2\rrangle_{L^2(E)},
\end{equation}
$\psi_1,\psi_2\in H^s(E)$, is a Hermitian inner product generating the topology on $H^s(E)$.
\item\label{P:Hs:compact}
If $s_1<s_2$, then the inclusions $H^{s_2}_\loc(E)\subseteq H^{s_1}_\loc(E)$ and $H^{s_2}_c(E)\subseteq H^{s_1}_c(E)$ are compact.
For closed $M$ we obtain a compact inclusion $H^{s_2}(E)\subseteq H^{s_1}(E)$.
\item\label{P:Hs:Sobolevemb}
If $M$ has homogeneous dimension $n$, see~\eqref{E:hdim}, then we have continuous Sobolev embeddings $H^s_\loc(E)\subseteq\Gamma^r(E)$ and $H^s_c(E)\subseteq\Gamma^r_c(E)$ for all $r\in\mathbb N_0$ with $r<s-n/2$, see Remark~\ref{R:GammarE}.
In particular, we obtain isomorphisms of locally convex spaces,
$$
\Gamma^\infty(E)=\bigcap_s H^s_\loc(E)
\qquad\text{and}\qquad
\Gamma_c^\infty(E)=\bigcap_s H^s_c(E),
$$ 
as well as $\Gamma^{-\infty}(E)=\bigcup_sH^s_\loc(E)$ and\/ $\Gamma^{-\infty}_c(E)=\bigcup_sH^s_c(E)$.
\end{enumerate}
\end{proposition}

\begin{proof}
The proof is a routine extension, cf.\ for instance \cite[Section~\S7]{S01} for the classical case, of the results established in the preceding sections.
For each complex number $s$ we choose operators $\Lambda_s\in\Psi^s_\prop(E)$, $\Lambda'_s\in\Psi^{-s}_\prop(E)$, and $R_s\in\mathcal O^{-\infty}_\prop(E)$ such that, see Lemma~\ref{L:Lambda},
$$
\Lambda_s'\Lambda_s=\id+R_s.
$$

We have a continuous inclusion $\Gamma^\infty(E)\subseteq H^s_\loc(E)$ since every operator $A\in\Psi^s_\prop(E,F)$ induces a continuous map $A\colon\Gamma^\infty(E)\to\Gamma^\infty(F)$, see Proposition~\ref{P:Psi}\itemref{P:Psi:O}, and the inclusion $\Gamma^\infty(F)\subseteq L^2_\loc(F)$ is continuous.
Using Proposition~\ref{P:Psi}\itemref{P:Psi:O} we see that the composition
$H^s_\loc(E)\xrightarrow{\Lambda_s}L^2_\loc(E)\subseteq\Gamma^{-\infty}(E)\xrightarrow{\Lambda'_s}\Gamma^{-\infty}(E)$
is continuous.
Moreover, $\mathcal O^{-\infty}_\prop(E)\subseteq\Psi^s_\prop(E)$, see Proposition~\ref{P:Psi}\itemref{P:Psi:smooth}, and thus the composition
$H^s_\loc(E)\xrightarrow{R_s}L^2_\loc(E)\subseteq\Gamma^{-\infty}(E)$ is continuous.
We conclude that $\id=\Lambda_s'\Lambda_s-R_s$ induces a continuous map $H^s_\loc(E)\subseteq\Gamma^{-\infty}(E)$.
To see that we have continuous inclusions $H^{s_2}_\loc(E)\subseteq H^{s_1}_\loc(E)$ for all $s_1\leq s_2$, we have to show that each $A\in\Psi^{s_1}_\prop(E,F)$ induces a continuous operator $A\colon H^{s_2}_\loc(E)\to L^2_\loc(E)$.
To achieve that,
note that $A=\Lambda'_{s_2-s_1}\Lambda_{s_2-s_1}A-R_{s_2-s_1}A$.
Moreover, $\Lambda_{s_2-s_1}A\in\Psi^{s_2}_\prop(E,F)$, see Proposition~\ref{P:Psi}\itemref{P:Psi:mult}, and thus $\Lambda_{s_2-s_1}A\colon H^{s_2}_\loc(E)\to L^2_\loc(F)$ is continuous.
Moreover, $\Lambda'_{s_2-s_1}$ induces a continuous operator $L^2_\loc(F)\to L^2_\loc(F)$ since $s_1-s_2\leq0$, see Proposition~\ref{P:Psimap}\itemref{P:Psimap:L2cont}.
Furthermore, $R_{s_2-s_1}A\in\mathcal O^{-\infty}_\prop(E,F)\subseteq\Psi^{s_2}_\prop(E,F)$ in view of Proposition~\ref{P:Psi}\itemref{P:Psi:smooth} and thus $R_{s_2-s_1}A\colon H^{s_2}_\loc(E)\to L^2_\loc(F)$ is continuous.
We conclude that $A=\Lambda'_{s_2-s_1}\Lambda_{s_2-s_1}A-R_{s_2-s_1}A$ induces a continuous map $H^{s_2}_\loc(E)\to L^2_\loc(F)$, whence the continuous inclusion $H^{s_2}_\loc(E)\subseteq H^{s_1}_\loc(E)$.
The completeness of $H^s_\loc(E)$ follows from the continuity of the inclusion $H^s_\loc(E)\subseteq\Gamma^{-\infty}(E)$, the completeness of the spaces $\Gamma^{-\infty}(E)$, the fact that each $A\in\Psi^s_\prop(E,F)$ induces a continuous operator $A\colon\Gamma^{-\infty}(E)\to\Gamma^{-\infty}(F)$, see Proposition~\ref{P:Psi}\itemref{P:Psi:O}, and the completeness of the spaces $L^2_\loc(F)$.
We have a continuous inclusion $H^0_\loc(E)\subseteq L^2_\loc(E)$ since $\id\in\Psi^0_\prop(E)$, see Proposition~\ref{P:Psi}\itemref{P:Psi:DO}.
By Proposition~\ref{P:Psimap}\itemref{P:Psimap:L2cont}, we also have the converse continuous inclusion $L^2_\loc(E)\subseteq H^0_\loc(E)$.
This shows $H^0_\loc(E)=L^2_\loc(E)$ and completes the proof of \itemref{P:Hs:locfilt}.
Part \itemref{P:Hs:cfilt} can be proved analogously.

To see \itemref{P:Hs:pairing},
note that the canonical pairing can be written as
\begin{equation}\label{E:pairHs}
\langle\phi,\psi\rangle
=\langle(\Lambda'_s)^t\phi,\Lambda_s\psi\rangle
-\langle\phi,R_s\psi\rangle,
\end{equation}
where $\phi\in\Gamma^\infty_c(E')$, $\psi\in\Gamma^\infty_c(E)$, and
$(\Lambda'_s)^t\in\Psi^{-s}_\prop(E')$ according to Proposition~\ref{P:Psi}\itemref{P:Psi:trans}.
Recall that the canonical pairing extends to a pairing $L^2_c(E')\times L^2_\loc(E)\to\C$.
Since $\Lambda_s\colon H^s_\loc(E)\to L^2_\loc(E)$ and $(\Lambda_s')^t\colon H^{-s}_c(E')\to L^2_c(E')$ are continuous, we see that the term $\langle(\Lambda'_s)^t\phi,\Lambda_s\psi\rangle$ extends to a separately continuous bilinear map $H^{-s}_c(E')\times H^s_\loc(E)\to\C$.
Since $R_s$ induces a continuous operator $R_s\colon\Gamma^{-\infty}_\loc(E)\to\Gamma^\infty_\loc(E)$, the term $\langle\phi,R_s\psi\rangle$ actually extends to a separately continuous bilinear map $\Gamma^{-\infty}_c(E')\times\Gamma^{-\infty}_\loc(E)\to\C$.
Using \itemref{P:Hs:locfilt}, \itemref{P:Hs:cfilt} and \eqref{E:pairHs} we conclude that the canonical pairing extends to a separately continuous bilinear map $H^{-s}_c(E')\times H^s_\loc(E)\to\C$.
Let us now verify that the induced linear map
\begin{equation}\label{E:HsHslin}
H^s_\loc(E)\to H^{-s}_c(E')^*
\end{equation}
is bijective.
Since the inclusion $\mathcal D(E)=\Gamma^\infty_c(E')\subseteq H^{-s}_c(E')$ is continuous, we obtain a continuous map $H^{-s}_c(E')^*\to\mathcal D(E)^*=\Gamma^{-\infty}(E)$ which, when composed with \eqref{E:HsHslin}, yields the canonical inclusion $H^s_\loc(E)\subseteq\Gamma^{-\infty}(E)$. This immediately implies that \eqref{E:HsHslin} is injective.
To see that it is onto, suppose $\alpha\in H^{-s}_c(E')^*$.
The preceding considerations show that there exists $\psi\in\Gamma^{-\infty}(E)$ such that $\langle\phi,\psi\rangle=\alpha(\phi)$ for all $\phi\in\mathcal D(E)$.
It remains to show that $\psi\in H^s_\loc(E)$.
To check this, let $A\in\Psi^s_\prop(E,F)$.
Then $\langle\tilde\phi,A\psi\rangle=\langle A^t\tilde\phi,\psi\rangle=(\alpha\circ A^t)(\tilde\phi)$ extends continuously to all $\tilde\phi\in L^2_c(F')$.
Since the pairing $L^2_c(F')\times L^2_\loc(F)\to\C$ induces a linear bijection $L^2_\loc(F)=L^2_c(F')^*$, we conclude $A\psi\in L^2_\loc(F)$, whence $\psi\in H^s_\loc(E)$.
Analogously, one can verify that the induced linear map $H^{-s}_c(E')\to H^s_\loc(E)^*$ is a bijection.
If $M$ is closed, then $\Lambda_s$ may be assumed to be invertible with inverse $\Lambda_s^{-1}\in\Psi^{-s}(E)$, see Lemma~\ref{L:Lambda}, hence $\Lambda_s$ induces a topological isomorphism $H^s(E)\cong L^2(E)$, whence $H^s(E)$ is a Hilbert space.
This also implies that the canonical pairing induces an isomorphism of Hilbert spaces, $H^s(E)^*=H^{-s}(E')$, for we have $\langle\phi,\psi\rangle=\langle(\Lambda_s^{-1})^t\phi,\Lambda_s\psi\rangle$ and the canonical pairing induces an isomorphism of Hilbert spaces $L^2(E)^*=L^2(E')$.

The mapping properties in \itemref{P:Hs:operators} are immediate consequences of Proposition~\ref{P:Psi}\itemref{P:Psi:mult}, provided $r$ is real.
For complex $r$, we use Lemma~\ref{L:Lambda} and Proposition~\ref{P:Psimap}\itemref{P:Psimap:L2cont}.

The statement in \itemref{P:Hs:innerproduct} is now obvious, for $\Lambda_s\colon H^s(E)\to L^2(E)$ is a topological isomorphism with inverse induced by $\Lambda_s^{-1}\colon L^2(E)=H^0(E)\to H^s(E)$, see \itemref{P:Hs:operators}.

To prove \itemref{P:Hs:compact}, note that $\Lambda_{s_1}\Lambda_{s_2}'\in\Psi^{s_1-s_2}_\prop(E)$ and $s_1-s_2<0$, hence $\Lambda_{s_1}\Lambda_{s_2}'$ induces a compact operator on $L^2_\loc(E)$, see Proposition~\ref{P:Psimap}\itemref{P:Psimap:compact}.
Hence, the operator $\Lambda_{s_1}'\Lambda_{s_1}\Lambda_{s_2}'\Lambda_{s_2}\colon H^{s_2}_\loc(E)\to H^{s_1}_\loc(E)$ is compact.
The latter operator differs from the canonical inclusion by a properly smoothing operator for we have $\Lambda_{s_1}'\Lambda_{s_1}\Lambda_{s_2}'\Lambda_{s_2}=\id+R_{s_1}+R_{s_2}+R_{s_1}R_{s_2}$.
Moreover, properly supported smoothing operators induce compact operators on each Sobolev space $H^s_\loc(E)$ for the continuous inclusion $\Gamma^\infty(E)\subseteq H^s_\loc(E)$ is compact in view of the Heine--Borel property of $\Gamma^\infty(E)$ which asserts that the identical map on $\Gamma^\infty(E)$ is compact (Arzel\`a--Ascoli).
We conclude that the inclusion $H^{s_2}_\loc(E)\subseteq H^{s_1}_\loc(E)$ is compact.
Similarly, one shows that the inclusion $H^{s_2}_c(E)\subseteq H^{s_1}_c(E)$ is compact.

To prove \itemref{P:Hs:Sobolevemb} we have to show that every differential operator $D\in\DO^r(E,F)$ of Heisenberg order at most $r<s-n/2$ induces a continuous operator $H^s_\loc(E)\to\Gamma(F)$.
By Proposition~\ref{P:Psimap}\itemref{P:Psimap:kL2}, $D\Lambda_s'\in\Psi_\prop^{r-s}(E,F)$ induces a continuous operator $D\Lambda_s'\colon L^2_\loc(E)\to\Gamma(F)$ and, thus, $D\Lambda_s'\Lambda_s\colon H^s_\loc(E)\to\Gamma(F)$ is continuous.
Since the inclusion $\Gamma^\infty(E)\subseteq\Gamma(E)$ is continuous, $DR_s\in\mathcal O^{-\infty}_\prop(E,F)$ induces a continuous operator $DR_s\colon L^2(E)\to\Gamma(F)$.
Using $D=D\Lambda_s'\Lambda_s+DR_s$, we conclude that $D\colon H^s_\loc(E)\to\Gamma(F)$ is indeed continuous, whence the continuous inclusion $H^s_\loc(E)\subseteq\Gamma^r(E)$.
Analogously, one verifies the continuous inclusion $H^s_c(E)\subseteq\Gamma_c^r(E)$, provided $r<s-n/2$.
\end{proof}


\begin{remark}
For non-degenerate CR manifolds a boundedness statement as in Proposition~\ref{P:Hs}\itemref{P:Hs:operators} can be traced back to \cite[Theorem~15.19]{FS74}.
For Heisenberg manifolds satisfying the bracket generating condition $H+[H,H]=TM$ a similar statement has been established in \cite[Proposition~5.5.8]{P08}.
\end{remark}

\begin{remark}\label{R:CrHr}
Given $r\in\N_0$ one might wonder if the space $H^r_\loc(E)$ can be characterized as the space of all $\psi\in\Gamma^{-\infty}(E)$ such that $D\psi\in L^2_\loc(F)$ for all differential operators $D\in\DO^r(E,F)$ of Heisenberg order at most $r$; and if these differential operators suffice to generate the topology on $H^r_\loc(E)$.
Moreover, one might ask if $\Gamma^r(E)$ includes (continuously) in $H^r_\loc(E)$, see Remark~\ref{R:GammarE}.
For general $M$ and $r$ neither of these properties will hold true.
For instance, if $M=\mathbb R\times N$ is filtered such that $TM=T^{-2}M\supseteq T^{-1}M=\mathbb R\times TN$, then differential operators of Heisenberg order $1$ do not suffice to characterize $H^1_\loc(E)$ since they do not capture any derivative in the $\mathbb R$ direction.
Filtered manifolds of this type occur naturally when studying heat operators $\partial_t+\Delta$.
However, if $r$ is a common multiple of $\{1,2,\dotsc,m\}$ where $m$ is such that $\mathfrak tM=\bigoplus_{j=1}^m\mathfrak t^{-j}M$, then the universal differential operator $j^r\in\DO^r(E,J^rE)$, see Remark~\ref{R:jk}, is Rockland and we obtain affirmative answers to all the questions above.
If $T^{-1}M$ is bracket generating, this holds true for all $r$, cf.~\cite{FS74,BG88,P08}.
\end{remark}

This Sobolev scale permits us to refine the hypoellipticity statement in Theorem~\ref{T:Rockland}, providing a common generalization of well known results, see for instance \cite[Propositions~5.5.9 and 5.5.14]{P08},
\cite[Theorem~16.6]{FS74}, \cite{HN85}, and \cite[Theorem~18.31]{BG88}.

\begin{corollary}[Regularity and maximal hypoelliptic estimate]\label{C:reg}
Let $E$ and $F$ be vector bundles over a filtered manifold $M$, and suppose $A\in\Psi^k(E,F)$ satisfies the Rockland condition where $k$ is some complex number.
If $\psi\in\Gamma^{-\infty}_c(E)$ and $A\psi\in H_\loc^{s-\Re(k)}(F)$ for some real number $s$, then $\psi\in H^s_c(E)$.
If, moreover, $M$ is closed and $\tilde s\leq s$, then there exists a constant $C=C_{A,s,\tilde s}\geq0$ such that the maximal hypoelliptic estimate
$$
\|\psi\|_{H^s(E)}\leq C\left(\|\psi\|_{H^{\tilde s}(E)}+\|A\psi\|_{H^{s-\Re(k)}(F)}\right)
$$
holds for all $\psi\in H^s(E)$.
Here we are using any norms generating the Hilbert space topologies on the corresponding Heisenberg Sobolev spaces.
If, moreover, $Q$ denotes the orthogonal projection, with respect to an inner product of the form \eqref{E:llrr}, onto the (finite dimensional) subspace $\ker(A)\subseteq\Gamma^\infty(E)$, then there exists a constant $C=C_{A,s}\geq0$ such that the maximal hypoelliptic estimate
$$
\|\psi\|_{H^s(E)}\leq C\left(\|Q\psi\|_{\ker(A)}+\|A\psi\|_{H^{s-\Re(k)}(F)}\right)
$$
holds for all $\psi\in H^s(E)$. Here $\|-\|_{\ker(A)}$ denotes any norm on $\ker(A)$.
\end{corollary}

\begin{proof}
For the first part use a left parametrix as in Theorem~\ref{T:Rockland} and Proposition~\ref{P:Hs}\itemref{P:Hs:locfilt}\&\itemref{P:Hs:operators}.
To see the second hypoelliptic estimate, observe that $A^*A\in\Psi^{2\Re(k)}(E)$ satisfies the Rockland condition and $\ker(A^*A)=\ker(A)$.
Hence, $B:=(A^*A+Q)^{-1}A^*\in\Psi^{-k}(F,E)$, see Corollary~\ref{C:PsiinvA}, and $\id=Q+BA$.
\end{proof}

\begin{remark}
Let us complement the regularity statement above with the following description of the topologies on the Heisenberg Sobolev spaces.
Suppose $A\in\Psi^k_\prop(E,F)$ satisfies the Rockland condition and put $s=\Re(k)$.
Then the topologies on $H^s_\loc(E)$ and $H^s_c(E)$ coincide with the topologies induced from the embeddings
$$
H^s_\loc(E)\xrightarrow{(A,\id)}L^2_\loc(F)\times\Gamma^{-\infty}(E)
\qquad\text{and}\qquad
H^s_c(E)\xrightarrow{(A,\id)}L^2_c(F)\times\Gamma^{-\infty}_c(E),
$$
respectively.
Moreover, if $B\in\Psi^{-r}_\prop(F,E)$ is a left parametrix such that $R:=BA-\id$ is a smoothing operator, see Theorem~\ref{T:Rockland}, then the topologies on $H^s_\loc(E)$ and $H^s_c(E)$ coincide with the topologies induced from the embeddings
$$
H^s_\loc(E)\xrightarrow{(A,R)}L^2_\loc(F)\times\Gamma^\infty(E)
\qquad\text{and}\qquad
H^s_c(E)\xrightarrow{(A,R)}L^2_c(F)\times\Gamma^\infty_c(E),
$$
respectively.
Indeed, $B-R$ provides a continuous left inverse for the first two inclusions; and $B-\id$ provides a continuous left inverse for the other two inclusions.
\end{remark}

Accordingly, the Hodge type decomposition for formally selfadjoint Rockland operators on closed filtered manifolds in Corollary~\ref{C:smooth-Hodge-decomposition} admits the following refinement:

\begin{corollary}[Hodge decomposition]\label{C:HsHodge}
Let $E$ be a vector bundle over a closed filtered manifold $M$.
Suppose $A\in\Psi^k(E)$ satisfies the Rockland condition and is formally selfadjoint, $A^*=A$, with respect to an $L^2$ inner product of the form~\eqref{E:llrr}.
Moreover, let $Q$ denote the orthogonal projection onto the (finite dimensional) subspace $\ker(A)\subseteq\Gamma^\infty(E)$.
Then, for every real number $s$, we have a topological isomorphism and a Hodge type decomposition
$$
A+Q\colon H^{s+\Re(k)}(E)\xrightarrow\cong H^s(E),
\qquad H^s(E)=\ker(A)\oplus A\bigl(H^{s+\Re(k)}(E)\bigr).
$$
\end{corollary}

\begin{proof}
This follows from Corollary~\ref{C:PsiinvA} and Proposition~\ref{P:Hs}\itemref{P:Hs:operators}.
\end{proof}

\begin{corollary}[Fredholm operators and index]\label{C:Fredholm}
Let $E$ and $F$ be a vector bundles over a closed filtered manifold $M$.
Suppose $A\in\Psi^k(E,F)$ is such that $A$ and $A^t$ both satisfy the Rockland condition, cf.\ Remark~\ref{R:AtRock}.
Then, for every real number $s$, we have an induced Fredholm operator $A\colon H^s(E)\to H^{s-\Re(k)}(F)$ whose index is independent of $s$ and can be expressed as
$$
\ind(A)=\dim\ker(A)-\dim\ker(A^t).
$$
Moreover, this index only depends on the Heisenberg principal cosymbol $\sigma^k(A)\in\Sigma^k(E,F)$.
\end{corollary}

\begin{proof}
According to Remark~\ref{R:AAtRock} there exists a parametrix $B\in\Psi^{-k}(F,E)$ such that $BA-\id$ and $AB-\id$ are both smoothing operators.
Since the inclusion $\Gamma^\infty(E)\subseteq H^s(E)$ is compact, see Proposition~\ref{P:Hs}\itemref{P:Hs:compact}, smoothing operators are compact on the Sobolev spaces $H^s(E)$ and $H^{s-\Re(k)}(F)$, respectively.
Hence, $B\colon H^{s-\Re(k)}(F)\to H^s(E)$ provides an inverse of $A$ mod compact operators.
Consequently, $A\colon  H^s(E)\to H^{s-\Re(k)}(F)$ is Fredholm.
Moreover, the canonical pairing induces a canonical isomorphism $\coker(A)=\ker(A^t)$, see Proposition~\ref{P:Hs}\itemref{P:Hs:pairing}, whence the index formula above.
By regularity, $\ker(A)$ and $\ker(A^t)$ consist of smooth sections and, thus, do not depend on $s$.
If $A'\in\Psi^k(E,F)$ is another operator with the same Heisenberg principal cosymbol, $\sigma^k(A')=\sigma^k(A)$, then $A'-A\in\Psi^{k-1}(E,F)$ in view of Proposition~\ref{P:Psi}\itemref{P:Psi:symbsequ}, hence $A'-A\colon H^s(E)\to H^{s-\Re(k)}(F)$ is a compact operator according to Proposition~\ref{P:Hs}\itemref{P:Hs:operators}\&\itemref{P:Hs:compact}, and thus $\ind(A')=\ind(A)$.
\end{proof}

Let us now apply these results to the Rockland sequences considered in Section~\ref{SS:hesDO}.
Recall that the Rumin--Seshadri operators $\Delta_i$ satisfy the Rockland condition, see Lemma~\ref{L:rockseq}.
Hence, the refined regularity statement in Corollary~\ref{C:reg} applies to $\Delta_i$.
Moreover, since $\Delta_i$ is formally selfadjoint, we have Hodge type decompositions and maximal hypoelliptic estimates for $\Delta_i$ as in Corollary~\ref{C:HsHodge}, provided the underlying manifold is closed.
For the Rockland sequences, we immediately obtain the following refinement of Corollary~\ref{C:RShypo}.

\begin{corollary}[Regularity and maximal hypoelliptic estimate]\label{C:regrockseq}
In the situation of Corollary~\ref{C:RShypo}, if $\psi\in\Gamma^{-\infty}(E_i)$ is such that $A_i\psi\in H_\loc^{s-k_i}(E_{i+1})$ and $A_{i-1}^*\in H_\loc^{s-k_{i-1}}(E_{i-1})$ for some real number $s$, then $\psi\in H^s_\loc(E_i)$.
Moreover, if $M$ is closed, then there exists a constant $C=C_{A_i,s}\geq0$ such that the maximal hypoelliptic estimate
$$
\|\psi\|_{H^s(E_i)}\leq C\Bigl(\|A_{i-1}^*\psi\|_{H^{s-k_{i-1}}(E_{i-1})}+\|Q_i\psi\|_{\ker(\Delta_i)}+\|A_i\psi\|_{H^{s-k_i}(E_{i+1})}\Bigr)
$$
holds for all $\psi\in H^s(E_i)$.
Here $\|-\|_{H^s(E_i)}$ are any norms generating the Hilbert space topology on the Heisenberg Sobolev spaces $H^s(E_i)$, $Q_i$ denotes the orthogonal projection onto the (finite dimensional) subspace $\ker(\Delta_i)=\ker(A_{i-1}^*)\cap\ker(A_i)\subseteq\Gamma^\infty(E_i)$, and $\|-\|_{\ker(\Delta_i)}$ denotes any norm on $\ker(\Delta_i)$.
\end{corollary}

For Rockland complexes over closed manifolds we immediately obtain the following Hodge decomposition, refining the statement in Corollary \ref{C:Hodge}:

\begin{corollary}[Hodge decomposition]\label{C:HsHodge-seq}
In the situation of Corollary~\ref{C:Hodge} we have
$$
H^s(E_i)=A_{i-1}(H^{s+k_{i-1}}(E_{i-1}))\oplus\ker(\Delta_i)\oplus A_i^*(H^{s+k_i}(E_{i+1}))
$$
and $\ker(A_i|_{H^s(E_i)})=A_{i-1}(H^{s+k_{i-1}}(E_{i-1}))\oplus\ker(\Delta_i)$, for every real number $s$.
\end{corollary}

\begin{remark}
For the $\bar\partial$ complex on a non-degenerate CR manifold the preceding statement can be found in \cite[Theorem~17.1]{FS74} for $s=0$.
\end{remark}

The preceding two corollaries remain true for Rockland sequences of pseudodifferential operators.
We will formulate this precisely and more generally in Section~\ref{S:grRockland} below.

\section{Graded hypoelliptic sequences}\label{S:grhesequences}

The de~Rham differentials on a filtered manifold will in general have Heisenberg order strictly larger than one, and the de~Rham complex will in general not be Rockland in the sense of Definition~\ref{def.Hypo-seq}.
The de~Rham complex is, however, Rockland in an appropriate graded sense.
More generally, this remains true for the sequence obtained by extending a linear connection to vector bundle valued differential forms, provided the curvature satisfies an algebraic condition, see Proposition~\ref{P:hypo} below.
Regarding these de~Rham sequences as graded Rockland sequences, allows extracting (splitting off) graded Rockland sequences of (higher order) differential operators acting between vector bundles of smaller rank, see Theorem~\ref{T:D} and Corollary~\ref{C:D} below.
This construction generalizes a construction for the de Rham complex on contact manifolds \cite{R90,R94} and more general C-C spaces \cite{R99,R01} due to Rumin.
Since all curved Bernstein--Gelfand--Gelfand (BGG) sequences \cite{CSS01,CD01} of regular parabolic geometries appear in this way, the latter are all graded Rockland, see Corollary~\ref{C:BGG} below.
In some cases, the sequences thus obtain are even Rockland in the ungraded sense of Section~\ref{SS:hesDO}.
In particular, this happens for the BGG sequences associated with a generic rank two distribution in dimension five, see Example~\ref{Ex:BGG235} below.

\subsection{Filtered vector bundles and differential operators}\label{SS:fVBDO}

In this section we introduce the filtration by graded Heisenberg order on differential operators acting between sections of filtered vector bundles over filtered manifolds. 
We discuss the corresponding graded Heisenberg principal symbol and establish some of its basic properties.
For trivially filtered vector bundles these concepts reduce to the filtration by Heisenberg order and the principal Heisenberg symbol discussed in Section~\ref{SS:DO} above.

Let $M$ be a filtered manifold.
Suppose $E$ is a filtered vector bundle over $M$, i.e., a smooth vector bundle which comes equipped with a filtration by smooth subbundles,
$$
\cdots\supseteq E^{p-1}\supseteq E^p\supseteq E^{p+1}\supseteq\cdots.
$$ 
We will always assume that the filtration is full, that is, $E=\bigcup_pE^p$ and $\bigcap_pE^p=0$.
Put $\gr_p(E):=E^p/E^{p+1}$ and let $\gr(E):=\bigoplus_p\gr_p(E)$ denote the associated graded vector bundle equipped with the filtration by the subbundles $\gr^q(E):=\bigoplus_{q\leq p}\gr_p(E)$.
By a \emph{splitting of the filtration} on $E$ we mean a filtration-preserving vector bundle isomorphism $S_E\colon\gr(E)\to E$ which induces the identity on the level of associated graded, $\gr(S_E)=\id_{\gr(E)}$.
More explicitly, $S_E$ maps $\gr_p(E)$ into $E^p$ such that the composition with the projection $E^p\to E^p/E^{p+1}=\gr_p(E)$ is the identity.
Such splittings always exist, in fact the space of all splittings is convex, hence contractible.

Suppose $F$ is another filtered vector bundle over $M$, and let $S_F\colon\gr(F)\to F$ be a splitting for its filtration.
A differential operator $A\in\DO(E,F)$ is said to have \emph{graded Heisenberg order} at most $k$ if the operator $S_F^{-1}AS_E\in\DO(\gr(E),\gr(F))$ has the following property:
The component $(S_F^{-1}AS_E)_{q,p}\in\DO(\gr_p(E),\gr_q(F))$ in the decomposition according to the gradings, $S_F^{-1}AS_E=\sum_{p,q}(S_F^{-1}AS_E)_{q,p}$, has Heisenberg order at most $k+q-p$.
Consider the Heisenberg principal symbols of these components, 
$$
\sigma_x^{k+q-p}((S_F^{-1}AS_E)_{q,p})\colon C^\infty(\mathcal T_xM,\gr_p(E_x))\to C^\infty(\mathcal T_xM,\gr_q(F_x)),
$$ 
see Section~\ref{SS:DO}, and define the \emph{graded Heisenberg principal symbol}
$$
\tilde\sigma_x^k(A)\colon C^\infty(\mathcal T_xM,\gr(E_x))\to C^\infty(\mathcal T_xM,\gr(F_x))
$$
by
$$
\tilde\sigma^k_x(A):=\sum_{p,q}\sigma^{k+q-p}_x\bigl((S_F^{-1}AS_E)_{q,p}\bigr).
$$
This is a left invariant differential operator which is \emph{homogeneous of degree $k$ in the graded sense}, that is, $\tilde\sigma^k_x(A)\circ l_g^*=l_g^*\circ\tilde\sigma_x^k(A)$ and
$$
\tilde\sigma^k_x(A)\circ\delta^{E_x}_\lambda\circ\delta_{\lambda,x}^*
=\lambda^k\cdot\delta^{F_x}_\lambda\circ\delta_{\lambda,x}^*\circ\tilde\sigma^k_x(A)
$$
for all $g\in\mathcal T_xM$ and $\lambda>0$, see \eqref{E:skAlg}.
Here $\delta^{E_x}_\lambda\in\Aut(\gr(E_x))$ denotes the isomorphism given by multiplication with $\lambda^p$ on the component $\gr_p(E_x)$.
Equivalently, the graded Heisenberg principal symbol at $x$ can be considered as an element of the degree $-k$ component of $\mathcal U(\mathfrak t_xM)\otimes\hom(\gr(E_x),\gr(F_x))$, that is,
$$
\tilde\sigma^k_x(A)\in
\bigl(\mathcal U(\mathfrak t_xM)\otimes\hom(\gr(E_x),\gr(F_x))\bigr)_{-k}
:=
\bigoplus_{p,q}\mathcal U_{-k+q-p}(\mathfrak t_xM)\otimes\hom(\gr_p(E_x),\gr_q(F_x)).
$$
Note that any two splittings of a filtered vector bundle differ (multiplicatively) by a filtration-preserving vector bundle isomorphism inducing the identity on the associated graded.
Thus, the filtration of differential operators by graded Heisenberg order does not depend on the choice of splittings $S_E$ and $S_F$, and the graded principal Heisenberg symbol is independent of this choice too.
We obtain a short exact sequence
$$
0\to\tDO^{k-1}(E,F)\to\tDO^k(E,F)\xrightarrow{\tilde\sigma^k}\Gamma^\infty\Bigl(\bigl(\mathcal U(\mathfrak tM)\otimes\hom(\gr(E),\gr(F))\bigr)_{-k}\Bigr)\to0
$$
where $\tDO^k(E,F)$ denotes the space of differential operators in $\DO(E,F)$ which are of graded Heisenberg order at most $k$.

If $G$ is another filtered vector bundle, and $B\in\DO(F,G)$ is a differential operator of graded Heisenberg order at most $l$, then the composition $BA\in\DO(E,G)$ has graded Heisenberg order at most $l+k$ and 
\begin{equation}\label{E:tsAB}
\tilde\sigma^{l+k}_x(BA)=\tilde\sigma^l_x(B)\tilde\sigma_x^k(A).
\end{equation}
This follows readily from \eqref{E:sABAt}.
Moreover, the transposed operator $A^t\in\DO(F',E')$ is of graded Heisenberg order at most $k$ and
\begin{equation}\label{E:grsAt}
\tilde\sigma^k_x(A^t)=\tilde\sigma_x^k(A)^t.
\end{equation}
Here the filtration on the bundle $E'=E^*\otimes|\Lambda|_M=\hom(E,|\Lambda|_M)$ is defined such that a section of $E'$ is in filtration degree $p$ iff it pairs $E^q$ into the $(p+q)$-th filtration subspace of the trivially filtered line bundle $|\Lambda|_M=|\Lambda|_M^0\supseteq|\Lambda|_M^1=0$, i.e.\ iff it vanishes on $E^{-p+1}$.
Note that the canonical isomorphism $\gr_p(E')=\hom(E^{-p}/E^{-p+1},|\Lambda|_M)=\gr_{-p}(E)'$ provides a canonical isomorphism of filtered vector bundles $\gr(E')=\gr(E)'$.
\footnote{The degree $p$ filtration subbundle of $\gr(E)$ is $\gr(E)^p=\bigoplus_{q\geq p}\gr_q(E)$ and thus the filtration on $\gr(E)'$ canonically identifies to $(\gr(E)')^p=\bigoplus_{q\leq-p}\gr_q(E)'=\bigoplus_{q\leq-p}\gr_{-q}(E')=\gr(E')^p$.}
If $S_E\colon\gr(E)\to E$ is a splitting of the filtration on $E$, then $S_E^t\colon E'\to\gr(E)'=\gr(E')$ is filtration-preserving and $S_{E'}:=(S_E^t)^{-1}\colon\gr(E')\to E'$ is a splitting of the filtration on $E'$.
Moreover, $(S_{E'}^{-1}A^tS_{F'})_{p,q}=((S_F^{-1}AS_E)_{-q,-p})^t=((S_F^{-1}AS_E)^t)_{p,q}$ and thus \eqref{E:grsAt} follows at once from \eqref{E:sABAt}.

\begin{remark}\label{R:grsigma}
Every differential operator $A\in\tDO^k(E,F)$ of graded Heisenberg order at most $k$ maps $\Gamma^\infty(E^p)$ into $\Gamma^\infty(F^{p-k})$. 
Indeed, $(S_F^{-1}AS_E)_{q,p}=0$ for $k+q-p<0$ since there are no non-trivial differential operators of negative order.
Moreover, the induced operator $\gr_k(A)\colon\Gamma^\infty(\gr_p(E))\to\Gamma^\infty(\gr_{p-k}(F))$ is tensorial and the corresponding vector bundle homomorphism $\gr_k(A)\colon\gr_p(E)\to\gr_{p-k}(F)$ coincides with the corresponding component of the graded principal Heisenberg symbol, $(\tilde\sigma^k(A))_{p-k,p}$.
If $A$ is of graded Heisenberg order at most $0$, then the associated graded will also be denoted by $\tilde A:=\gr(A):=\gr_0(A)$.
\end{remark}

Generalizing Definition~\ref{def.Hypo-seq} we have:

\begin{definition}[Graded Rockland sequences of differential operators]\label{D:graded_hypoelliptic_seq}
Let $E_i$ be filtered vector bundles over a filtered manifold $M$.
A sequence of differential operators
$$
\cdots\to\Gamma^\infty(E_{i-1})\xrightarrow{A_{i-1}}\Gamma^\infty(E_i)\xrightarrow{A_i}\Gamma^\infty(E_{i+1})\to\cdots
$$ 
which are of graded Heisenberg order at most $k_i$, respectively, is said to be \emph{graded Rockland sequence} if, for each $x\in M$ and every non-trivial irreducible unitary representation $\pi\colon\mathcal T_xM\to U(\mathcal H)$ of the osculating group $\mathcal T_xM$ the sequence
$$
\cdots\to
\mathcal H_\infty\otimes\gr(E_{i-1,x})\xrightarrow{\pi(\tilde\sigma^{k_{i-1}}_x(A_{i-1}))}
\mathcal H_\infty\otimes\gr(E_{i,x})\xrightarrow{\pi(\tilde\sigma^{k_i}_x(A_i))}
\mathcal H_\infty\otimes\gr(E_{i+1,x})\to\cdots
$$
is weakly exact, i.e., the image of each arrow is contained and dense in the kernel of the subsequent arrow.
Here $\mathcal H_\infty$ denotes the subspace of smooth vectors in the Hilbert space $\mathcal H$.
\end{definition}

\subsection{Differential projectors}\label{SS:DP}

In this section we consider a sequence of differential operators acting between sections of filtered vector bundles over a filtered manifold.
Following the BGG machinery \cite{CSS01,CD01,CS15} we will present a construction which permits us to extract (split off) sequences of differential operators acting between sections of certain subbundles.
If the original sequence was graded Rockland, then so is the new one, see Proposition~\ref{P:DB}(b) below.
For the de Rham complex a similar construction can be found in \cite{R99,R01}.

The construction is based on the following simple fact which will be used repeatedly below.
A similar argument can be found in the construction of curved BGG sequences, see \cite{CSS01} or \cite[Theorem~5.2]{CD01}, and in Rumin's work, see \cite[Lemma~2.5]{R01} or \cite[Lemma~1]{R99}.

\begin{lemma}\label{L:inv}
Let $E$ and $F$ be filtered vector bundles over a filtered manifold $M$.
Suppose $A\colon\Gamma^\infty(E)\to\Gamma^\infty(F)$ is a differential operator of graded Heisenberg order at most zero and suppose the induced vector bundle homomorphism, $\tilde A\colon\gr(E)\to\gr(F)$, is invertible, cf.\ Remark~\ref{R:grsigma}.
Then $A$ is invertible and its inverse, $A^{-1}\colon\Gamma^\infty(F)\to\Gamma^\infty(E)$, is a differential operator of graded Heisenberg order at most zero.
\end{lemma}

\begin{proof}
Choose splittings of the filtrations $S_E\colon\gr(E)\to E$ and $S_F\colon\gr(F)\to F$.
Then the differential operator 
$$
\id-AS_E\tilde A^{-1}S_F^{-1}\colon\Gamma^\infty(F)\to\Gamma^\infty(F)
$$ 
is nilpotent for it induces zero on the associated graded.
Hence, the inverse of $A$ can be expressed using the finite Neumann series,
\begin{equation}\label{E:invneumann}
A^{-1}=S_E\tilde A^{-1}S_F^{-1}\sum_{n=0}^\infty\Bigl(\id-AS_E\tilde A^{-1}S_F^{-1}\Bigr)^n.
\end{equation}
In view of this formula, $A^{-1}$ has graded Heisenberg order at most zero.
\end{proof}

\begin{lemma}\label{L:EP}
Consider a filtered vector bundle $E$ over a filtered manifold $M$.
Suppose $\Box\colon\Gamma^\infty(E)\to\Gamma^\infty(E)$ is a differential operator of graded Heisenberg order at most zero, and let $\tilde\Box\colon\gr(E)\to\gr(E)$ denote the associated graded vector bundle endomorphism, see Remark~\ref{R:grsigma}.
For each $x\in M$ let $\tilde P_x\in\eend(\gr(E_x))$ denote the spectral projection onto the generalized zero eigenspace of\/ $\tilde\Box_x\in\eend(\gr(E_x))$.
Assume that the rank of $\tilde P_x$ is locally constant in $x\in M$.

\begin{enumerate}[(a)]
\item 
Then $\tilde P\colon\gr(E)\to\gr(E)$ is a smooth vector bundle homomorphism,
$\tilde P^2=\tilde P$, $\tilde P\tilde\Box=\tilde\Box\tilde P$, and we obtain a decomposition of graded vector bundles,
\begin{equation}\label{E:tpdeco}
\gr(E)=\img(\tilde P)\oplus\ker(\tilde P),
\end{equation}
invariant under $\tilde\Box$, and such that $\tilde\Box$ is nilpotent on $\img(\tilde P)$ and invertible on $\ker(\tilde P)$.

\item
There exists a unique filtration-preserving differential operator $P\colon\Gamma^\infty(E)\to\Gamma^\infty(E)$ such that $P^2=P$, $P\Box=\Box P$, and $\gr(P)=\tilde P$.
This operator $P$ has graded Heisenberg order at most zero and provides a decomposition of filtered vector spaces,
\begin{equation}\label{E:Pdeco}
\Gamma^\infty(E)=\img(P)\oplus\ker(P),
\end{equation}
which is invariant under $\Box$ and such that $\Box$ is nilpotent on $\img(P):=P(\Gamma^\infty(E))$ and invertible on $\ker(P):=\{\psi\in\Gamma^\infty(E):\Box\psi=0\}$.

\item
Let $S\colon\gr(E)\to E$ be a splitting of the filtration.
Then
\begin{equation}\label{E:tL}
L\colon\Gamma^\infty(\gr(E))\to\Gamma^\infty(E),\qquad
L:=PS\tilde P+(\id-P)S(\id-\tilde P),
\end{equation}
is an invertible differential operator of graded Heisenberg order at most zero such that $\gr(L)=\id$ and $L^{-1}PL=\tilde P$.
Hence, $L$ induces filtration-preserving isomorphisms
$$
L\colon\Gamma^\infty(\img(\tilde P))\xrightarrow\cong\img(P)
\qquad\text{and}\qquad
L\colon\Gamma^\infty(\ker(\tilde P))\xrightarrow\cong\ker(P).
$$
Moreover, $L^{-1}\Box L$ is a differential operator of graded Heisenberg order at most zero satisfying $\gr(L^{-1}\Box L)=\tilde\Box$.
Furthermore, $L^{-1}\Box L$ preserves the decomposition \eqref{E:tpdeco}, its restriction to $\Gamma^\infty(\img(\tilde P))$ is nilpotent, and its restriction to $\Gamma^\infty(\ker(\tilde P))$ is invertible.
\end{enumerate}
\end{lemma}

\begin{proof}
Part (a) is well known.
To prove the existence of an operator $P$ as in (b) we assume, for a moment, that there exists $\varepsilon>0$ such that all non-trivial eigenvalues of $\tilde\Box_x$ lie outside the disk of radius $2\varepsilon$ centered at the origin in the complex plane for all $x\in M$.
Hence, the vector bundle homomorphism $z-\tilde\Box$ is invertible, for all $0<|z|<2\varepsilon$.
According to Lemma~\ref{L:inv}, $(z-\Box)^{-1}$ is a family of differential operators of graded Heisenberg order at most zero depending rationally on $z$ for $|z|<2\varepsilon$.
Hence,
\begin{equation}\label{E:P}
P:=\frac1{2\pi\ii}\oint_{|z|=\varepsilon}(z-\Box)^{-1}dz,
\end{equation}
defines a differential operator of graded Heisenberg order at most zero.
Clearly, $P^2=P$ and $\Box P=P\Box$.
Moreover, $P$ is filtration-preserving and $\gr(P)=\tilde P$ in view of $\tilde P_x=\frac1{2\pi\mathbf i}\oint_{|z|=\varepsilon}(z-\tilde \Box_x)^{-1}dz$.
In particular, $P$ gives rise to a decomposition of filtered vector spaces which is invariant under $\Box$ as indicated in \eqref{E:Pdeco}.
In general, the $\varepsilon$ used above, will only exist locally.
However, since the operator in \eqref{E:P} does not depend on the choice of $\varepsilon$, these locally defined differential operators match up and give rise to a globally defined differential operator $P$ with said properties.
Using $P^2=P$ and $\tilde P^2=\tilde P$, we immediately obtain $PL=L\tilde P$.
In view of $\gr(P)=\tilde P$ and $\gr(S)=\id$, we have $\gr(L)=\id$, hence $L$ is invertible according to Lemma~\ref{L:inv} and $L^{-1}$ is a differential operator of graded Heisenberg order at most zero.
Since $\Box$ preserves the decomposition \eqref{E:Pdeco}, $L^{-1}\Box L$ preserves the decomposition \eqref{E:tpdeco}.
From $\gr(L)=\id$, we get $\gr(L^{-1}\Box L)=\tilde\Box$.
Using the latter it is easy to see that $\Box$ is nilpotent on $\img(P)$ and invertible on $\ker(P)$.
Indeed, since $\tilde\Box$ is nilpotent on $\img(\tilde P)$, we conclude that $L^{-1}\Box L$ is nilpotent on $\Gamma^\infty(\img(\tilde P))$, and thus $\Box$ is nilpotent on $\img(P)$.
Furthermore, since $\tilde\Box$ is invertible on $\ker(\tilde P)$, Lemma~\ref{L:inv} implies that $L^{-1}\Box L$ is invertible on $\Gamma^\infty(\ker(\tilde P))$, and thus $\Box$ is invertible on $\ker(P)$.
Using the latter property one readily checks the uniqueness assertion in (b).
\end{proof}

After these preparations let us now turn to the construction of the sequences mentioned at the beginning of this section.
Let $E_i$ be filtered vector bundles over a filtered manifold $M$, and consider a sequence of differential operators,
\begin{equation}\label{E:EAE}
\cdots\to\Gamma^\infty(E_{i-1})\xrightarrow{A_{i-1}}\Gamma^\infty(E_i)\xrightarrow{A_i}\Gamma^\infty(E_{i+1})\to\cdots,
\end{equation}
such that $A_i$ is of graded Heisenberg order at most $k_i$.
Suppose
\begin{equation}\label{E:EdelE}
\cdots\leftarrow\Gamma^\infty(E_{i-1})\xleftarrow{\delta_i}\Gamma^\infty(E_i)\xleftarrow{\delta_{i+1}}\Gamma^\infty(E_{i+1})\leftarrow\cdots
\end{equation}
is a sequence of differential operators such that $\delta_i$ is of graded Heisenberg order at most $-k_{i-1}$.
Then the differential operators
\begin{equation}\label{E:EboxE}
\Box_i\colon\Gamma^\infty(E_i)\to\Gamma^\infty(E_i),\qquad
\Box_i:=A_{i-1}\delta_i+\delta_{i+1}A_i,
\end{equation}
are of graded Heisenberg order at most zero, see Section~\ref{SS:fVBDO}.
\footnote{In subsequent sections we will restrict our attention to the case when $\delta_i$ are (tensorial) vector bundle homomorphisms satisfying $\delta_i\delta_{i+1}=0$, but these restrictions would not be helpful here.}

We let $\tilde\Box_i\colon\gr(E_i)\to\gr(E_i)$, $\tilde\Box_i:=\gr(\Box_i)$, denote the associated graded vector bundle homomorphism, see Remark~\ref{R:grsigma}.
Moreover, for each $x\in M$, we let $\tilde P_{x,i}\colon\gr(E_{x,i})\to \gr(E_{x,i})$ denote the spectral projection onto the generalized zero eigenspace of $\tilde\Box_{x,i}\in\eend(\gr(E_{x,i}))$.
Note that the projectors $\tilde P_{x,i}$ preserve the grading on $\gr(E_{x,i})$.
We assume that the rank of $\tilde P_{x,i}$ is locally constant in $x$. 
Consequently, these fiber-wise projectors provide a smooth vector bundle projector $\tilde P_i\colon\gr(E_i)\to\gr(E_i)$ and we obtain a decomposition of graded vector bundles, 
\begin{equation}\label{E:decogrE}
\gr(E_i)=\img(\tilde P_i)\oplus\ker(\tilde P_i),
\end{equation}
which is invariant under $\tilde\Box_i$, see Lemma~\ref{L:EP}(a).
By construction, $\tilde\Box_i$ is nilpotent on $\img(\tilde P_i)$ and invertible on $\ker(\tilde P_i)$.
We let $\tilde A_i\colon\gr_*(E_i)\to\gr_{*-k_i}(E_{i+1})$ and $\tilde\delta_i\colon\gr_*(E_i)\to\gr_{*+k_{i-1}}(E_{i-1})$ denote the associated graded vector bundle homomorphisms of $A_i$ and $\delta_i$, respectively, see Remark~\ref{R:grsigma}.
Clearly, $\tilde\Box_i=\tilde A_{i-1}\tilde\delta_i+\tilde\delta_{i+1}\tilde A_i$, see \eqref{E:EboxE}.
If $\tilde A_i\tilde A_{i-1}=0$ for all $i$, then $\tilde\Box_{i+1}\tilde A_i=\tilde A_i\tilde\Box_i$, $\tilde P_{i+1}\tilde A_i=\tilde A_i\tilde P_i$, and, thus, $\tilde A_i$ preserves the decompositions \eqref{E:decogrE}.
Similarly, if $\tilde\delta_{i-1}\tilde\delta_i=0$ for all $i$, then $\tilde\Box_{i-1}\tilde\delta_i=\tilde\delta_i\tilde\Box_i$, $\tilde P_{i-1}\tilde\delta_i=\tilde\delta_i\tilde P_i$, and $\tilde\delta_i$ preserves the decompositions in \eqref{E:decogrE}.

According to Lemma~\ref{L:EP}(b), there exists a unique filtration-preserving differential operator 
$$
P_i\colon\Gamma^\infty(E_i)\to\Gamma^\infty(E_i)
$$ 
such that $P_i^2=P_i$, $P_i\Box_i=\Box_iP_i$, and $\gr(P_i)=\tilde P_i$.
These projectors are of graded Heisenberg order at most zero and provide a decomposition $\Gamma^\infty(E_i)=\img(P_i)\oplus\ker(P_i)$ such that $\Box_i$ is nilpotent on $\img(P_i)$ and invertible on $\ker(P_i)$.

We fix splittings for the filtrations, $S_i\colon\gr(E_i)\to E_i$, and consider the differential operators 
$$
L_i\colon\Gamma^\infty(\gr(E_i))\to\Gamma^\infty(E_i),\qquad
L_i:=P_iS_i\tilde P_i+(\id-P_i)S_i(\id-\tilde P_i).
$$
In view of Lemma~\ref{L:EP}(c), $L_i$ is an invertible differential operator of graded Heisenberg order at most zero such that $\gr(L_i)=\id$ and $L_i^{-1}P_iL_i=\tilde P_i$.
Moreover, $L^{-1}_i\Box_iL_i$ is a differential operator of graded Heisenberg order at most zero such that $\gr(L_i^{-1}\Box_iL_i)=\tilde\Box_i$.
Furthermore, $L^{-1}_i\Box_iL_i$ preserves the decomposition \eqref{E:decogrE}, its restriction to $\Gamma^\infty(\img(\tilde P_i))$ is nilpotent and its restriction to $\Gamma^\infty(\ker(\tilde P_i))$ is invertible.

Conjugating the original sequence \eqref{E:EAE} by $L_i$, we obtain two sequences,
\begin{equation}\label{E:EDE}
\cdots\to\Gamma^\infty(\img(\tilde P_{i-1}))\xrightarrow{D_{i-1}}\Gamma^\infty(\img(\tilde P_i))\xrightarrow{D_i}\Gamma^\infty(\img(\tilde P_{i+1}))\to\cdots
\end{equation}
and
\begin{equation}\label{E:EBE}
\cdots\to\Gamma^\infty(\ker(\tilde P_{i-1}))\xrightarrow{B_{i-1}}\Gamma^\infty(\ker(\tilde P_i))\xrightarrow{B_i}\Gamma^\infty(\ker(\tilde P_{i+1}))\to\cdots,
\end{equation}
where 
\begin{equation}\label{E:defDB}
D_i:=\tilde P_{i+1}L_{i+1}^{-1}A_iL_i|_{\Gamma^\infty(\img(\tilde P_i))}
\qquad\text{and}\qquad
B_i:=(\id-\tilde P_{i+1})L_{i+1}^{-1}A_iL_i|_{\Gamma^\infty(\ker(\tilde P_i))}
\end{equation}
are differential operators of graded Heisenberg order at most $k_i$.
This generalizes a construction for the de Rham complex due to Rumin, see \cite[Theorem~2.6]{R01} and \cite[Theorem~1]{R99}.

\begin{proposition}\label{P:DB}
In this situation the following hold true:
\begin{enumerate}[(a)]
\item
If $\tilde\sigma^{k_i}_x(A_i)\tilde\sigma^{k_{i-1}}_x(A_{i-1})=0$ for all $i$ and $x$, then 
$$
\tilde\sigma^{k_i}_x(L^{-1}_{i+1}A_iL_i)=\tilde\sigma^{k_i}_x(D_i)\oplus\tilde\sigma^{k_i}_x(B_i).
$$

\item 
If the sequence \eqref{E:EAE} is graded Rockland, then so are the sequences \eqref{E:EDE} and \eqref{E:EBE}.

\item
If $A_iA_{i-1}=0$ for all $i$, then the operator $L_{i+1}^{-1}A_iL_i$ decouples,
$$
L^{-1}_{i+1}A_iL_i=D_i\oplus B_i,
$$ 
and we have $D_iD_{i-1}=0$, $B_iB_{i-1}=0$.
In this situation, $G_i\colon\Gamma^\infty(\ker(\tilde P_i))\to\Gamma^\infty(\ker(\tilde P_i))$,
$$
G_i:=B_{i-1}(\id-\tilde P_{i-1})\tilde\delta_i\tilde\Box_i^{-1}
+(\id-\tilde P_i)\tilde\delta_{i+1}\tilde\Box_{i+1}^{-1}\tilde A_i,
$$ 
is an invertible differential operator of graded Heisenberg order at most zero with $\gr(G_i)=\id$ that conjugates the complex \eqref{E:EBE} into an acyclic tensorial complex, namely, 
$$
G_{i+1}^{-1}B_iG_i=\tilde A_i|_{\Gamma^\infty(\ker(\tilde P_i))}.
$$
Moreover, the restriction $L_i\colon\Gamma^\infty(\img(\tilde P_i))\to\Gamma^\infty(E_i)$ provides a chain map, that is, $A_iL_i|_{\Gamma^\infty(\img(\tilde P_i))}=L_{i+1}D_i$, which induces an isomorphism between the cohomologies of \eqref{E:EDE} and \eqref{E:EAE}.
More precisely, $\pi_i:=\tilde P_iL_i^{-1}\colon\Gamma^\infty(E_i)\to\Gamma^\infty(\img(\tilde P_i))$, is a chain map, $D_i\pi_i=\pi_{i+1}A_i$, which is an inverse of $L_i$ up to homotopy, i.e., $\pi_iL_i|_{\Gamma^\infty(\img(\tilde P_i))}=\id$ and\/ $\id-L_i\pi_i=A_{i-1}h_i+h_{i+1}A_i$ where $h_i\colon\Gamma^\infty(E_i)\to\Gamma^\infty(E_{i-1})$ is a differential operator of graded Heisenberg order at most $-k_{i-1}$ given by $h_i:=L_{i-1}G_{i-1}(\id-\tilde P_{i-1})\tilde\delta_i\tilde\Box_i^{-1}G_i^{-1}(\id-\tilde P_i)L_i^{-1}$.
\end{enumerate}
\end{proposition}

\begin{proof}
To show part (c) suppose $A_iA_{i-1}=0$ for all $i$.
Then $\Box_{i+1}A_i=A_i\Box_i$, see \eqref{E:EboxE}, and thus $P_{i+1}A_i=A_iP_i$, see \eqref{E:P}.
Using the relation $L_i^{-1}P_iL_i=\tilde P_i$ from Lemma~\ref{L:EP}(c), we obtain $\tilde P_{i+1}(L_{i+1}^{-1}A_iL_i)=(L_{i+1}^{-1}A_iL_i)\tilde P_i$, hence $\tilde P_{i+1}(L_{i+1}^{-1}A_iL_i)(\id-\tilde P_i)=0=(\id-\tilde P_{i+1})(L_{i+1}^{-1}A_iL_i)\tilde P_i$,
$$
L^{-1}_{i+1}A_iL_i=\tilde P_{i+1}(L^{-1}_{i+1}A_iL_i)\tilde P_i
+(\id-\tilde P_{i+1})(L^{-1}_{i+1}A_iL_i)(\id-\tilde P_i),
$$ 
and thus $L^{-1}_{i+1}A_iL=D_i\oplus B_i$, cf.\ \eqref{E:defDB}.
Clearly, $A_iA_{i-1}=0$ implies $D_iD_{i-1}=0$, $B_iB_{i-1}=0$, and $B_iG_i=G_{i+1}\tilde A_i$.
Clearly, $\gr(G_i)=\id$ and, thus, $G_i$ is invertible according to Lemma~\ref{L:inv}.
The remaining assertions in (c) are now straight forward.

To show part (a) suppose $\tilde\sigma^{k_i}_x(A_i)\tilde\sigma^{k_{i-1}}_x(A_{i-1})=0$ for all $i$ and all $x\in M$.
Using the multiplicativity of the graded Heisenberg principal symbol, see \eqref{E:tsAB}, we conclude $\tilde\sigma^0_x(\Box_{i+1})\tilde\sigma^{k_i}_x(A_i)=\tilde\sigma_x^{k_i}(A_i)\tilde\sigma^0_x(\Box_i)$, see \eqref{E:EboxE}, and thus $\tilde\sigma^0_x(P_{i+1})\tilde\sigma^{k_i}_x(A_i)=\tilde\sigma^{k_i}_x(A_i)\tilde\sigma^0_x(P_i)$, see \eqref{E:P}.
Combining this with $L_i^{-1}P_iL_i=\tilde P_i$ from Lemma~\ref{L:EP}(c), we obtain $\tilde\sigma_x^{k_i}\bigl(\tilde P_{i+1}(L_{i+1}^{-1}A_iL_i)\bigr)=\tilde\sigma^{k_i}_x\bigl((L_{i+1}^{-1}A_iL_i)\tilde P_i\bigr)$, hence 
$\tilde\sigma^{k_i}_x\bigl(\tilde P_{i+1}(L_{i+1}^{-1}A_iL_i)(\id-\tilde P_i)\bigr)=0=\tilde\sigma^{k_i}_x\bigl((\id-\tilde P_{i+1})(L_{i+1}^{-1}A_iL_i)\tilde P_i\bigr)$,
$$
\tilde\sigma^{k_i}_x(L^{-1}_{i+1}A_iL_i)
=\tilde\sigma^{k_i}_x\bigl(\tilde P_{i+1}(L^{-1}_{i+1}A_iL_i)\tilde P_i\bigr)
+\tilde\sigma^{k_i}_x\bigl((\id-\tilde P_{i+1})(L^{-1}_{i+1}A_iL_i)(\id-\tilde P_i)\bigr),
$$ 
and thus $\tilde\sigma^{k_i}_x(L^{-1}_{i+1}A_iL_i)=\tilde\sigma^{k_i}_x(D_i)\oplus\tilde\sigma^{k_i}_x(B_i)$.

To see (b) suppose the sequence \eqref{E:EAE} is graded Rockland, that is, the graded Heisenberg principal symbol sequences $\sigma^{k_i}_x(A_i)$ are weakly exact in every non-trivial irreducible unitary representation of $\mathcal T_xM$.
Clearly, the conjugated sequence $\tilde\sigma^{k_i}_x(L_{i+1}^{-1}A_iL_i)=\tilde\sigma^0_x(L_{i+1})^{-1}\tilde\sigma^{k_i}_x(A_i)\tilde\sigma^0_x(L_i)$ has the same property.
Note that $\tilde\sigma^{k_i}_x(A_i)\tilde\sigma^{k_{i-1}}_x(A_{i-1})=0$ since this relation holds true in every non-trivial irreducible unitary representation of $\mathcal T_xM$.
Hence part (a), permits us to conclude that the graded Heisenberg principal symbol sequences $\tilde\sigma^{k_i}_x(D_i)$ and $\tilde\sigma^{k_i}_x(B_i)$ are weakly exact in every non-trivial irreducible unitary representation of $\mathcal T_xM$ too.
\end{proof}

\subsection{Splitting operators}\label{SS:splitting-operators}

The sequences \eqref{E:EDE} and \eqref{E:EBE} constructed above depend on the operators $A_i$ and $\delta_i$, see \eqref{E:EAE} and \eqref{E:EdelE} but also on the splittings $S_i\colon\gr(E_i)\to E_i$. 
We will now specialize to a situation in which one can construct a variant of the sequence \eqref{E:EDE} which does not depend on the splittings $S_i$, but only on $A_i$ and $\delta_i$.
The sequences obtain in this way, see Proposition~\ref{P:D}, generalize the curved BGG sequences \cite{CSS01,CD01} discussed in Section~\ref{SS:BGG}.

We continue to use the notation from the preceding sections.
In particular, $E_i$ are filtered vector bundles over a filtered manifold $M$, and we consider a sequence,
\begin{equation}\label{E:EAE2}
\cdots\to\Gamma^\infty(E_{i-1})\xrightarrow{A_{i-1}}\Gamma^\infty(E_i)\xrightarrow{A_i}\Gamma^\infty(E_{i+1})\to\cdots,
\end{equation}
where $A_i$ is a differential operator of graded Heisenberg order at most $k_i$.

\begin{definition}[Codifferential of Kostant type]\label{D:Kdelta}
A sequence of vector bundle homomorphisms,
\begin{equation}\label{E:admEdelE}
\cdots\leftarrow E_{i-1}\xleftarrow{\delta_i}E_i\xleftarrow{\delta_{i+1}}E_{i+1}\leftarrow\cdots,
\end{equation}
will be called a \emph{codifferential of Kostant type} for the sequence \eqref{E:EAE2} if it has the following properties:
\begin{enumerate}[(i)]
\item
$\delta_i\delta_{i+1}=0$ for all $i$.
\item
$\delta_i$ maps the filtration space $E_i^p$ into $E_{i-1}^{p+k_{i-1}}$.\footnote{In particular, $\delta_i$ is of graded Heisenberg order at most $-{k_{i-1}}$.}
\item
There exist splittings of the filtrations, $S_i\colon\gr(E_i)\to E_i$, such that $\tilde\delta_i=S_{i-1}^{-1}\delta_iS_i$.
\footnote{Only the existence of such splittings is required, $S_i$ is not part of the data.}
\item
$\tilde\delta_{i,x}\tilde P_{i,x}=0$ for each $x\in M$.
\end{enumerate}
Here $\tilde A_i\colon\gr_*(E_i)\to\gr_{*-k_i}(E_{i+1})$, 
$\tilde\delta_i\colon\gr_*(E_i)\to\gr_{*+k_{i-1}}(E_{i-1})$, and
$\tilde\Box_i\colon\gr_*(E_i)\to\gr_*(E_i)$
denote the vector bundle homomorphisms induced by $A_i$, $\delta_i$ and $\Box_i=A_{i-1}\delta_i+\delta_{i+1}A_i$ on the associated graded vector bundles, respectively, see Remark~\ref{R:grsigma}.
Moreover, $\tilde P_{i,x}\colon E_{i,x}\to E_{i,x}$ denotes the spectral projection onto the generalized zero eigenspace of $\tilde\Box_{i,x}$.
\end{definition}

\begin{lemma}\label{L:L}
Consider a sequence of vector bundle homomorphisms $\delta_i$ as in \eqref{E:admEdelE} which is a codifferential of Kostant type for the sequence \eqref{E:EAE2}, see Definition~\ref{D:Kdelta}, and assume that the rank of $\delta_{x,i}$ is locally constant in $x$, for each $i$.
Then the following hold true:

\begin{enumerate}[(a)]
\item
The rank of $\tilde P_{i,x}$ is locally constant in $x$, these families provide smooth vector bundle projectors, $\tilde P_i\colon\gr(E_i)\to\gr(E_i)$, such that $\tilde\delta_i\tilde\delta_{i+1}=0$, $\tilde\Box_{i-1}\tilde\delta_i=\tilde\delta_i\tilde\Box_i$, $\tilde P_i\tilde\Box_i=\tilde\Box_i\tilde P_i$, $\tilde P_{i-1}\tilde\delta_i=\tilde\delta_i\tilde P_i=0$, and we have a decomposition of graded  vector bundles,
\begin{equation}\label{E:kertddeco}
\ker(\tilde\delta_i)=\img(\tilde P_i)\oplus\img(\tilde\delta_{i+1}).
\end{equation}

\item
Let $P_i\colon\Gamma^\infty(E_i)\to\Gamma^\infty(E_i)$ denote the unique filtration-preserving differential operator such that $P_i^2=P_i$, $P_i\Box_i=\Box_iP_i$ and $\gr(P_i)=\tilde P_i$, see Lemma~\ref{L:EP}(b).
Then $\Box_{i-1}\delta_i=\delta_i\Box_i$, $P_{i-1}\delta_i=\delta_iP_i=0$, and we obtain a decomposition of filtered spaces,
\begin{equation}\label{E:kertdeco}
\Gamma^\infty(\ker(\delta_i))=\img(P_i)\oplus\Gamma^\infty(\img(\delta_{i+1})).
\end{equation}
Moreover, the quotient bundle $\mathcal H_i:=\ker(\delta_i)/\img(\delta_{i+1})$ is a filtered vector bundle and $P_i$ factors to a differential operator of graded Heisenberg order at most zero,
$$
\bar L_i\colon\Gamma^\infty(\mathcal H_i)\xrightarrow\cong\img(P_i)\subseteq\Gamma^\infty(E_i),
$$
which is inverse to the restriction of the canonical projection $\bar\pi_i\colon\ker(\delta_i)\to\mathcal H_i$, that is, $\bar\pi_i\bar L_i=\id$ and $\bar L_i\bar\pi_i=P_i$.

\item
Assume, moreover, $\tilde A_{i+1}\tilde A_i=0$ for all $i$.
Then $\tilde\Box_{i+1}\tilde A_i=\tilde A_i\tilde\Box_i$, $\tilde P_{i+1}\tilde A_i=\tilde A_i\tilde P_i$, $\tilde\delta_{i+1}\tilde A_i\tilde P_i=0$, and 
\begin{equation}\label{E:kertbox}
\img(\tilde P_i)=\ker(\tilde\Box_i)
=\ker(\tilde\delta_i)\cap\ker(\tilde\delta_{i+1}\tilde A_i).
\end{equation}
Moreover, $\delta_{i+1}A_iP_i=0$, and 
\begin{equation}\label{E:kerbox}
\img(P_i)=\ker(\Box_i)=\ker(\delta_i)\cap\ker(\delta_{i+1}A_i).
\end{equation}
\end{enumerate}
\end{lemma}

\begin{proof}
In view of $\tilde\delta_{i-1,x}\tilde\delta_{i,x}=0$ and $\tilde\delta_{i,x}\tilde P_{i,x}=0$ we also have $\tilde\Box_{i-1,x}\tilde\delta_{i,x}=\tilde\delta_{i,x}\tilde\Box_{i,x}$ and $\tilde P_{i-1,x}\tilde\delta_{i,x}=\tilde\delta_{i,x}\tilde P_{i,x}=0$.
Moreover, since $\tilde A_{i-1,x}\tilde\delta_{i,x}+\tilde\delta_{i+1,x}\tilde A_{i,x}=\tilde\Box_{i,x}$ is invertible on $\ker(\tilde P_{i,x})$, it is straight forward to establish the finite dimensional decomposition $\ker(\tilde\delta_{i,x})=\img(\tilde P_{i,x})\oplus\img(\tilde\delta_{i+1,x})$ for each $x\in M$.
Since $\tilde\delta_{i,x}$ is assumed to have locally constant rank, the same must be true for $\tilde P_{i,x}$.
Hence, $\tilde P_i$ is a smooth vector bundle homomorphism, and \eqref{E:kertddeco} is a decomposition of smooth vector bundles.
The remaining assertions in (a) are obvious.

Clearly, $\delta_i\Box_i=\Box_{i-1}\delta_i$ and $\delta_iP_i=P_{i-1}\delta_i$, see \eqref{E:EboxE} and \eqref{E:P}.
To show $\delta_iP_i=0$, we consider the operator 
\begin{equation}\label{E:Li4}
L_i\colon\Gamma^\infty(\gr(E_i))\to\Gamma^\infty(E_i),\qquad
L_i=P_iS_i\tilde P_i+(\id-P_i)S_i(\id-\tilde P_i),
\end{equation}
cf.~\eqref{E:tL} in Lemma~\ref{L:EP}, where the splittings $S_i$ are as in Definition~\ref{D:Kdelta}(iii).
Then $L_{i-1}^{-1}\delta_iL_i=\tilde\delta_i$ and $L^{-1}_iP_iL_i=\tilde P_i$.
The assertion $\delta_iP_i=0$ thus follows from $\tilde\delta_i\tilde P_i=0$, see Definition~\ref{D:Kdelta}(iv).
Since $A_{i-1}\delta_i+\delta_{i+1}A_i=\Box_i$ is invertible on $\ker(P_i)$, it is now straight forward to establish the decomposition \eqref{E:kertdeco}.
Clearly, $\img(\delta_{i+1})$, $\ker(\delta_i)$ and $\mathcal H_i=\ker(\delta_i)/\img(\delta_{i+1})$ are smooth vector bundles.
Moreover, $P_i$ factorizes to a linear map $\bar L_i\colon\Gamma^\infty(\mathcal H_i)\to\img(P_i)$ such that $\bar L_i\bar\pi_i=P_i$ and $\bar\pi_i\bar L_i=\id$.
To see that $\bar L_i$ is a differential operator of graded Heisenberg order at most zero, it suffices to observe that $\bar L_i$ can be written (locally) as the composition of $P_i$ with a filtration-preserving vector bundle homomorphism $\mathcal H_i\to\ker(\delta_i)$.

If $\tilde A_i\tilde A_{i-1}=0$, then the assertions $\tilde\Box_{i+1}\tilde A_i=\tilde A_i\tilde\Box_i$, $\tilde P_{i+1}\tilde A_i=\tilde A_i\tilde P_i$, and $\tilde\delta_{i+1}\tilde A_i\tilde P_i=0$, as well as the equalities in \eqref{E:kertbox} are now obvious.
Using $\Box_i=\delta_{i+1}A_i+A_{i-1}\delta_i$, $\delta_iP_i=0=P_{i-1}\delta_i$, $\Box_iP_i=P_i\Box_i$ and $P^2_i=P_i$, we obtain $\delta_{i+1}A_iP_i=\Box_iP_i=P_i\Box_iP_i=0$, whence the equalities in \eqref{E:kerbox}.
The remaining assertions follow at once.
\end{proof}

\begin{proposition}\label{P:D}
Let $E_i$ be filtered vector bundles over a filtered manifold $M$, and consider a sequence of differential operators,
\begin{equation}\label{E:EAE3}
\cdots\to\Gamma^\infty(E_{i-1})\xrightarrow{A_{i-1}}\Gamma^\infty(E_i)\xrightarrow{A_i}\Gamma^\infty(E_{i+1})\to\cdots,
\end{equation}
such that $A_i$ is of graded Heisenberg order at most $k_i$ and $\tilde A_i\tilde A_{i-1}=0$ for all $i$.
Moreover, let 
$$
\cdots\leftarrow E_{i-1}\xleftarrow{\delta_i}E_i\xleftarrow{\delta_{i+1}}E_{i+1}\leftarrow\cdots
$$
be a codifferential of Kostant type for the sequence \eqref{E:EAE3}, see Definition~\ref{D:Kdelta}.
Assume that each $\delta_i$ has locally constant rank, and let $\bar\pi_i\colon\ker(\delta_i)\to\mathcal H_i:=\ker(\delta_i)/\img(\delta_{i+1})$ denote the canonical vector bundle projection.
Then the following hold true:

\begin{enumerate}[(a)]
\item
There exists a unique differential operator $\bar L_i\colon\Gamma^\infty(\mathcal H_i)\to\Gamma^\infty(E_i)$ such that $\delta_i\bar L_i=0$, $\bar\pi_i\bar L_i=\id$, and $\delta_{i+1}A_i\bar L_i=0$.
Moreover, $\bar L_i$ is of graded Heisenberg order at most zero, and the differential operator 
\begin{equation}\label{E:defbarD}
\bar D_i\colon\Gamma^\infty(\mathcal H_i)\to\Gamma^\infty(\mathcal H_{i+1}),\qquad\bar D_i:=\bar\pi_{i+1}A_i\bar L_i,
\end{equation}
is of graded Heisenberg order at most $k_i$.

\item
If the sequence \eqref{E:EAE3} is graded Rockland, then so is the sequence:
\begin{equation}\label{E:D}
\cdots\to\Gamma^\infty(\mathcal H_{i-1})\xrightarrow{\bar D_{i-1}}\Gamma^\infty(\mathcal H_i)\xrightarrow{\bar D_i}\Gamma^\infty(\mathcal H_{i+1})\to\cdots.
\end{equation}

\item
If $A_iA_{i-1}=0$, then $\bar D_i\bar D_{i-1}=0$, and $\bar L_i\colon\Gamma^\infty(\mathcal H_i)\to\Gamma^\infty(E_i)$ is a chain map, $A_i\bar L_i=\bar L_{i+1}\bar D_i$, inducing an isomorphism between the cohomologies of \eqref{E:D} and \eqref{E:EAE3}.

\item
There exist invertible differential operators, $V_i\colon\Gamma^\infty(\img(\tilde P_i))\to\Gamma^\infty(\mathcal H_i)$ with inverse $V_i^{-1}\colon\Gamma^\infty(\mathcal H_i)\to\Gamma^\infty(\img(\tilde P_i))$, both of graded Heisenberg order at most zero, such that 
$$
V_{i+1}^{-1}\bar D_iV_i=D_i,
$$
where $D_i\colon\Gamma^\infty(\img(\tilde P_i))\to\Gamma^\infty(\img(\tilde P_{i+1}))$ denotes the operator considered in Section~\ref{SS:DP}, see \eqref{E:defDB}, associated with splittings $S_i$ as in Definition~\ref{D:Kdelta}(iii).
\end{enumerate}
\end{proposition}

\begin{proof}
Part (a) follows immediately from Lemma~\ref{L:L}(b)\&(c).
Recall the differential operator $L_i$ associated with the splittings $S_i$, see \eqref{E:Li4}.
The differential operator $V_i\colon\Gamma^\infty(\img(\tilde P_i))\to\Gamma^\infty(\mathcal H_i)$, $V_i:=\bar\pi_iL_i|_{\img(\tilde P_i)}$ has graded Heisenberg order at most zero and so does its inverse, $V_i^{-1}\colon\Gamma^\infty(\mathcal H_i)\to\Gamma^\infty(\img(\tilde P_i))$, $V_i^{-1}=L_i^{-1}\bar L_i$.
Clearly, $V_{i+1}^{-1}\bar D_iV_i=D_i$, whence (d).
Combining this with Proposition~\ref{P:DB}(b)\&(c), we obtain the statements in (b) in (c).
\end{proof}

The operators $\bar L_i$ in Proposition~\ref{P:D}(a) above are direct generalizations of the well known splitting operators in parabolic geometry, see \cite[Theorem~2.4]{CS12}.
The operators $\bar D_i$ generalize the BGG operators.

\begin{remark}
If the filtration on each $\mathcal H_i$ is trivial, then \eqref{E:D} is a Rockland sequence in the ungraded sense of Definition~\ref{def.Hypo-seq}.
\end{remark}

\begin{remark}
For the Kostant type codifferential $\delta_i=0$ all statements in Proposition~\ref{P:D} are trivially true, $\bar\pi_i=\id=\bar L_i$. 
\end{remark}

\begin{remark}\label{R:Kdeltabound}
If $\tilde A_i\tilde A_{i-1}=0$, then the rank of a Kostant type codifferential is bounded by
\begin{equation}\label{E:Kdeltaesti}
\sum_i\rank(\delta_{i,x})\leq\sum_i\rank(\tilde A_{i,x}).
\end{equation}
Indeed, from $\tilde\delta_i\tilde\delta_{i+1}=0$ and $\tilde\delta_i\tilde P_i=0$ we get $\ker(\tilde\delta_{i,x})=\img(\tilde P_{i,x})\oplus\img(\tilde\delta_{i+1,x})$ and, consequently, $\ker(\delta_{i,x})/\img(\delta_{i+1,x})\cong\ker(\tilde\delta_{i,x})/\img(\tilde\delta_{i+1,x})\cong\img(\tilde P_{i,x})$.
Moreover, $\tilde A_{i,x}$ preserves the decomposition $\gr(E_i)=\img(\tilde P_{i,x})\oplus\ker(\tilde P_{i,x})$ and its restriction to $\ker(\tilde P_{i,x})$ is acyclic since $\tilde\Box_{i,x}$ is invertible on $\ker(\tilde P_{i,x})$.
We conclude 
\begin{equation}\label{E:dimHH}
\dim\bigl(\ker(\tilde A_{i,x})/\img(\tilde A_{i-1,x})\bigr)
\leq\dim\bigl(\ker(\delta_{i,x})/\img(\delta_{i+1,x})\bigr)
\end{equation}
for all $i$.
To show \eqref{E:Kdeltaesti} it thus remains to recall that the equation $\dim(C)=2\rank(d)+\dim(\ker(d)/\img(d))$ holds for every finite dimensional vector space $C$ and every linear map $d\colon C\to C$ satisfying $d^2=0$.
We have equality in \eqref{E:Kdeltaesti} if and only if we have equality in \eqref{E:dimHH} for all $i$.
In this case the codifferential is of maximal rank and $\mathcal H_{x,i}\cong\ker(\tilde A_{i,x})/\img(\tilde A_{i-1,x})$.
\end{remark}

\begin{remark}\label{R:Kdeltaexi}
If $\tilde A_i\tilde A_{i-1}=0$, then there exists a Kostant type codifferential of maximal rank.
To construct such a codifferential, fix fiber-wise graded Hermitian inner products on the graded vector bundles $\gr(E_i)$, let $\tilde A_i^*\colon\gr_*(E_{i+1})\to\gr_{*-k_i}(E_i)$ denote the fiber-wise adjoint of $\tilde A_i\colon\gr_*(E_i)\to\gr_{*+k_i}(E_{i+1})$, and consider $\delta_i:=S_{i-1}\tilde A_{i-1}^*S_i^{-1}$ where $S_i\colon\gr(E_i)\to E_i$ are some splittings for the filtrations.
In this case $\tilde\Box_i=\tilde A_i^*\tilde A_i+\tilde A_{i-1}\tilde A_{i-1}^*$, hence $\img(\tilde P_{i,x})=\ker(\tilde\Box_{i,x})=\ker(\tilde A_{i,x})\cap\ker(\tilde\delta_{i,x})$ 
and 
$$
\gr(E_{i,x})=\img(\tilde A_{i-1,x})\oplus\img(\tilde P_{i,x})\oplus\img(\tilde\delta_{i+1,x})
$$
for each $x\in M$.
We conclude that $\delta$ is a Kostant type codifferential of maximal rank, see Definition~\ref{D:Kdelta} and Remark~\ref{R:Kdeltabound}.
If the dimension of $\ker(\tilde A_{i,x})/\img(\tilde A_{i-1,x})$ is locally constant for all $i$, then $\delta_{i,x}$ has locally constant rank for all $i$, hence $\delta_i$ meets the assumptions in Proposition~\ref{P:D}, we obtain a sequence of operators $\bar D_i$ as in \eqref{E:D}, acting between sections of the vector bundles $\mathcal H_i\cong\ker(\tilde A_i)/\img(\tilde A_{i-1})$, and this is a graded Rockland sequence provided the original sequence $A_i$ was.
\end{remark}

\subsection{Linear connections on filtered manifolds}\label{SS:linearconnections}

In this section we consider a linear connection on a filtered vector bundle over a filtered manifold and its extension to bundle valued differential forms.
We will show that this induces a graded Rockland sequence in the sense of Definition~\ref{D:graded_hypoelliptic_seq} provided the linear connection is filtration-preserving and its curvature is contained in filtration degree one, see Proposition~\ref{P:hypo} below.

Suppose $E$ is a filtered vector bundle over a filtered manifold $M$.
We consider the induced filtration on the vector bundles $\Lambda^kT^*M\otimes E$.
To be explicit, $\psi\in\Lambda^kT^*_xM\otimes E_x$ is in filtration degree $p$ iff, for all tangent vectors $X_i\in T^{p_i}_xM$, we have $\psi(X_1,\dotsc,X_k)\in E_x^{p+p_1+\cdots+p_k}$.
Let us introduce the following notation for the associated graded vector bundle:
\begin{equation}\label{E:CkME}
C^k(M;E):=
\gr(\Lambda^kT^*M\otimes E)=\Lambda^k\mathfrak t^*M\otimes\gr(E).
\end{equation}
We will denote the grading by $C^k(M;E)=\bigoplus_pC^k(M;E)_p$.

Suppose $\nabla$ is a linear connection on $E$ such that
$\nabla_X\psi\in\Gamma^\infty(E^{p+q})$ for all $X\in\Gamma^\infty(T^pM)$ and $\psi\in\Gamma^\infty(E^q)$.
In other words, $\nabla\colon\Gamma^\infty(E)\to\Omega^1(M;E)$, is assumed to be filtration-preserving.
It is easy to see that filtration-preserving connections always exist, the space of all such connections is affine over the space of filtration-preserving vector bundle homomorphisms $E\to T^*M\otimes E$.
The Leibniz rule implies that the induced operator on the associated graded bundles, $\omega:=\gr(\nabla)\colon\Gamma^\infty(\gr(E))\to\Gamma^\infty(\gr(T^*M\otimes E))$, is tensorial, i.e.\
$$
\omega\in\Gamma^\infty\bigl(C^1(M;\eend(E))_0\bigr)
$$
where $C^1(M;\eend(E))=\gr(T^*M\otimes\eend(E))=\mathfrak t^*M\otimes\eend(\gr(E))$ with $0$-th grading component $C^1(M;\eend(E))_0=\bigoplus_{p,q}(\mathfrak t^pM)^*\otimes\hom(\gr_q(E),\gr_{p+q}(E))$.
Let 
\begin{equation}\label{E:dnabla}
d^\nabla\colon\Omega^*(M;E)\to\Omega^{*+1}(M;E)
\end{equation}
denote the usual extension of $\nabla$ characterized by the Leibniz rule,
\begin{equation}\label{E:leibn}
d^\nabla(\alpha\wedge\psi)=d\alpha\wedge\psi+(-1)^k\alpha\wedge d^\nabla\psi,
\end{equation}
for all $\alpha\in\Omega^k(M)$ and all $\psi\in\Omega^*(M;E)$.
Recall the explicit formula
\begin{multline}\label{E:dnablaf}
(d^\nabla\psi)(X_0,\dotsc,X_k)
=\sum_{i=0}^k(-1)^i\nabla_{X_i}\psi(X_0,\dotsc,\hat i,\dotsc,X_k)
\\+\sum_{0\leq i<j\leq k}(-1)^{i+j}\psi([X_i,X_j],X_0,\dotsc,\hat i,\dotsc,\hat j,\dotsc,X_k)
\end{multline}
for $\psi\in\Omega^k(M;E)$ and vector fields $X_0,\dotsc,X_k$.

\begin{lemma}\label{L:symb}
Let $\nabla\colon\Gamma^\infty(E)\to\Omega^1(M;E)$ be a filtration-preserving linear connection on a filtered vector bundle $E$ over a filtered manifold $M$.
Then the extension $d^\nabla\colon\Omega^k(M;E)\to\Omega^{k+1}(M;E)$ is a differential operator of graded Heisenberg order at most zero.
Moreover, the graded principal Heisenberg symbol at $x\in M$ fits into the following commutative diagram:
$$
\xymatrix{
C^\infty\bigl(\mathcal T_xM,C_x^k(M;E)\bigr)\ar@{=}[r]\ar[d]^-{\tilde\sigma^0_x(d^\nabla)}
&\Omega^k\bigl(\mathcal T_xM;\gr(E_x)\bigr)\ar[d]^-{d+\omega_x\wedge-}
\\
C^\infty\bigl(\mathcal T_xM,C_x^{k+1}(M;E)\bigr)\ar@{=}[r]
&\Omega^{k+1}\bigl(\mathcal T_xM;\gr(E_x)\bigr)
}
$$
Here the horizontal identifications are obtained by tensorizing the identity on $\gr(E_x)$ with the identification $C^\infty(\mathcal T_xM,\Lambda^k\mathfrak t_x^*M)=\Omega^k(\mathcal T_xM)$ induced by left trivialization of the tangent bundle of the group $\mathcal T_xM$.
Moreover, $\omega_x\in C^1_x(M;\eend(E))_0$ is considered as a left invariant $\eend(\gr(E_x))$-valued $1$-form on $\mathcal T_xM$.
\end{lemma}

\begin{proof}
This follows readily from  \eqref{E:dnablaf} and \eqref{E:snablaX}.
Alternatively, this can be understood as follows.
Let $A\colon E\to T^*M\otimes E$ be a filtration-preserving vector bundle homomorphism. 
Then $\nabla+A$ is another filtration-preserving linear connection on $E$, and $\omega^{\nabla+A}=\omega^\nabla+\tilde A$ where $\tilde A=\gr(A)\colon\gr(E)\to\gr(T^*M\otimes E)$ denotes the associate graded vector bundle homomorphism induced by $A$.
Moreover, $d^{\nabla+A}=d^\nabla+A\wedge-$ and thus $\tilde\sigma^0_x(d^{\nabla+A})=\tilde\sigma^0_x(d^\nabla)+\tilde A_x\wedge-$.
We conclude that the statement of the lemma holds for $\nabla$ iff it holds for $\nabla+A$.
Since this statement is local, and since the space of filtration-preserving linear connections on $E$ is affine over the space of filtration-preserving vector bundle homomorphisms $E\to T^*M\otimes E$, we may w.l.o.g.\ assume that $\nabla$ is the trivial connection on a trivial bundle $E=M\times E_0$.
By compatibility with direct sums, it suffices to consider the trivial line bundle $E=M\times\C$, that is, we may assume $d^\nabla=d$, the de~Rham differential on $\Omega^*(M)$.
In view of the Leibniz rule, it suffices to show that $d\colon\Omega^k(M)\to\Omega^{k+1}(M)$ has graded Heisenberg order at most zero and $\tilde\sigma^0_x(d)=d$, for $k=0,1$.
This, however, can readily be checked.
\end{proof}

Using the Leibniz rule (or Lemma~\ref{L:symb} above), one readily checks that the differential operator \eqref{E:dnabla} is filtration-preserving.
Moreover, the operator induced on the associated graded is tensorial, given by a vector bundle homomorphism of bidegree $(1,0)$,
\begin{equation}\label{E:delomega}
\partial^\omega\colon C^k(M;E)_p\to C^{k+1}(M;E)_p,\qquad
\partial^\omega:=\gr(d^\nabla),
\end{equation}
which can be characterized as the unique extension of $C^0(M;E)\xrightarrow\omega C^1(M;E)$ satisfying the Leibniz rule
\begin{equation}\label{E:leibdel}
\partial^\omega(\alpha\wedge\psi)
=\partial\alpha\wedge\psi+(-1)^k\alpha\wedge \partial^\omega\psi,
\end{equation}
for all $\alpha\in\Lambda^k\mathfrak t^*_xM$ and $\psi\in C_x^*(M;E)$.
Here $\partial\colon\Lambda^k\mathfrak t^*M\to\Lambda^{k+1}\mathfrak t^*M$ denotes the fiber-wise Chevalley--Eilenberg differential.
More explicitly, we have
\begin{equation}\label{E:delomega2}
	\partial^\omega=\partial\otimes\id_E+\omega\wedge-.
\end{equation}

\begin{lemma}\label{L:curv}
For each $x\in M$ the following are equivalent:
\begin{enumerate}[(a)]
\item\label{L:curc:a} The curvature $F^\nabla_x\in\Omega_x^2(M;\eend(E))$ is contained in filtration degree one, that is, for all $X_i\in T_x^{p_i}M$ and $\psi\in E^p_x$ we have $F_x^\nabla(X_1,X_2)\psi\in E^{p+p_1+p_2+1}_x$.
\item\label{L:curc:b} $\omega_x\colon\mathfrak t_xM\to\eend(\gr(E_x))$ provides a graded representation of the graded nilpotent Lie algebra $\mathfrak t_xM$ on the graded vector space $\gr(E_x)$.
\item\label{L:curc:c} $(\partial_x^\omega)^2=0$.
\item\label{L:curc:d} $\tilde\sigma^0_x(d^\nabla)^2=0$.
\end{enumerate}
In this case, the Chevalley--Eilenberg differential of the Lie algebra $\mathfrak t_xM$ with coefficients in the representation $\gr(E_x)$ coincides with \eqref{E:delomega} at the point $x$.
\end{lemma}

\begin{proof}
Recall that $(d^\nabla)^2=F^\nabla\wedge-$ is tensorial.
Using using \eqref{E:tsAB} and Remark~\ref{R:grsigma}, we obtain
$$
\tilde\sigma^0_x(d^\nabla)^2
=\tilde\sigma_x^0((d^\nabla)^2)
=\gr_x((d^\nabla)^2)
=\gr_x(d^\nabla)^2
=(\partial_x^\omega)^2,
$$
whence the equivalences \itemref{L:curc:a}$\Leftrightarrow$\itemref{L:curc:c}$\Leftrightarrow$\itemref{L:curc:d}.
To see the equivalence \itemref{L:curc:b}$\Leftrightarrow$\itemref{L:curc:c}, we note that \eqref{E:leibdel} and $\partial^2=0$ give $(\partial_x^\omega)^2(\alpha\wedge\psi)=\alpha\wedge(\partial_x^\omega)^2\psi$.
Furthermore, using \eqref{E:delomega2} one readily shows
$$
\bigl((\partial_x^\omega)^2\phi\bigr)(X,Y)=\bigl(\omega(X)\omega(Y)-\omega(Y)\omega(X)-\omega([X,Y])\bigr)\phi,
$$
for $\phi\in E_x=C^0_x(M;E)$ and $X,Y\in\mathfrak t_xM$.
\end{proof}

For flat connections (on nilpotent Lie groups) the following is due to Rumin, see \cite[Theorem~5.2]{R01} and \cite[Theorem~3]{R99}.

\begin{proposition}\label{P:hypo}
Let $\nabla\colon\Gamma^\infty(E)\to\Omega^1(M;E)$ be a filtration-preserving linear connection on a filtered vector bundle $E$ over a filtered manifold $M$.
If the curvature of $\nabla$ is contained in filtration degree one, see Lemma~\ref{L:curv}, then 
\begin{equation}\label{E:dnablak}
\cdots\to\Omega^{k-1}(M;E)\xrightarrow{d^\nabla}\Omega^k(M;E)
\xrightarrow{d^\nabla}\Omega^{k+1}(M;E)\to\cdots
\end{equation}
is a graded Rockland sequence.
\end{proposition}

\begin{proof}
Fix $x\in M$, consider the nilpotent Lie group $G:=\mathcal T_xM$ with Lie algebra $\goe:=\mathfrak t_xM$ and the $\goe$-module $V:=\gr(E_x)$, see Lemma~\ref{L:curv}.
Hence, $C_x^k(M;E)=\Lambda^k\goe^*\otimes V$.
According to Lemma~\ref{L:symb}, the Heisenberg principal symbol sequence
$$
\cdots\to C^\infty(G,\Lambda^k\goe^*\otimes V)\xrightarrow{\tilde\sigma_x^0(d^\nabla)}
C^\infty(G,\Lambda^{k+1}\goe^*\otimes V)\to\cdots
$$
is isomorphic to the Chevalley--Eilenberg complex of the Lie algebra $\goe$ with values in the $\goe$-representation $C^\infty(G)\otimes V$.
More explicitly, if $X_j$ is a basis of $\goe$ and $\alpha^j$ denotes the dual basis of $\goe^*$ then, in $\mathcal U(\goe)\otimes\hom\bigl(\Lambda^k\goe^*\otimes V,\Lambda^{k+1}\goe^*\otimes V\bigr)$, we have
\begin{equation}\label{E:ChE}
\tilde\sigma_x^0(d^\nabla)=\sum_jX_j\otimes e_{\alpha^j}+\sum_j1\otimes e_{d\alpha^j}i_{X_j}+1\otimes\omega_x
\end{equation}
Here $e_{\alpha^j}\in\hom\bigl(\Lambda^k\goe^*\otimes V,\Lambda^{k+1}\goe^*\otimes V\bigr)$ denotes the exterior product with $\alpha^j\in\goe^*$; $i_{X_j}\in\hom\bigl(\Lambda^k\goe^*\otimes V,\Lambda^{k-1}\goe^*\otimes V\bigr)$ denotes the contraction with $X_j\in\goe$; $e_{d\alpha^j}\in\hom\bigl(\Lambda^{k-1}\goe^*\otimes V,\Lambda^{k+1}\goe^*\otimes V\bigr)$ denotes the exterior product with $d\alpha^j\in\Lambda^2\goe^*$; and $\omega_x\in\hom\bigl(\Lambda^k\goe^*\otimes V,\Lambda^{k+1}\goe^*\otimes V\bigr)$ denotes the exterior product with the representation $\omega_x\colon\goe\to\eend(V)$.

Suppose $\pi\colon G\to U(\mathcal H)$ is a non-trivial irreducible unitary representation on a Hilbert space $\mathcal H$, and let $\mathcal H_\infty$ denote the space of smooth vectors.
Using \eqref{E:ChE} one readily checks, see Section~\ref{SS:paraDO}, that the sequence
\begin{equation*}\label{E:tbse}
\cdots\to\mathcal H_\infty\otimes\Lambda^k\goe^*\otimes V\xrightarrow{\pi(\tilde\sigma^0_x(d^\nabla))}\mathcal H_\infty\otimes\Lambda^{k+1}\goe^*\otimes V\to\cdots
\end{equation*}
is isomorphic to the Chevalley--Eilenberg complex of the Lie algebra $\goe$ with values in the $\goe$-representation $\mathcal H_\infty\otimes V$.
Consequently, it remains to show that the Lie algebra cohomology of $\goe$ with coefficients in the $\goe$-module $\mathcal H_\infty\otimes V$ vanishes, that is, $H^*(\goe;\mathcal H_\infty\otimes V)=0$.

Since $\goe$ is nilpotent, there exists a $1$-dimensional subspace $W\subseteq V$ on which $\goe$ acts trivially.
The corresponding short exact sequence of $\goe$-modules, $0\to W\to V\to V/W\to0$, yields a short exact sequence of $\goe$-modules $0\to\mathcal H_\infty\to\mathcal H_\infty\otimes  V\to\mathcal H_\infty\otimes  V/W\to0$ which, in turn, gives rise to a long exact sequence:
$$
\cdots\to H^q(\goe;\mathcal H_\infty)\to H^q(\goe;\mathcal H_\infty\otimes V)\to H^q(\goe;\mathcal H_\infty\otimes V/W)\xrightarrow\partial H^{q+1}(\goe;\mathcal H_\infty)\to\cdots
$$
Hence, by induction on the dimension of $V$, it suffices to show $H^q(\goe;\mathcal H_\infty)=0$, for all $q$.
The statement thus follows from Lemma~\ref{L:Hnoe} below.
\end{proof}

\begin{lemma}\label{L:Hnoe}
Consider a non-trivial irreducible unitary representation of a finite dimensional simply connected nilpotent Lie group $G$ on a Hilbert space $\mathcal H$, and the associated representation of the corresponding Lie algebra $\goe$ on the space of smooth vectors $\mathcal H_\infty$.
Then the Lie algebra cohomology of $\goe$ with coefficients in $\mathcal H_\infty$ is trivial, that is, $H^*(\goe;\mathcal H_\infty)=0$.
\end{lemma}

\begin{proof}
The proof proceeds by induction on the dimension of $\goe$.
Since $\goe$ is nilpotent, there exists a 1-dimensional central subalgebra $\zoe\subseteq\goe$.
Recall that there is a Hochschild--Serre spectral sequence \cite{HS53} converging to $H^*(\goe;\mathcal H_\infty)$ with $E_2$-term 
$$
E_2^{p,q}\cong H^p\bigl(\goe/\zoe;H^q(\zoe;\mathcal H_\infty)\bigr).
$$
Since the representation of $G$ on the Hilbert space $\mathcal H$ is irreducible, $\zoe$ acts by scalars on $\mathcal H_\infty$, see \cite[Theorem~5 in Appendix~V]{K04}.
If this action is non-trivial, then $H^*(\zoe;\mathcal H_\infty)=0$ and, consequently, the $E_2$-term vanishes.
We may thus assume that the action of $\zoe$ on $\mathcal H_\infty$ is trivial. 
Hence, $H^*(\zoe;\mathcal H_\infty)=H^0(\zoe;\mathcal H_\infty)\oplus H^1(\zoe;\mathcal H_\infty)\cong\mathcal H_\infty\oplus\mathcal H_\infty$ as $\goe/\zoe$-modules.

Consider the closed connected central subgroup $Z:=\exp(\zoe)$ in $G$, and note that $G/Z$ is a simply connected nilpotent Lie group with Lie algebra $\goe/\zoe$.
Since $\mathcal H_\infty$ is dense in $\mathcal H$, the subgroup $Z$ acts trivially on $\mathcal H$, see \cite[Theorem~4 in Appendix~V]{K04}.
Hence the representation of $G$ factors through a representation of $G/Z$ on $\mathcal H$.
Clearly, this is a non-trivial irreducible unitary representation of $G/Z$ whose space of smooth vectors coincides with $\mathcal H_\infty$.
Hence, $H^*(\goe/\zoe;\mathcal H_\infty)=0$, by induction.
We conclude $H^p(\goe/\zoe;H^*(\zoe;\mathcal H_\infty))\cong H^p(\goe/\zoe;\mathcal H_\infty\oplus\mathcal H_\infty)=H^p(\goe/\zoe;\mathcal H_\infty)\oplus H^p(\goe/\zoe;\mathcal H_\infty)=0$.
Thus the $E_2$-term vanishes, and the proof is complete.
\end{proof}

\subsection{Subcomplexes of the de~Rham complex}\label{SS:subdeRham}

In this section we apply the observations and constructions from Sections~\ref{SS:DP} and \ref{SS:splitting-operators} to the de~Rham sequence associated with a filtration-preserving linear connection.
This directly leads to the main result of this section, see Theorem~\ref{T:D} and Corollary~\ref{C:D} below.
In the subsequent Section~\ref{SS:BGG} we will apply this to the curved BGG sequences in parabolic geometry which appear as a special case of the sequences considered here.

Let $\nabla$ be a filtration-preserving linear connection on a filtered vector bundle $E$ over a filtered manifold $M$, and consider its extension to $E$-valued differential forms characterized by the Leibniz rule in \eqref{E:leibn},
\begin{equation}\label{E:Tdnabla}
\cdots\to\Omega^{k-1}(M;E)\xrightarrow{d^\nabla_{k-1}}\Omega^k(M;E)\xrightarrow{d^\nabla_k}\Omega^{k+1}(M;E)\to\cdots.
\end{equation}
Consider a sequence of differential operators
\begin{equation}\label{E:Tdelta}
\cdots\leftarrow\Omega^{k-1}(M;E)\xleftarrow{\delta_k}\Omega^k(M;E)\xleftarrow{\delta_{k+1}}\Omega^{k+1}(M;E)\leftarrow\cdots
\end{equation}
which are of graded Heisenberg order at most zero.
Then the differential operator
$$
\Box_k\colon\Omega^k(M;E)\to\Omega^k(M;E),\qquad
\Box_k:=d^\nabla_{k-1}\delta_k+\delta_{k+1}d^\nabla_k,
$$
is of graded Heisenberg order at most zero.
Let 
\begin{align*}
\partial_k^\omega&\colon C^k(M;E)\to C^{k+1}(M;E),&
\partial_k^\omega&:=\gr(d^\nabla_k),
\\
\tilde\delta_k&\colon C^k(M;E)\to C^{k-1}(M;E),&
\tilde\delta_k&:=\gr(\delta_k),
\\
\tilde\Box_k&\colon C^k(M;E)\to C^k(M;E),&
\tilde\Box_k&:=\gr(\Box_k)=\partial^\omega_{k-1}\tilde\delta_k+\tilde\delta_{k+1}\partial^\omega_k,
\end{align*}
denote the associated graded vector bundle homomorphisms, see Remark~\ref{R:grsigma}, \eqref{E:CkME}, and \eqref{E:delomega}.

For each $x\in M$, let $\tilde P_{k,x}\colon C_x^k(M;E)\to C_x^k(M;E)$ denote the spectral projection onto the generalized zero eigenspace of $\tilde\Box_{x,k}\colon C_x^k(M;E)\to C^k_x(M;E)$, the restriction of $\tilde\Box_k$ to the fiber over $x$.
Assume that the rank of $\tilde P_{k,x}$ is locally constant in $x$ for each $k$.
Then, see Lemma~\ref{L:EP}(a), $\tilde P_k\colon C^k(M;E)\to C^k(M;E)$ is a smooth vector bundle homomorphism,
$\tilde P_k^2=\tilde P_k$, $\tilde P_k\tilde\Box_k=\tilde\Box_k\tilde P_k$,  and we obtain a decomposition of graded vector bundles,
\begin{equation}\label{E:tpdeco2}
C^k(M;E)=\img(\tilde P_k)\oplus\ker(\tilde P_k),
\end{equation}
which is invariant under $\tilde\Box_k$.
Moreover, $\tilde\Box_k$ is nilpotent on $\img(\tilde P_k)$ and invertible on $\ker(\tilde P_k)$.

According to Lemma~\ref{L:EP}(b), there exists a unique filtration-preserving differential operator
$$
P_k\colon\Omega^k(M;E)\to\Omega^k(M;E)
$$ 
such that $P_k^2=P_k$, $P_k\Box_k=\Box_kP_k$, and $\gr(P_k)=\tilde P_k$.
This operator $P_k$ has graded Heisenberg order at most zero and provides a decomposition of filtered vector spaces, 
\begin{equation}\label{E:Pdeco2}
\Omega^k(M;E)=\img(P_k)\oplus\ker(P_k),
\end{equation}
invariant under $\Box_k$ and such that $\Box_k$ is nilpotent on $\img(P_k)$ and invertible on $\ker(P_k)$.
If
$$
C^k(M,E)=\gr(\Lambda^kT^*M\otimes E)\xrightarrow{S_k}\Lambda^kT^*M\otimes E,
$$ 
are splittings of the filtrations, then according to Lemma~\ref{L:EP}(c)
\begin{equation}\label{E:tL2}
L_k\colon\Gamma^\infty(C^k(M;E))\to\Omega^k(M;E),\qquad
L_k:=P_kS_k\tilde P_k+(\id-P_k)S_k(\id-\tilde P_k),
\end{equation}
is an invertible differential operator of graded Heisenberg order at most zero such that $\gr(L_k)=\id$ and $L^{-1}_kP_kL_k=\tilde P_k$.
Hence, $L_k$ induces filtration-preserving isomorphisms
$$
L_k\colon\Gamma^\infty(\img(\tilde P_k))\xrightarrow\cong\img(P_k)
\qquad\text{and}\qquad
L_k\colon\Gamma^\infty(\ker(\tilde P_k))\xrightarrow\cong\ker(P_k).
$$
Moreover, $L^{-1}_k\Box_kL_k$ is a differential operator of graded Heisenberg order at most zero preserving the decomposition \eqref{E:tpdeco2} and satisfying $\gr(L^{-1}_k\Box_kL_k)=\tilde\Box_k$.
Putting 
\begin{equation}\label{E:defDB2}
D_k:=\tilde P_{k+1}L_{k+1}^{-1}d^\nabla_kL_k|_{\Gamma^\infty(\img(\tilde P_k))}
\qquad\text{and}\qquad
B_k:=(\id-\tilde P_{k+1})L_{k+1}^{-1}d^\nabla_kL_k|_{\Gamma^\infty(\ker(\tilde P_k))}
\end{equation}
we obtain two sequences of differential operators,
\begin{equation}\label{E:TEDE}
\cdots\to\Gamma^\infty(\img(\tilde P_{k-1}))\xrightarrow{D_{k-1}}\Gamma^\infty(\img(\tilde P_k))\xrightarrow{D_k}\Gamma^\infty(\img(\tilde P_{k+1}))\to\cdots
\end{equation}
and
\begin{equation}\label{E:TEBE}
\cdots\to\Gamma^\infty(\ker(\tilde P_{k-1}))\xrightarrow{B_{k-1}}\Gamma^\infty(\ker(\tilde P_k))\xrightarrow{B_k}\Gamma^\infty(\ker(\tilde P_{k+1}))\to\cdots,
\end{equation}
all of which have graded Heisenberg order at most zero.

Combining Proposition~\ref{P:DB} with Proposition~\ref{P:hypo}, we immediately obtain:

\begin{theorem}\label{T:D}
In this situation the following hold true:
\begin{enumerate}[(a)]
\item
If the curvature of $\nabla$ is contained in filtration degree one, cf.\ Lemma~\ref{L:curv}, then \eqref{E:TEDE} and \eqref{E:TEBE} are both graded Rockland sequences.
\item
If the curvature of\/ $\nabla$ vanishes, then the operator $L_{k+1}^{-1}d^\nabla_kL_k$ decouples,
$$
L^{-1}_{k+1}d^\nabla_kL_k=D_k\oplus B_k,
$$ 
and we have $D_kD_{k-1}=0$ as well as $B_kB_{k-1}=0$ for all $k$.
In this situation, $G_k\colon\Gamma^\infty(\ker(\tilde P_k))\to\Gamma^\infty(\ker(\tilde P_k))$,
$$
G_k:=B_{k-1}(\id-\tilde P_{k-1})\tilde\delta_k\tilde\Box_k^{-1}
+(\id-\tilde P_k)\tilde\delta_{k+1}\tilde\Box_{k+1}^{-1}\partial^\omega_k,
$$ 
is an invertible differential operator of graded Heisenberg order at most zero with $\gr(G_k)=\id$ that conjugates the complex \eqref{E:TEBE} into an acyclic tensorial complex, namely, 
$$
G_{k+1}^{-1}B_kG_k=\partial^\omega_k|_{\Gamma^\infty(\ker(\tilde P_k))}.
$$
Moreover, the restriction $L_k\colon\Gamma^\infty(\img(\tilde P_k))\to\Omega^k(M;E)$ provides a chain map, $d^\nabla_kL_k|_{\Gamma^\infty(\img(\tilde P_k))}=L_{k+1}D_k$, inducing an isomorphism between the cohomologies of \eqref{E:TEDE} and \eqref{E:Tdnabla}.
More precisely, $\pi_k:=\tilde P_kL_k^{-1}\colon\Omega^k(M;E)\to\Gamma^\infty(\img(\tilde P_k))$, is a chain map, $D_k\pi_k=\pi_{k+1}d^\nabla_k$, which is an inverse of $L_k$ up to homotopy, i.e., $\pi_kL_k|_{\Gamma^\infty(\img(\tilde P_k))}=\id$ and $\id-L_k\pi_k=d^\nabla_{k-1}h_k+h_{k+1}d^\nabla_k$ where the homotopy $h_k\colon\Omega^k(M;E)\to\Omega^{k-1}(M;E)$ is a differential operator of graded Heisenberg order at most zero given by $h_k:=L_{k-1}G_{k-1}(\id-\tilde P_k)\tilde\delta_k\tilde\Box_k^{-1}G_k^{-1}(\id-\tilde P_k)L_k^{-1}$.
\end{enumerate}
\end{theorem}

For flat connections the preceding theorem is due to Rumin, see \cite[Theorems~2.6 and 5.2]{R01} or \cite[Theorems~1 and 3]{R99}.

Combining Proposition~\ref{P:D} with Proposition~\ref{P:hypo} we immediately obtain:

\begin{corollary}\label{C:D}
Let $\nabla$ be a filtration-preserving linear connection on a filtered vector bundle $E$ over a filtered manifold $M$ whose curvature is contained in filtration degree one, and consider its extension to $E$-valued differential forms characterized by the Leibniz rule, see~\eqref{E:leibn},
\begin{equation}\label{E:Cdnabla}
\cdots\to\Omega^{k-1}(M;E)\xrightarrow{d^\nabla}\Omega^k(M;E)\xrightarrow{d^\nabla}\Omega^{k+1}(M;E)\to\cdots.
\end{equation}
Moreover, consider a codifferential of Kostant type, see Definition~\ref{D:Kdelta},
$$
\cdots\leftarrow\Lambda^{k-1}T^*M\otimes E\xleftarrow{\delta_i}\Lambda^kT^*M\otimes E\xleftarrow{\delta_{i+1}}\Lambda^{k+1}T^*M\otimes E\leftarrow\cdots,
$$
which has locally constant rank, and let $\bar\pi_k\colon\ker(\delta_k)\to\mathcal H_k:=\ker(\delta_k)/\img(\delta_{k+1})$ denote the canonical vector bundle projection.
Then the following hold true:
\begin{enumerate}[(a)]
\item
There exists a unique differential operator $\bar L_k\colon\Gamma^\infty(\mathcal H_k)\to\Omega^k(M;E)$ such that $\delta_k\bar L_k=0$, $\bar\pi_k\bar L_k=\id$, and $\delta_{k+1}d^\nabla_k\bar L_k=0$.
Moreover, $\bar L_k$ is of graded Heisenberg order at most zero.
\item
The differential operator
$$
\bar D_k\colon\Gamma^\infty(\mathcal H_k)\to\Gamma^\infty(\mathcal H_{k+1}),\qquad\bar D_k:=\bar\pi_{k+1}d^\nabla_kL_k,
$$ 
is of graded Heisenberg order at most zero, and 
\begin{equation}\label{E:bD}
\cdots\to\Gamma^\infty(\mathcal H_{k-1})\xrightarrow{\bar D_{k-1}}\Gamma^\infty(\mathcal H_k)\xrightarrow{\bar D_k}\Gamma^\infty(\mathcal H_{k+1})\to\cdots
\end{equation}
is graded Rockland sequence.
\item
If $\nabla$ has vanishing curvature, then $\bar D_k\bar D_{k-1}=0$, and $\bar L_k\colon\Gamma^\infty(\mathcal H_k)\to\Omega^k(M;E)$ provides a chain map, $d^\nabla_k\bar L_k=\bar L_{k+1}\bar D_k$, inducing an isomorphism between the cohomologies of \eqref{E:bD} and \eqref{E:Cdnabla}.
\item
There exist invertible differential operators, $V_i\colon\Gamma^\infty(\img(\tilde P_i))\to\Gamma^\infty(\mathcal H_i)$ with inverse $V_i^{-1}\colon\Gamma^\infty(\mathcal H_i)\to\Gamma^\infty(\img(\tilde P_i))$, both of graded Heisenberg order at most zero, such that 
$$
V_{i+1}^{-1}\bar D_iV_i=D_i,
$$ 
where $D_i\colon\Gamma^\infty(\img(\tilde P_i))\to\Gamma^\infty(\img(\tilde P_{i+1}))$ denotes the operator considered in Theorem~\ref{T:D}, see \eqref{E:defDB2}, corresponding to splittings $S_k$ as in Definition~\ref{D:Kdelta}(iii).
\end{enumerate}
\end{corollary}

The operators $\bar L_k$ and $\bar D_k$ in Corollary~\ref{C:D} are direct generalizations of the splitting operators and the curved BGG operators in parabolic geometry, respectively.
This aspect will be addressed in Section~\ref{SS:BGG} below.
The differential projector $P_k$ generalizes the operator obtained by composing (5.1) with (5.2) in \cite{CD01}.

Detailed explanations of how Rumin's complex \cite{R90,R94} appears among the sequences \eqref{E:bD} considered in Corollary~\ref{C:D} can be found in \cite[Example~4.21]{DH17}.
Hypoellipticity of this sequence has been established by Rumin, see \cite[Section~3]{R94}.
Let us mention that Rumin and Seshadri \cite{RS12} have introduced an analytic torsion based on hypoelliptic Laplacians associated with Rumin's complex.
We expect that their construction can be be extended to all (ungraded) sequences appearing in Corollary~\ref{C:D}(b).

We conclude this section with a 4-dimensional geometry for which the sequence in Theorem~\ref{T:D}(b) turns out to be graded Rockland but not Rockland in the ungraded sense.

\begin{example}[Engel structures]\label{Ex:Engel}
Recall that an Engel structure \cite{P16} on a smooth 4-manifold $M$ is a smooth rank two distribution $T^{-1}M\subseteq TM$ with growth vector $(2,3,4)$.
More explicitly, Lie brackets of sections of $T^{-1}M$ generate a rank three bundle $T^{-2}M$, and triple brackets of sections of $T^{-1}M$ generate all of $TM$.
Hence, $M$ is a filtered manifold,
$$
TM=T^{-3}M\supseteq T^{-2}M\supseteq T^{-1}M\supseteq T^0M=0.
$$
One readily checks that the bundle of osculating algebras is locally trivial with typical fiber isomorphic to the graded nilpotent Lie algebra $\goe=\goe_{-3}\oplus\goe_{-2}\oplus\goe_{-1}$ with non-trivial brackets 
$$
[X_1,X_2]=X_3,\qquad [X_1,X_3]=X_4,\qquad [X_2,X_3]=0,
$$ 
where $X_1,X_2$ is a basis of $\goe_{-1}$, $X_3$ is a basis of $\goe_{-2}$ and $X_4$ is a basis of $\goe_{-3}$.
If $G$ denotes the simply connected nilpotent Lie group with Lie algebra $\goe$, then the left invariant distribution corresponding to $\goe_{-1}$ provides an Engel structure on $G$.
Locally, every Engel structure is diffeomorphic to this left invariant structure.
According to a result of Vogel \cite{V09}, every closed parallelizable smooth 4-manifold admits an (orientable) Engel structure.

We will exhibit explicit formulas for the graded Heisenberg principal symbol of the operators $D_k$ in Theorem~\ref{T:D} corresponding to the de~Rham complex, that is, the trivial flat line bundle $E$. 
We work on $M=G$ equipped with the left invariant Engel structure mentioned before.
We use the left invariant splitting for $S\colon\gr(\Lambda^kT^*G)\to\Lambda^kT^*G$ provided by the decomposition $\goe=\goe_{-3}\oplus\goe_{-2}\oplus\goe_{-1}$.
Moreover, we consider the left invariant codifferential $\delta_k\colon\Lambda^kT^*G\to\Lambda^{k-1}T^*G$ which, at the identity, is dual to the Chevalley--Eilenberg differential $\partial_{k-1}\colon\Lambda^{k-1}\goe^*\to\Lambda^k\goe^*$ with respect to the basis $X_1,\dotsc,X_4$. 
Clearly, this is a Kostant type codifferential of maximal rank.
In this situation, Theorem~\ref{T:D}(b) provides a graded Rockland complex of left invariant operators, cf.\ \cite{BEGN19} and \cite[Section~2.3]{R01}:
\begin{equation}\label{E:DEngel}
C^\infty(G)\xrightarrow{D_0}
C^\infty(G)^2\xrightarrow{D_1}
\begin{array}{c}C^\infty(G)\\\oplus\\ C^\infty(G)\end{array}\xrightarrow{D_2}
C^\infty(G)^2\xrightarrow{D_3}
C^\infty(G).
\end{equation}
Using matrices with entries in the universal enveloping algebra of $\goe$, and using the notation $X_{i_1\dotsc i_k}=X_{i_1}\cdots X_{i_k}$, these operators can be expressed as:
$$
D_0=\left(\begin{array}{c}X_1\\X_2\end{array}\right)
$$

$$
D_1=\left(\begin{array}{cc}-X_{22}&X_{12}-2X_3\\\hline-X_4-X_{112}-X_{13}&X_{111}\end{array}\right)
$$

$$
D_2=\left(\begin{array}{c|c}X_{111}&-X_3-X_{12}\\ 3X_4-3X_{13}+X_{112}&-X_{22}\end{array}\right)
$$

$$
D_3=\left(\begin{array}{cc}-X_2&X_1\end{array}\right)
$$
A detailed verification of these claims can be found in \cite[Appendix~A]{DH17}.
The direct sum symbol in the sequence \eqref{E:DEngel} indicates that its filtration is non-trivial.
Correspondingly, the operators $D_1$ and $D_2$ are not homogeneous, whence the lines in their matrices separating different degrees.
Evidently, the sequence \eqref{E:DEngel} fails to be Rockland in the ungraded sense.
\end{example}

\subsection{BGG sequences}\label{SS:BGG}

Every regular parabolic geometry has an underlying filtered manifold $M$ whose bundle of osculating Lie algebras is locally trivial.
The Cartan connection induces a linear connection $\nabla$ on every associated tractor bundle $E$ over $M$.
This linear connection is filtration-preserving and its curvature is contained in filtration degree one.
Moreover, Kostant's codifferential provides a vector bundle homomorphism $\delta\colon\Lambda^kT^*M\otimes E\to\Lambda^{k-1}T^*M\otimes E$, satisfying the assumptions in Corollary~\ref{C:D}.
In this situation, the operator $\bar L$ in Corollary~\ref{C:D}(a) coincides with the splitting operator in \cite[Theorem~2.4]{CS12}, see also \cite{CSS01,CD01}.
Furthermore, the sequence in Corollary~\ref{C:D}(b) reduces to the  ``torsion free BGG sequence'' in \cite[Section~5]{CD01} and coincides with the sequence considered in \cite[Section~2.4]{CS12}.
For torsion free parabolic geometries this coincides with the original curved BGG sequence constructed by \v Cap, Slov\'ak and Sou\v cek in \cite{CSS01}.
From Corollary~\ref{C:D}(b) we conclude that these BGG operators have graded Heisenberg order at most zero and form a graded Rockland sequence, see Corollary~\ref{C:BGG} below.

Let us point out that there is a normalization condition for the curvature of the Cartan connection such that normal regular parabolic geometries can equivalently be described by underlying geometric structures, see Theorem~3.1.14, Section~3.1.16, and the historical remarks at the end of Section~3 in \cite{CS09}.
For a large class of parabolic geometries this underlying geometric structure consists merely of the underlying filtered manifold, \cite[Proposition~4.3.1]{CS09}.
These provide intriguing classes of filtered manifolds to which one can associate graded Rockland sequences of differential operators in a natural way.

In the remaining part of this section we will, for the reader's convenience, briefly recall basic facts on parabolic geometries and provided detailed references supporting the claims made above.
We closely follow the presentation in \cite{CS09}, \cite{CSS01} and \cite[Section~2]{CS12}.

Consider a $|k|$-graded semisimple Lie algebra
\begin{equation}\label{E:grading}
\goe=\goe_{-k}\oplus\cdots\oplus\goe_{-1}\oplus\goe_0\oplus\goe_1\oplus\cdots\oplus\goe_k.
\end{equation}
More precisely, $[\goe_i,\goe_j]\subseteq\goe_{i+j}$ for all $i,j$, and the subalgebra $\goe_-:=\goe_{-k}\oplus\cdots\oplus\goe_{-1}$ is generated by $\goe_{-1}$, see \cite[Definition~3.1.2]{CS09}.
Consider the filtration 
\begin{equation}\label{E:filt}
\goe=\goe^{-k}\supseteq\goe^{-k+1}\supseteq\cdots\supseteq\goe^{-1}\supseteq\goe^0\supseteq\goe^1\supseteq\cdots\supseteq\goe^{k-1}\supseteq\goe^k
\end{equation}
where $\goe^i:=\goe_i\oplus\cdots\oplus\goe_k$.
The subalgebras $\goe_0$ and $\poe:=\goe^0=\goe_0\oplus\cdots\oplus\goe_k$ can be characterized as grading and filtration-preserving subalgebras, respectively.
More precisely, $\goe_0=\{X\in\goe\mid\forall i:\ad_X(\goe_i)\subseteq\goe_i\}$ and $\poe=\{X\in\goe\mid\forall i:\ad_X(\goe^i)\subseteq\goe^i\}$, see\cite[Lemma~3.1.3(1)]{CS09}.
Also note that $\poe_+:=\goe^1=\goe_1\oplus\cdots\oplus\goe_k$ is a nilpotent ideal in $\poe$.

Let $G$ be a not necessarily connected Lie group with Lie algebra $\goe$.
Then
\begin{equation}\label{E:Pt}
\{g\in G\mid\forall i:\Ad_g(\goe^i)\subseteq\goe^i\}
\end{equation}
is a closed subgroup of $G$ with Lie algebra $\poe$, see \cite[Lemma~3.1.3(2)]{CS09}.
Let $P$ be a parabolic subgroup of $G$ corresponding to the $|k|$-grading \eqref{E:grading}, i.e.\ a subgroup between \eqref{E:Pt}
and its connected component, see \cite[Definition~3.1.3]{CS09}. Hence, $P$ has Lie algebra $\poe$, and
the corresponding Levi subgroup,
$$
G_0:=\{g\in P\mid\forall i:\Ad_g(\goe_i)\subseteq\goe_i\},
$$
has Lie algebra $\goe_0$. According to \cite[Theorem~3.1.3]{CS09} we have a diffeomorphism
\begin{equation}\label{E:PG0iso}
G_0\times\poe_+\cong P,\qquad (g,X)\mapsto g\exp(X).
\end{equation}
In particular, $P_+:=\exp(\poe_+)$ is a closed normal nilpotent subgroup of $P$, and the inclusion $G_0\subseteq P$ induces a canonical isomorphism $G_0=P/P_+$.
Note that $P_+$ acts trivially on the associated graded, $\gr(\goe)$, of the filtration \eqref{E:filt}. 
In particular, $\gr(\goe/\poe)$ can be considered as a representation of $P/P_+=G_0$. 
The inclusion $\goe_-\subseteq\goe$ induces a canonical isomorphism of $G_0$-modules,
\begin{equation}\label{E:grgp}
\goe_-=\gr(\goe/\poe).
\end{equation}

A parabolic geometry of type $(G,P)$ consists of a principal $P$-bundle $p\colon\mathcal G\to M$ and a Cartan connection $\omega\in\Omega^1(\mathcal G;\goe)$, see \cite[Definition~3.1.4 and Section~1.5]{CS09}.
Hence, the $\goe$-valued $1$-form $\omega$ provides a $P$-equivariant trivialization of the tangent bundle $T\mathcal G$, that is, $\omega_u\colon T_u\mathcal G\to\goe$ is a linear isomorphism for each $u\in\mathcal G$, and $(r^g)^*\omega=\Ad_{g^{-1}}\omega$ for all $g\in P$,
where $r^g$ denotes the principal right action of $g\in P$ on $\mathcal G$.
Moreover, $\omega$ reproduces the generators of the right $P$-action, i.e.\ for all $X\in\poe$, we have $\omega(\zeta_X)=X$ where $\zeta_X:=\frac d{dt}|_0r^{\exp(tX)}$ denotes the fundamental vector field.
The prototypical example of a parabolic geometry of type $(G,P)$ is its flat model, that is, the generalized flag variety $G/P$ with the canonical projection $G\to G/P$ and the Maurer--Cartan form on $G$.

Using the Cartan connection, the filtration \eqref{E:filt} provides a filtration of the tangent bundle $T\mathcal G$ by $P$-invariant subbundles, 
$$
T\mathcal G=T^{-k}\mathcal G\supseteq T^{-k+1}\mathcal G\supseteq\cdots\supseteq T^k\mathcal G,
$$
where $T^i_u\mathcal G:=\omega_u^{-1}(\goe^i)$ for $u\in\mathcal G$. 
Note that $T^0\mathcal G$ coincides with the vertical bundle of the projection $p\colon\mathcal G\to M$.
Hence there exists a unique filtration of $TM$ by subbundles,
\begin{equation}\label{E:TMfilt}
TM=T^{-k}M\supseteq T^{-k+1}M\supseteq\cdots\supseteq T^{-1}M,
\end{equation}
such that $(Tp)^{-1}(T^iM)=T^i\mathcal G$.
The Cartan connection induces an isomorphism
\begin{equation}\label{E:TMgp}
TM\cong\mathcal G\times_P(\goe/\poe)
\end{equation}
intertwining the filtration \eqref{E:TMfilt} with the filtration induced from \eqref{E:filt}.
Using \eqref{E:grgp} for the associated graded we obtain an isomorphism of vector bundles,
\begin{equation}\label{E:grTM}
\gr(TM)\cong\mathcal G_0\times_{G_0}\goe_-.
\end{equation}
Here $\mathcal G_0:=\mathcal G/P_+$ is considered as a principal $G_0$-bundle over $M$.

We assume that the Cartan connection $\omega$ is regular, see \cite[Definition~3.1.7]{CS09}.
Hence, the filtration on $TM$, see~\eqref{E:TMfilt}, turns $M$ into a filtered manifold, and the corresponding Levi bracket $\gr(TM)\otimes\gr(TM)\to\gr(TM)$ induced by the Lie bracket of vector fields coincides with the algebraic bracket induced by the Lie bracket $\goe_-\otimes\goe_-\to\goe_-$ via \eqref{E:grTM}.
In other words, using the notation from Section~\ref{SS:DO}, the Cartan connection of a regular parabolic geometry provides an isomorphism of bundles of graded nilpotent Lie algebras
\begin{equation}\label{E:tMG0}
\mathfrak tM\cong\mathcal G_0\times_{G_0}\goe_-.
\end{equation}

Recall that the Cartan connection induces a principal connection on $\mathcal P\times_PG$.
More precisely, there exists a unique principal connection on the principal $G$-bundle $\mathcal G\times_PG\to M$ which restricts to the Cartan connection $\omega$ along the inclusion 
$\mathcal G\subseteq\mathcal G\times_PG$, see \cite[Theorem~1.5.6]{CS09}.
Consequently, for every finite dimensional $G$-representation $\mathbb E$, the Cartan connection induces a linear connection $\nabla$ on the associated tractor bundle 
$$
E:=\mathcal G\times_P\mathbb E=(\mathcal G\times_PG)\times_G\mathbb E.
$$
Recall that $\mathbb E$ admits a grading, $\mathbb E=\mathbb E_{-l}\oplus\cdots\oplus\mathbb E_l$, which is compatible with the grading of $\goe$, i.e.\ $X\cdot v\in\mathbb E_{i+j}$ for all $X\in\goe_i$ and $X\in\mathbb E_j$.
Indeed, there exists a unique grading element in $\goe$ which acts by multiplication with $j$ on the component $\goe_j$, see \cite[Proposition~3.1.2(1)]{CS09}, and the eigenspaces of its action on $\mathbb E$ provide the desired decomposition.
The grading of $\mathbb E$ is $G_0$-invariant since the uniqueness of the the grading element implies that it is stabilized by $G_0$.
Hence, the associated filtration $\mathbb E^i:=\bigoplus_{j\geq i}\mathbb E_j$ is $P$-invariant, see \eqref{E:PG0iso}.
Moreover, $P_+$ acts trivially on the associated graded, and $\gr(\mathbb E)=\mathbb E$ as representations of $P/P_+=G_0$.
The $P$-invariant filtration of $\mathbb E$ induces a filtration of $E$ by subbundles $E^i:=\mathcal G\times_P\mathbb E^i$.
Clearly, the linear connection $\nabla$ is filtration-preserving, that is, for all $X\in\Gamma^\infty(T^pM)$ and $\psi\in\Gamma^\infty(E^q)$ we have $\nabla_X\psi\in\Gamma^\infty(E^{p+q})$.
Since the Cartan connection is assumed to be regular, its curvature $F^\nabla\in\Omega^2(M;\eend(E))$ is contained in filtration degree one, that is, for all $X_i\in\Gamma^\infty(T^{p_i}M)$ and $\psi\in\Gamma^\infty(E^q)$ we have $F^\nabla(X_1,X_2)\psi\in\Gamma^\infty(E^{p_1+p_2+q+1})$, see \cite[Corollary~3.1.8(2) and Theorem~3.1.22(3)]{CS09}.
The isomorphism \eqref{E:TMgp} induces an isomorphism
\begin{equation}\label{E:forms}
\Lambda^kT^*M\otimes E\cong\mathcal G\times_P(\Lambda^k(\goe/\poe)^*\otimes\mathbb E)
\end{equation}
which intertwines the filtration on $\Lambda^kT^*M\otimes E$ with the one induced from the filtration on $\Lambda^k(\goe/\poe)^*\otimes\mathbb E$.
Moreover, \eqref{E:grgp} provides an isomorphism of $G_0$-modules,
\begin{equation}\label{E:123}
\gr\bigl(\Lambda^k(\goe/\poe)^*\otimes\mathbb E\bigr)=C^k(\goe_-;\mathbb E),
\end{equation}
where $C^k(\goe_-;\mathbb E):=\Lambda^k\goe_-^*\otimes\mathbb E$. 
Hence, \eqref{E:forms} induces an isomorphism
\begin{equation}\label{E:grforms}
\gr\bigl(\Lambda^kT^*M\otimes E\bigr)
\cong\mathcal G_0\times_{G_0}C^k(\goe_-;\mathbb E).
\end{equation}
The extension $d^\nabla\colon\Omega^k(M;E)\to\Omega^{k+1}(M;E)$ characterized by the Leibniz rule, see \eqref{E:leibn}, is filtration-preserving, and via \eqref{E:grforms} we have
\begin{equation}\label{E:grd}
\gr(d^\nabla)=\mathcal G_0\times_{G_0}\partial_{\goe_-}
\end{equation}
where $\partial_{\goe_-}\colon C^k(\goe_-;\mathbb E)\to C^{k+1}(\goe_-;\mathbb E)$ denotes the differential in the standard complex computing Lie algebra cohomology $H^*(\goe_-;\mathbb E)$, see Lemma~\ref{L:curv}.

Let $\delta_{\poe_+}\colon\Lambda^k\poe_+\otimes\mathbb E\to\Lambda^{k-1}\poe_+\otimes\mathbb E$ denote the differential in the standard complex computing the Lie algebra homology $H_*(\poe_+;\mathbb E)$ with coefficients in $\mathbb E$. 
Since $\delta_{\poe_+}$ is $P$-equivariant, and since the Killing form provides an isomorphism of $P$-modules $(\goe/\poe)^*\cong\poe_+$, the differential $\delta_{\poe_+}$ dualizes to a $P$-equivariant map
\begin{equation}\label{E:del*gp}
\delta_{\goe/\poe}\colon\Lambda^k(\goe/\poe)^*\otimes\mathbb E\to\Lambda^{k-1}(\goe/\poe)^*\otimes\mathbb E.
\end{equation}
Via the identification \eqref{E:forms}, it gives rise to a vector bundle homomorphism,
\begin{equation}\label{E:deltaVV}
\delta\colon\Lambda^kT^*M\otimes E\to\Lambda^{k-1}T^*M\otimes E,\qquad\delta:=\mathcal G\times_P\delta_{\goe/\poe}.
\end{equation}
In the literature \cite{CSS01,CS12,CS09} this homomorphism is often denoted $\partial^*$.
Clearly, $\delta^2=0$. 
Moreover, $\delta$ is filtration-preserving and via \eqref{E:grforms}
\begin{equation}\label{E:grdel}
\gr(\delta)=\mathcal G_0\times_{G_0}\delta_{\goe_-},
\end{equation}
where $\delta_{\goe_-}\colon C^k(\goe_-;\mathbb E)\to C^{k-1}(\goe_-;\mathbb E)$ is obtained from \eqref{E:del*gp} by passing to the associated graded and using the identification \eqref{E:123}, that is, $\delta_{\goe_-}=\gr(\delta_{\goe/\poe})$.
Actually, $\delta_{\goe/\poe}$ is grading preserving, hence $\delta_{\goe_-}=\delta_{\goe/\poe}$ via the isomorphism of $G_0$-modules 
\begin{equation}\label{E:g-gp}
\Lambda^k\goe_-^*\otimes\mathbb E=\Lambda^k(\goe/\poe)^*\otimes\mathbb E
\end{equation}
induced by the identification $\goe_-=\goe/\poe$.

Recall that a Weyl structure is a $G_0$-equivariant section of the principal $P_+$-bundle $\mathcal G\to\mathcal G_0=\mathcal G/P_+$, see \cite[Definition~5.1.1]{CS09}.
Global Weyl structures always exist, see \cite[Proposition~5.1.1]{CS09}. Moreover, the (contractible) space of sections of the bundle of groups $\mathcal G_0\times_{G_0}P_+$
acts free and transitively on the space of Weyl structures.
Using the isomorphism of $G_0$-modules \eqref{E:g-gp}, every Weyl structure $\mathcal G_0\to\mathcal G$ induces a filtration-preserving isomorphism of vector bundles,
\begin{equation*}\label{E:sigma}
\mathcal G_0\times_{G_0}C^k(\goe_-;\mathbb E)\xrightarrow\cong\mathcal G\times_P\bigl((\Lambda^k(\goe/\poe)^*\otimes\mathbb E\bigr),
\end{equation*}
inducing the identity on the associated graded, see \eqref{E:123}.
Via \eqref{E:forms} and \eqref{E:grforms} this corresponds to a splitting of the filtration
$$
S\colon\gr(\Lambda^kT^*M\otimes E)\to\Lambda^kT^*M\otimes E
$$
satisfying $\delta\circ S=S\circ\gr(\delta)$.

Kostant \cite{K61} observed that $\delta_{\goe_-}$ and $\partial_{\goe_-}$ are adjoint with respect to positive definite inner products on the spaces $C^k(\goe_-;\mathbb E)$, see \cite[Proposition~3.1.1]{CS09}.
Hence the Laplacian
$$
\Box_{\goe_-}\colon C^*(\goe_-;\mathbb E)\to C^*(\goe_-;\mathbb E),\qquad
\Box_{\goe_-}:=\delta_{\goe_-}\circ\partial_{\goe_-}+\partial_{\goe_-}\circ\delta_{\goe_-},
$$
gives rise to a finite dimensional Hodge decomposition
\begin{equation}\label{E:fdhodge}
C^*(\goe_-;\mathbb E)=\img(\delta_{\goe_-})\oplus\ker(\Box_{\goe_-})\oplus\img(\partial_{\goe_-}).
\end{equation}
Using \eqref{E:grd} and \eqref{E:grdel} we see that $\Box\colon\Omega^*(M;E)\to\Omega^*(M;E)$, $\Box=\delta\circ d^\nabla+d^\nabla\circ\delta$, is filtration-preserving, and via the identification \eqref{E:grforms} we have
\begin{equation}\label{E:grBox}
\gr(\Box)=\mathcal G_0\times_{G_0}\Box_{\goe_-}.
\end{equation}
For the fiber-wise projection $\tilde P\colon\gr(\Lambda^kT^*M\otimes E)\to\gr(\Lambda^kT^*M\otimes E)$ onto the (generalized) zero eigenspace of $\gr(\Box)$, we obtain, via \eqref{E:grforms},
$$
\tilde P=\mathcal G_0\times_{G_0}P_{\goe_-}
$$
where $P_{\goe_-}\colon C^*(\goe_-;\mathbb E)\to C^*(\goe_-;\mathbb E)$ denotes the projection onto $\ker(\Box_{\goe_-})$ along the decomposition \eqref{E:fdhodge}.
In particular, $\gr(\delta)\circ\tilde P=0$.

From the discussion above we conclude that the homomorphism $\delta$, see \eqref{E:deltaVV}, is a Kostant type codifferential of maximal rank for the linear connection $\nabla$ on the tractor bundle $E$, see Definition~\ref{D:Kdelta} and Remark~\ref{R:Kdeltabound}.
In this situation Corollary~\ref{C:D}(a) reduces to the statement in \cite[Theorem~2.4]{CS12}, see also \cite{CSS01,CD01}.
Since $P_+$ acts trivially on $H_*(\poe_+;\mathbb E)$ the $P$-action on $H_*(\poe_+;\mathbb E)$ factors to an action by $P/P_+=G_0$ and we have canonical identifications:
$$
\mathcal H_*
=\mathcal G_0\times_{G_0}H_*(\poe_+;\mathbb E)
=\mathcal G_0\times_{G_0}\ker(\Box_{\goe_-})
=\mathcal G_0\times_{G_0}H^*(\goe_-;\mathbb E).
$$
The corresponding sequence of differential operators in Corollary~\ref{C:D}(b)
coincides with one version of (curved) BGG sequences that can be found in the literature.
This is called ``torsion free BGG sequence'' in \cite[Section~5]{CD01} and coincides with the sequence constructed in \cite[Section~2.4]{CS12}.
For torsion free parabolic geometries, see \cite[Section~1.5.7]{CS09}, this coincides with the original curved BGG sequence constructed by \v Cap, Slov\'ak and Sou\v cek in \cite{CSS01}.
From Corollary~\ref{C:D}(b) we thus obtain

\begin{corollary}\label{C:BGG}
The (torsion free) BGG sequence associated to a regular parabolic geometry of type $(G,P)$ and a finite dimensional $G$-representation is a graded Rockland sequence of differential operators which have graded Heisenberg order at most zero.
\end{corollary}

Kostant's version of the Bott--Borel--Weil theorem permits us to effectively compute the homologies $H_k(\poe_+;\mathbb E)$ as modules over $\goe_0$.
More precisely, $H_k(\poe_+;\mathbb E)$ decomposes as a direct sum of irreducible $\goe_0$-modules whose dominant weights can be read off the Hasse diagram of $\poe$, see \cite{K61} or \cite[Theorem~3.3.5 and Proposition~3.3.6]{CS09}.
Moreover, the grading element acts as a scalar on each of these irreducible components which can easily be computed too, see \cite[Section~3.2.12]{CS09}.
Consequently, representation theory permits us to determine the decomposition according to the grading $H_k(\poe_+;\mathbb E)=\bigoplus_pH_k(\poe_+;\mathbb E)_p$.
Decomposing the BGG operators accordingly, $\bar D_k=\sum_{p,q}(\bar D_k)_{qp}$, with
\begin{equation}\label{E:Dkqp}
\Gamma^\infty\bigl(\mathcal G_0\times_{G_0}H_k(\poe_+;\mathbb E)_p\bigr)
\xrightarrow{(\bar D_k)_{qp}}
\Gamma^\infty\bigl(\mathcal G_0\times_{G_0}H_{k+1}(\poe_+;\mathbb E)_q\bigr),
\end{equation}
one part of Corollary~\ref{C:BGG} asserts that the differential operator in \eqref{E:Dkqp} is of Heisenberg order at most $q-p$, a statement which appears to be well known.
Using a frame $(\mathcal G_0)_x\cong G_0$ at $x\in M$, the Heisenberg principal symbol of the operator \eqref{E:Dkqp} at $x$ can be regarded as a left invariant differential operator which is homogeneous of order $q-p$,
$$
C^\infty\bigl(\mathcal T_xM,H_k(\poe_+;\mathbb E)_p\bigr)\xrightarrow{\sigma^{q-p}_x((\bar D_k)_{qp})}C^\infty\bigl(\mathcal T_xM,H_{k+1}(\poe_+;\mathbb E)_q\bigr).
$$
The second part of Corollary~\ref{C:BGG} asserts that these Heisenberg principal symbols combine to form a sequence of left invariant differential operators
\begin{equation}\label{E:tsBGG}
\cdots\to C^\infty\bigl(\mathcal T_xM,H_k(\poe_+;\mathbb E)\bigr)
\xrightarrow{\tilde\sigma^0_x(\bar D_k)}
C^\infty\bigl(\mathcal T_xM,H_{k+1}(\poe_+;\mathbb E)\bigr)\to\cdots
\end{equation}
where $\tilde\sigma_x^0(\bar D_k)=\sum_{p,q}\sigma^{q-p}_x((\bar D_k)_{qp})$, which is Rockland in the sense that it becomes exact in every non-trivial irreducible unitary representation of $\mathcal T_xM$.
Up to the isomorphism $\mathcal T_xM\cong G_-:=\exp(\goe_-)$ provided by the frame, the graded Heisenberg principal symbol of a BGG operator in \eqref{E:tsBGG} coincides with the corresponding BGG operator on the flat model $G/P$ restricted along the local diffeomorphism $G_-\to G/P$ obtained form the inclusion $G_-\subseteq G$.

If the homology $H_k(\poe_+;\mathbb E)$ is concentrated in a single degree for each $k$, that is, if there exist numbers $p_k$ such that $H_k(\poe_+;\mathbb E)=H_k(\poe_+;\mathbb E)_{p_k}$, then the corresponding BGG operator $\bar D_k\colon\Gamma^\infty(\mathcal H_k)\to\Gamma^\infty(\mathcal H_{k+1})$ is of Heisenberg order at most $p_{k+1}-p_k$ and the BGG sequence is Rockland in the ungraded sense, see Definition~\ref{def.Hypo-seq}.
If, moreover, $p_{k+1}-p_k\geq1$, then the analytic results established in the preceding sections are applicable, see, in particular, Corollaries~\ref{C:RShypo}, \ref{C:Hodge}, \ref{C:regrockseq}, and \ref{C:HsHodge-seq}.
Below we will discuss a classical example of this type.

\begin{example}[Generic rank two distributions in dimension five]\label{Ex:BGG235}
Let $M$ be a 5-manifold equipped with a rank two distribution of Cartan type, $T^{-1}M\subseteq TM$, see \cite{C10,BH93,S08}.
Hence, $T^{-1}M$ is a rank two subbundle of $TM$ with growth vector $(2,3,5)$, that is, Lie brackets of sections of $T^{-1}M$ span a rank three subbundle $T^{-2}M$ of $TM$ and triple brackets of sections of $T^{-1}M$ span all of $TM$.
These geometric structures are also known as generic rank two distributions in dimension five, see \cite{S08,CS09a}.
The topological obstructions to global existence of such a distribution are well understood in the orientable case, see \cite[Theorem~1]{DH16}.
On open 5-manifolds, Gromov's h-principle is applicable and establishes existence, once the topological requirements are met, see \cite[Theorem~2]{DH16}.
It is unclear, however, if there are further geometric obstructions on closed 5-manifolds.
Whether rank two distributions of Cartan type also abide by an h-principle on closed manifolds, appears to be an intriguing open question and is a major motivation for our investigation of hypoelliptic sequences.

In his celebrated paper \cite{C10} Cartan has shown that, up to isomorphism, there exists a unique regular normal parabolic geometry of type $(G,P)$ on $M$ with underlying filtration:
$$
TM=T^{-3}M\supseteq T^{-2}M\supseteq T^{-1}M\supseteq T^0M=0.
$$
Here $G$ denotes the split real form of the exceptional Lie group $G_2$ and $P$ denotes the maximal parabolic subgroup corresponding to the shorter simple root.
Hence, every finite dimensional representation $\mathbb E$ of $G$ gives rise to a curved BGG sequence on $M$,
\begin{equation}\label{E:BGG235}
\Gamma^\infty(\mathcal H_0)\xrightarrow{\bar D_0}
\Gamma^\infty(\mathcal H_1)\xrightarrow{\bar D_1}
\Gamma^\infty(\mathcal H_2)\xrightarrow{\bar D_2}
\Gamma^\infty(\mathcal H_3)\xrightarrow{\bar D_3}
\Gamma^\infty(\mathcal H_4)\xrightarrow{\bar D_4}
\Gamma^\infty(\mathcal H_5),
\end{equation}
where $\mathcal H_k:=\mathcal G_0\times_{G_0}H_k(\poe_+;\mathbb E)$.
We label the longer simple root of the $G_2$ root system by $\alpha_1$ and let $\alpha_2$ denote the shorter simple root.
Hence, $2\alpha_1+3\alpha_2$ is the highest root, and the corresponding fundamental weights are $\lambda_1=2\alpha_1+3\alpha_2$ and $\lambda_2=\alpha_1+2\alpha_2$.
Suppose $\mathbb E$ is the irreducible complex representation with highest weight $a\lambda_1+b\lambda_2$ where $a,b\in\N_0$.
Then $H_k(\poe_+;\mathbb E)$ is an irreducible complex module of $\goe_0\subseteq\goe_0^\C\cong\mathfrak g\mathfrak l_2(\C)$ which can readily be determined by working out the Hasse diagram of $\poe^\C$, see \cite[Section 3.2.16]{CS09}, and using Kostant's version of the Bott--Borel--Weil theorem, see \cite[Theorem~3.3.5 and Proposition~3.3.6]{CS09}.
Denoting the highest weight of $H_k(\poe_+;\mathbb E)$ by $a_k\lambda_1+b_k\lambda_2$, we obtain the first three columns in the following table:
\footnote{As formulated in \cite[Theorem~3.3.5]{CS09}, Kostant's version of the Bott--Borel--Weil theorem computes the cohomology $H^k(\poe_+;\mathbb E)$.
Using the following facts, this permits working out the homology $H_k(\poe_+;\mathbb E)$ as well:
$H^k(\poe_+;\mathbb E^*)\cong H_k(\poe_+;\mathbb E)^*$ as $\goe_0$-modules; 
$\mathbb E\cong\mathbb E^*$ as $\goe$-modules; 
If $W$ is an irreducible $\goe_0$ module with highest weight $a\lambda_1+b\lambda_2$, then $W^*$ is an irreducible $\goe_0$ module with highest weight $a'\lambda_1+b'\lambda_2$, where $a'=a$ and $b'=-3a-b$.}
\begin{equation}\label{E:G2table}
\begin{array}{r||r|r||r|r}
k & a_k & b_k & \dim H_k(\poe_+;\mathbb E) & p_k
\\\hline\hline
0 & a      & -3a-b    & a+1    & -3a-2b  
\\
1 & a+b+1  & -3a-2b-1 & a+b+2  & -3a-b+1 
\\
2 & 2a+b+2 & -3a-2b-1 & 2a+b+3 & -b+4
\\
3 & 2a+b+2 & -3a-b    & 2a+b+3 & b+6
\\
4 & a+b+1  & -b+3     & a+b+2  & 3a+b+9
\\
5 & a      & b+5      &  a+1   & 3a+2b+10
\end{array}
\end{equation}
The grading element acts by multiplication with the scalar $p_k=3a_k+2b_k$ on $H_k(\poe_+;\mathbb E)$, see \cite[Section~3.2.12]{CS09}, whence the last column.
Moreover, the highest weight of $H_k(\poe_+;\mathbb E)$ considered as $\mathfrak s\mathfrak l_2(\C)$-module is $a_k$ times the fundamental weight of $\mathfrak s\mathfrak l_2(\C)$, hence $\dim H_k(\poe_+;\mathbb E)=a_k+1$, whence the remaining column.
Since $\bar D_k$ is of Heisenberg order $p_{k+1}-p_k$, we conclude that $\bar D_0$ and $\bar D_4$ are of Heisenberg order $b+1$; $\bar D_1$ and $\bar D_3$ are of Heisenberg order $3(a+1)$; and $\bar D_2$ is of Heisenberg order $2(b+1)$.
According to Corollary~\ref{C:BGG} the sequence \eqref{E:BGG235} is Rockland in the ungraded sense, see Definition~\ref{def.Hypo-seq}.
In particular, the differential operators
$\bar D_0^*\bar D_0$, 
$(\bar D_0\bar D_0^*)^{3(a+1)}+(\bar D_1^*\bar D_1)^{b+1}$, 
$(\bar D_1\bar D_1^*)^{2(b+1)}+(\bar D_2^*\bar D_2)^{3(a+1)}$,
$(\bar D_2\bar D_2^*)^{3(a+1)}+(\bar D_3^*\bar D_3)^{2(b+1)}$,
$(\bar D_3\bar D_3^*)^{b+1}+(\bar D_4^*\bar D_4)^{3(a+1)}$, and
$\bar D_4\bar D_4^*$ 
are all hypoelliptic and maximal hypoelliptic estimates are available, see Lemma~\ref{L:rockseq}, Theorem~\ref{T:Rockland}, as well as, Corollaries~\ref{C:PsiinvA}, \ref{C:reg}, and \ref{C:HsHodge}.
Here $\bar D_k^*$ denotes the formal adjoint of $\bar D_k$ with respect to any fiber-wise Hermitian metrics on the vector bundles $\mathcal H_k$ and any volume density on $M$.

Let us finally put down explicit formulas for the Heisenberg principal symbol of the BGG operators corresponding to the trivial representation $\mathbb E$.
We consider these operators as left invariant differential operators on the simply connected nilpotent Lie group $G_-$.
According to the discussion above, this  BGG sequence has the form
\begin{equation}\label{E:BGG235dR}
C^\infty(G_-)\xrightarrow{D_0}
C^\infty(G_-)^2\xrightarrow{D_1}
C^\infty(G_-)^3\xrightarrow{D_2}
C^\infty(G_-)^3\xrightarrow{D_3}
C^\infty(G_-)^2\xrightarrow{D_4}
C^\infty(G_-)
\end{equation}
where $D_0$ and $D_4$ are homogeneous of degree $1$; $D_1$ and $D_3$ are homogeneous of degree $3$; and $D_2$ is homogeneous of degree $2$, see also \cite{BEGN19}.
Using matrices with entries in the universal enveloping algebra of $\goe_-$ these operators can be expressed as:

$$
D_0=\left(\begin{array}{c}
X_1\\
X_2
\end{array}\right)
$$

$$
D_1=\left(\begin{array}{cc}
-X_4-X_{112}-X_{13}&X_{111}\\
-X_5-X_{122}&X_{112}-2X_{13}\\
-X_{222}&X_{122}-3X_{23}
\end{array}\right)
$$

$$
D_2=\left(\begin{array}{ccc}
-X_{12}-X_3&X_{11}& 0\\         
-X_{22}&-3X_3&  X_{11}\\ 
0&        -X_{22}&X_{12}-2X_3\\
\end{array}\right)
$$

$$
D_3=\left(\begin{array}{ccc}
X_{122}+X_{23}-2X_5&-X_{112}+X_4&X_{111}\\
X_{222}&-X_{122}+2X_{32}&X_{112}-3X_{13}+3X_4
\end{array}\right)
$$

$$
D_4=\left(\begin{array}{cc}
-X_2&X_1
\end{array}\right)
$$
Here $X_5,X_4|X_3|X_2,X_1$ is a graded basis of $\goe_-=\goe_{-3}\oplus\goe_{-2}\oplus\goe_{-1}$ such that
$$
[X_1,X_2]=X_3,\quad[X_1,X_3]=X_4,\quad[X_2,X_3]=X_5.
$$
The vertical bars above indicate that $X_1,X_2$ is a basis of $\goe_{-1}$, $X_3$ is a basis of $\goe_{-2}$, and $X_4,X_5$ is a basis of $\goe_{-3}$.
Moreover, we use the notation $X_{i_1\dotsc i_k}=X_{i_1}\cdots X_{i_k}$.
These formulas are derived in \cite[Appendix~B]{DH17}.
\end{example}

\section{Graded hypoelliptic analysis}\label{S:grRockland}

In this section we adapt the analysis discussed in Section~\ref{S:PDO} to the filtered setup required to deal with the sequences constructed in Section~\ref{S:grhesequences}.
Everything generalizes effortlessly, but one bit: Formal adjoints of graded (pseudo)differential operators are in general only available if the underlying manifold is closed.
This is related to the fact that we can construct invertible $\Lambda_s\in\Psi^s(E)$ with $\Lambda_s^{-1}\in\Psi^{-s}(E)$ only on closed manifolds, see Lemma~\ref{L:Lambda} and \eqref{E:Asharp} below.

\subsection{Graded pseudodifferential operators}\label{SS:gradedhypo}

The concept of graded Heisenberg order for differential operators introduced in Section~\ref{SS:fVBDO} can be generalized to pseudodifferential operators in a straight forward manner as follows.
Let $E$ and $F$ be two filtered vector bundles over a filtered manifold $M$, suppose $A\in\mathcal O(E,F)$, and let $s$ be a complex number.
Choose splittings of the filtrations, $S_E\colon\gr(E)\to E$ and $S_F\colon\gr(F)\to F$ and decompose the operator accordingly, $S_F^{-1}AS_E=\sum_{q,p}(S_F^{-1}AS_E)_{q,p}$, where $(S_F^{-1}AS_E)_{q,p}\in\mathcal O(\gr_p(E),\gr_q(F))$.
We say $A$ has \emph{graded Heisenberg order} $s$ if $(S_F^{-1}AS_E)_{q,p}\in\Psi^{s+q-p}(\gr_p(E),\gr_q(F))$ for all $p$ and $q$.
We let $\tilde\Psi^s(E,F)$ denote the space of pseudodifferential operators of graded Heisenberg order $s$.
One readily checks that this space does not depend on the choice of splittings $S_E$ and $S_F$.

Let us define the \emph{space of principal cosymbols of graded order $s$} by
$$
\tilde\Sigma^s(E,F)
:=\left\{k\in\frac{\mathcal K(\mathcal TM;\gr(E),\gr(F))}{\mathcal K^\infty(\mathcal TM;\gr(E),\gr(F))}:\textrm{$(\delta_\lambda)_*k=\lambda^s\delta_\lambda^Fk\delta^E_{1/\lambda}$ for all $\lambda>0$}\right\},
$$
where $\delta^E_\lambda\in\Aut(\gr(E))$ denotes the automorphism given by multiplication with $\lambda^p$ on the grading component $\gr_p(E)$.
These are essentially homogeneous kernels in a graded sense, taking the grading on $\gr(E)$ and $\gr(F)$ into account.
They can be canonically identified with matrices of ordinary principal cosymbols,
\begin{equation}\label{E:tSigma}
\tilde\Sigma^s(E,F)=\bigoplus_{q,p}\Sigma^{s+q-p}(\gr_p(E),\gr_q(F)).
\end{equation}

For $A\in\tilde\Psi^s(E,F)$ we define the \emph{graded Heisenberg principal cosymbol} $\tilde\sigma^s(A)\in\tilde\Sigma^s(E,F)$ by $\tilde\sigma^s(A):=\sum_{p,q}\sigma^{s+q-p}\bigl((S_F^{-1}AS_E)_{q,p}\bigr)$ where $\sigma^{s+q-p}((S_F^{-1}AS_E)_{q,p})\in\Sigma^s(\gr_p(E),\gr_q(F))$ are the Heisenberg principal symbols of the components.
One readily checks that the graded principal Heisenberg cosymbol is independent of the choice of splittings $S_E$ and $S_F$.
From Proposition~\ref{P:Psi}\itemref{P:Psi:symbsequ} we immediately obtain a short exact sequence:
$$
0\to\tilde\Psi^{s-1}(E,F)\to\tilde\Psi^s(E,F)\xrightarrow{\tilde\sigma^s}\tilde\Sigma^s(E,F)\to0.
$$
If $A\in\tilde\Psi^s(E,F)$ and $B\in\tilde\Psi^r(F,G)$, then $BA\in\tilde\Psi^{r+s}(E,G)$ and
\begin{equation}\label{E:grsigmaAB}
\tilde\sigma^{r+s}(BA)=\tilde\sigma^r(B)\tilde\sigma^s(A),
\end{equation}
provided at least one operator is properly supported.
Moreover, $A^t\in\tilde\Psi^s(F',E')$ and
\begin{equation}\label{E:grsigmaAt}
\tilde\sigma^s(A^t)=\tilde\sigma^s(A)^t.
\end{equation}
These two properties follow immediately from the corresponding statements in the ungraded case, see Proposition~\ref{P:Psi}\itemref{P:Psi:mult}\&\itemref{P:Psi:trans}.
Recall that the bundle $E'=E^*\otimes|\Lambda|_{M}$ is equipped with the dual filtration as explained in Section~\ref{SS:fVBDO}.

For trivially filtered vector bundles these concepts clearly reduce to the ungraded case discussed in Section~\ref{SS:calculus}.
Moreover, for differential operators we recover the graded Heisenberg order and graded Heisenberg symbol from Section~\ref{SS:fVBDO}.
More precisely, for every non-negative integer $k$, we have $\DO(E,F)\cap\tilde\Psi^k(E,F)=\widetilde{\DO}^k(E,F)$, and the graded Heisenberg principal symbol from Section~\ref{SS:fVBDO} coincides with principal Heisenberg cosymbol introduced in this section via the canonical inclusion
\begin{equation}\label{E:incUSgr}
\bigl(\mathcal U(\mathfrak tM)\otimes\hom(\gr(E),\gr(F))\bigr)_{-k}\subseteq\tilde\Sigma^k(E,F),
\end{equation}
see Proposition~\ref{P:Psi}\itemref{P:Psi:DO} and \eqref{E:incUS}.

\begin{lemma}\label{L:grLambda}
Let $E$ be a filtered vector bundle over a filtered manifold $M$, and let $\mathbf E$ denote the same vector bundle equipped with the trivial filtration, $\mathbf E=\mathbf E^0\supseteq\mathbf E^1=0$.
For every complex number $s$ there exist $\tilde\Lambda_s\in\tilde\Psi_\prop^s(E,\mathbf E)$ and $\tilde\Lambda_s'\in\tilde\Psi^{-s}_\prop(\mathbf E,E)$ such that $\tilde\Lambda_s\tilde\Lambda_s'-\id$ and $\tilde\Lambda_s'\tilde\Lambda_s-\id$ are both smoothing operators.
Moreover, these operators may be chosen such that $\tilde\Lambda_s\colon\Gamma^{-\infty}_c(E)\to\Gamma^{-\infty}_c(\mathbf E)$ and $\tilde\Lambda_s'\colon\Gamma^{-\infty}_c(\mathbf E)\to\Gamma^{-\infty}_c(E)$ are injective.
On a closed manifold these operators may even be chosen such that $\tilde\Lambda_s\tilde\Lambda_s'=\id$ and $\tilde\Lambda_s'\tilde\Lambda_s=\id$.
\end{lemma}

\begin{proof}
Choose a splitting of the filtration, $S\colon\gr(E)\to E$.
Let $\Lambda_{s-p}\in\Psi^{s-p}_\prop(\gr_p(E))$ and $\Lambda_{s-p}'\in\Psi^{-(s-p)}_\prop(\gr_p(E))$ be as in Lemma~\ref{L:Lambda}.
Then the operators
$$
\tilde\Lambda_s:=S\bigl(\textstyle\bigoplus_p\Lambda_{s-p}\bigr)S^{-1}
\qquad\text{and}\qquad
\tilde\Lambda_s':=S\bigl(\textstyle\bigoplus_p\Lambda'_{s-p}\bigr)S^{-1}
$$
have the desired properties.
\end{proof}

\subsection{Graded Heisenberg Sobolev scale}\label{SS:grsobolev}

Let $E$ be a filtered vector bundle over a filtered manifold $M$.
For each real number $s$ we let $\tilde H^s_\loc(E)$ denote the space of all distributional sections $\psi\in\Gamma^{-\infty}(E)$ such that $A\psi\in L^2_\loc(\mathbf F)$ for all $A\in\tilde\Psi^s_\prop(E,\mathbf F)$ and all trivially filtered vector bundles $\mathbf F$ over $M$, that is, $\mathbf F=\mathbf F^0\supseteq\mathbf F^1=0$.
We equip $\tilde H^s_\loc(E)$ with the coarsest topology such that $A\colon\tilde H^s_\loc(E)\to L^2_\loc(\mathbf F)$ is continuous for all such $A\in\tilde\Psi^s_\prop(E,\mathbf F)$.
Similarly, we let $\tilde H^s_c(E)$ denote the space of all compactly supported distributional sections $\psi\in\Gamma_c^{-\infty}(E)$ such that $A\psi\in L^2_\loc(\mathbf F)$ for all $A\in\tilde\Psi^s(E,\mathbf F)$ and all trivially filtered vector bundles $\mathbf F$ over $M$.
We equip $\tilde H^s_c(E)$ with the coarsest topology such that $A\colon\tilde H^s_c(E)\to L^2_\loc(\mathbf F)$ is continuous for all such $A\in\tilde\Psi^s(E,\mathbf F)$.
We will refer to these spaces as \emph{graded Heisenberg Sobolev spaces}.
Any splitting of the filtration on $E$ gives rise to non-canonical topological isomorphisms
$$
\tilde H^s_\loc(E)\cong\bigoplus_p H^{s-p}_\loc\bigl(\gr_p(E)\bigr)
\qquad\text{and}\qquad
\tilde H^s_c(E)\cong\bigoplus_pH^{s-p}_c\bigl(\gr_p(E)\bigr).
$$

Generalizing Proposition~\ref{P:Hs}\itemref{P:Hs:locfilt}\&\itemref{P:Hs:cfilt}, we have continuous inclusions
$$
\Gamma^\infty(E)\subseteq\tilde H^{s_2}_\loc(E)\subseteq\tilde H^{s_1}_\loc(E)\subseteq\Gamma^{-\infty}(E)
$$
and
$$
\Gamma^\infty_c(E)\subseteq\tilde H^{s_2}_c(E)\subseteq\tilde H^{s_1}_c(E)\subseteq\Gamma_c^{-\infty}(E)
$$
for all real numbers $s_1\leq s_2$.
If $F$ is another filtered vector bundle, then each $A\in\tilde\Psi^k(E,F)$ induces continuous operators $A\colon\tilde H^s_c(E)\to\tilde H_\loc^{s-\Re(k)}(F)$ for all real $s$, cf.\ Proposition~\ref{P:Hs}\itemref{P:Hs:operators}.
As in Proposition~\ref{P:Hs}\itemref{P:Hs:pairing}, the canonical pairing $\Gamma^\infty_c(E')\times\Gamma^\infty(E)\to\C$ extends to a pairing
$$
\tilde H^{-s}_c(E')\times\tilde H^s_\loc(E)\to\C
$$
inducing linear bijections $\tilde H^s_\loc(E)^*=\tilde H^{-s}_c(E')$ and $\tilde H^{-s}_c(E')^*=\tilde H^s_\loc(E)$.
If, moreover, $M$ is closed, then $\tilde H^s_c(E)=\tilde H^s_\loc(E)$ is a Hilbert space we denote by $\tilde H^s(E)$, and the pairing induces an isomorphism of Hilbert spaces, $\tilde H^s(E)^*=\tilde H^{-s}(E')$.
This can all be proved as in Proposition~\ref{P:Hs} using Lemma~\ref{L:grLambda}.

Suppose $M$ is closed.
Fix a smooth volume density on $M$ and a smooth fiber-wise Hermitian metric $h$ on $\mathbf E$.
Moreover, let $s$ be a real number, choose invertible $\tilde\Lambda_s\in\tilde\Psi^s(E,\mathbf E)$ with inverse $\tilde\Lambda_s^{-1}\in\tilde\Psi^{-s}(\mathbf E,E)$, see Lemma~\ref{L:grLambda}, and consider the associated Hermitian inner product, cf.~\eqref{E:llrr},
\begin{equation}\label{E:grHSllrr}
\llangle\psi_1,\psi_2\rrangle_{\tilde H^s(E)}
:=\llangle\tilde\Lambda_s\psi_1,\tilde\Lambda_s\psi_2\rrangle_{L^2(\mathbf E)}
=\langle\tilde\Lambda_s^t(h\otimes dx)\tilde\Lambda_s\psi_1,\psi_2\rangle
\end{equation}
where $\psi_1,\psi_2\in\Gamma^\infty(E)$.
In the expression on the right hand side $h\otimes dx\colon\bar{\mathbf E}\to\mathbf E'$ is considered as a vector bundle isomorphism, $\tilde\Lambda_s^t\in\tilde\Psi^s(\mathbf E',E')$, and $\langle-,-\rangle$ denotes the canonical pairing for sections of $E$.
The sesquilinear form in~\eqref{E:grHSllrr} extends to an inner product generating the Hilbert space topology on the graded Heisenberg Sobolev space $\tilde H^s(E)$.

With respect to inner products on $\tilde H^{s_1}(E)$ and $\tilde H^{s_2}(F)$ as above, every $A\in\tilde\Psi^k(E,F)$ admits a formal adjoint, $A^\sharp\in\tilde\Psi^{\bar k+2(s_2-s_1)}(F,E)$ such that
\begin{equation}\label{E:Asharpllrr}
\llangle A^\sharp\phi,\psi\rrangle_{\tilde H^{s_1}(E)}=\llangle\phi,A\psi\rrangle_{\tilde H^{s_2}(F)}
\end{equation}
for all $\psi\in\Gamma^\infty(E)$ and $\phi\in\Gamma^\infty(F)$.
Indeed,
\begin{equation}\label{E:Asharp}
A^\sharp
=\tilde\Lambda_{E,s_1}^{-1}(\tilde\Lambda_{F,s_2}A\tilde\Lambda_{E,s_1}^{-1})^*\tilde\Lambda_{F,s_2}
=(\tilde\Lambda_{E,s_1}^t(h_E\otimes dx)\tilde\Lambda_{E,s_1})^{-1}\,A^t\,\tilde\Lambda_{F,s_2}^t(h_F\otimes dx)\tilde\Lambda_{F,s_2}.
\end{equation}
In the first expression the star denotes the adjoint of $\tilde\Lambda_{F,s_2}A\tilde\Lambda_{E,s_1}^{-1}\in\Psi^{k+s_2-s_1}(\mathbf E,\mathbf F)$ with respect to the $L^2$ inner products associated with the fiber-wise Hermitian metrics $h_E$ and $h_F$ and the volume density $dx$.
One readily verifies:
\begin{equation}\label{E:BAsharp}
(BA)^\sharp=A^\sharp B^\sharp
\qquad\text{and}\qquad
(A^\sharp)^\sharp=A.
\end{equation}

\subsection{Graded Rockland operators}

Let $E$ and $F$ be filtered vector bundles over a filtered manifold $M$.
To formulate the graded Rockland condition for operators in $\tilde\Psi^s(E,F)$, we begin by extending the definition of $\bar\pi(a)$ to graded cosymbols $a\in\tilde\Sigma^s_x(E,F)$ at $x\in M$, where $\pi\colon\mathcal T_xM\to U(\mathcal H)$ is a non-trivial irreducible unitary representation of the osculating group:
Write $a=\sum_{p,q}a_{p,q}$ according to the decomposition \eqref{E:tSigma} with $a_{q,p}\in\Sigma^{s+q-p}_x(\gr_p(E),\gr_q(F))$, and put $\bar\pi(a):=\sum_{p,q}\bar\pi(a_{q,p})$ where $\bar\pi(a_{q,p})$ denotes the unbounded operator from $\mathcal H\otimes\gr_p(E_x)$ to $\mathcal H\otimes\gr_q(F_x)$ described in Section~\ref{SS:para}.
Hence, $\bar\pi(a)$ is an unbounded operator form $\mathcal H\otimes\gr(E_x)$ to $\mathcal H\otimes\gr(F_x)$.
Moreover, the subspace $\mathcal H_\infty\otimes\gr(E_x)$ is contained in the domain of definition and mapped into $\mathcal H_\infty\otimes\gr(F_x)$.
From \eqref{E:piab} we immediately obtain 
\begin{equation}\label{E:grpiab}
\bar\pi(ba)=\bar\pi(b)\bar\pi(a)
\end{equation}
for all $a\in\tilde\Sigma^s_x(E,F)$ and $b\in\tilde\Sigma^{s'}_x(F,G)$.
For trivially filtered vector bundles, this clearly specializes to the definition in Section~\ref{SS:para}.
If $k$ is a non-negative integer and, see~\eqref{E:incUSgr}, $a\in\bigl(\mathcal U(\mathfrak t_xM)\otimes\hom(\gr(E_x),\gr(F_x))\bigr)_{-k}\subseteq\tilde\Sigma^k_x(E,F)$ then, on $\mathcal H_\infty\otimes\gr(E_x)$, the operator $\bar\pi(a)$ coincides with $\pi(a)$ considered in Section~\ref{SS:fVBDO}, cf.\ Definition~\ref{D:graded_hypoelliptic_seq}.

Generalizing Definition~\ref{D:Rockland} to the graded situation we have:

\begin{definition}[Graded Rockland condition]\label{D:graded-Rockland}
Let $E$ and $F$ be filtered vector bundles over a filtered manifold $M$.
A graded principal cosymbol $a\in\tilde\Sigma_x^s(E,F)$ at $x\in M$ is said to satisfy the \emph{graded Rockland condition} if, for every non-trivial irreducible unitary representation $\pi\colon\mathcal T_xM\to U(\mathcal H)$, the unbounded operator $\bar\pi(a)$ is injective on $\mathcal H_\infty\otimes\gr(E_x)$.
An operator $A\in\tilde\Psi^s(E,F)$ is said to satisfy the \emph{graded Rockland condition} if its graded principal cosymbol, $\tilde\sigma^s_x(A)\in\tilde\Sigma^s_x(E,F)$, satisfies the graded Rockland condition at each point $x\in M$.
\end{definition}

We obtain the following generalization of Theorem~\ref{T:Rockland} and Corollary~\ref{C:reg}.

\begin{corollary}[Left parametrix and graded regularity]\label{C:graded-regularity}
Let $E$ and $F$ be two filtered vector bundles over a filtered manifold $M$, let $k$ be a complex number, and suppose $A\in\tilde\Psi^k(E,F)$ satisfies the graded Rockland condition.
Then there exists a properly supported left parametrix $B\in\tilde\Psi^{-k}_\prop(F,E)$ such that $BA-\id$ is a smoothing operator.
In particular, $A$ is hypoelliptic. 
More precisely, if $\psi\in\Gamma^{-\infty}_c(E)$ and $A\psi\in\tilde H^{r-\Re(k)}_\loc(F)$, then $\psi\in\tilde H^r_c(E)$.
If, moreover, $M$ is closed, then $\ker(A)$ is a finite dimensional subspace of\/ $\Gamma^\infty(E)$, and for every $r'\leq r$ there exists a constant $C=C_{A,r,r'}\geq0$ such that the maximal graded hypoelliptic estimate
\begin{equation}\label{E:grmhesti}
\|\psi\|_{\tilde H^r(E)}\leq C\left(\|\psi\|_{\tilde H^{r'}(E)}+\|A\psi\|_{\tilde H^{r-\Re(k)}(F)}\right)
\end{equation}
holds for all $\psi\in\tilde H^r(E)$.
Here we are using any norms generating the Hilbert space topologies on the corresponding graded Heisenberg Sobolev spaces.
Moreover, if $Q$ denotes the orthogonal projection, with respect to an inner product of the form \eqref{E:grHSllrr}, onto the (finite dimensional) subspace $\ker(A)\subseteq\Gamma^\infty(E)$, then there exists a constant $C=C_{A,r,s}\geq0$ such that the maximal graded hypoelliptic estimate
\begin{equation}\label{E:grmheestiQ}
\|\psi\|_{\tilde H^r(E)}\leq C\left(\|Q\psi\|+\|A\psi\|_{\tilde H^{r-\Re(k)}(F)}\right)
\end{equation}
holds for all $\psi\in\tilde H^r(E)$. Here $\|-\|$ denotes any norm on $\ker(A)$.
\end{corollary}

\begin{proof}
Let $\mathbf E$ and $\mathbf F$ denote the vector bundles $E$ and $F$ equipped with the trivial filtrations, respectively, that is to say, $\mathbf E=\mathbf E^0\supseteq\mathbf E^1=0$ and $\mathbf F=\mathbf F^0\supseteq\mathbf F^1=0$.
According to Lemma~\ref{L:grLambda} there exist $\tilde\Lambda_E\in\tilde\Psi^0(E,\mathbf E)$, $\tilde\Lambda_E'\in\tilde\Psi^0(\mathbf E,E)$, $\tilde\Lambda_F\in\tilde\Psi^0(F,\mathbf F)$, $\tilde\Lambda_F'\in\tilde\Psi^0(\mathbf F,F)$ such that $\tilde\Lambda_E\tilde\Lambda_E'-\id$, $\tilde\Lambda_E'\tilde\Lambda_E-\id$, $\tilde\Lambda_F\tilde\Lambda_F'-\id$, and $\tilde\Lambda_F'\tilde\Lambda_F-\id$ are all smoothing operators.
Then $\mathbf A:=\tilde\Lambda_FA\tilde\Lambda_E'\in\Psi^k(\mathbf E,\mathbf F)$ has Heisenberg order $k$ in the ungraded sense, and
$$
\sigma_x^k(\mathbf A)=\tilde\sigma^k_x(\mathbf A)=\tilde\sigma^0_x(\tilde\Lambda_F)\tilde\sigma^k_x(A)\tilde\sigma^0_x(\tilde\Lambda_E'),
$$
see \eqref{E:grsigmaAB}.
Since $A$ satisfies the graded Rockland condition, and since $\tilde\sigma^0_x(\tilde\Lambda_F)$ and $\tilde\sigma^0_x(\tilde\Lambda_E')$ are invertible with inverses $\tilde\sigma^0_x(\tilde\Lambda_F')$ and $\tilde\sigma^0_x(\tilde\Lambda_E)$, respectively, we conclude that $\mathbf A$ satisfies the (ungraded) Rockland condition, see \eqref{E:grpiab}.
Hence, by Theorem~\ref{T:Rockland}, there exists a left parametrix $\mathbf B\in\Psi^{-k}_\prop(\mathbf F,\mathbf E)$ such that $\mathbf B\mathbf A-\id$ is a smoothing operator.
Putting $B:=\tilde\Lambda_E'\mathbf B\tilde\Lambda_F\in\tilde\Psi^{-k}(F,E)$ and using the fact that $\tilde\Lambda_E'\tilde\Lambda_E-\id$ is a smoothing operator, we see that $BA-\id$ is a smoothing operator.
Hence, $B$ is the desired left parametrix.
The hypoellipticity statements follow immediately from the pseudolocality of $B$ and the mapping property $B\colon\tilde H_\loc^{r-\Re(k)}(F)\to\tilde H^r_\loc(E)$.
Assume $M$ closed.
As in Corollary~\ref{C:hypo} we see that $\ker(A)$ is a finite dimensional subspace of $\Gamma^\infty(E)$.
For the maximal graded hypoelliptic estimate \eqref{E:grmhesti} we use boundedness of $B\colon\tilde H^{r-\Re(k)}(F)\to\tilde H^r(E)$ and the fact that smoothing operators induce bounded operators $\tilde H^{r'}(E)\to\tilde H^r(E)$.
To see the other hypoelliptic estimate, we consider the formal adjoint $A^\sharp\in\tilde\Psi^{\bar k}(F,E)$ with respect to inner products of the form \eqref{E:grHSllrr}, see \eqref{E:Asharp} with $s_1=s_2=s$.
Clearly, $A^\sharp A\in\tilde\Psi^{2\Re(k)}(E)$ satisfies the graded Rockland condition and $\ker(A^\sharp A)=\ker(A)$.
Hence, Corollary~\ref{C:graded-Hodge} below implies that $A^\sharp A+Q$ is invertible with inverse $(A^\sharp A+Q)^{-1}\in\tilde\Psi^{-2\Re(k)}(E)$.
Thus, $B':=(A^\sharp A+Q)^{-1}A^\sharp\in\tilde\Psi^{-k}(F,E)$ is a parametrix such that $B'A=\id-Q$, whence \eqref{E:grmheestiQ}.
\end{proof}

In view of Corollary~\ref{C:graded-regularity}, the graded Rockland condition implies Rumin's C-C ellipticity, cf.~\cite[Definition~5.1]{R01} or \cite[Section~2]{R99}.

We obtain the following generalization of Corollaries~\ref{C:smooth-Hodge-decomposition}, \ref{C:PsiinvA}, and \ref{C:HsHodge}.
A Hodge decomposition for the Rumin complex on an equiregular C-C manifold has been established in \cite[Proposition~3.6]{R00}.

\begin{corollary}[Graded Hodge decomposition]\label{C:graded-Hodge}
Let $E$ be a filtered vector bundle over a closed filtered manifold $M$.
Suppose $A\in\tilde\Psi^k(E)$ satisfies the graded Rockland condition and is formally selfadjoint, $A^\sharp=A$, with respect to a graded Sobolev inner product of the form~\eqref{E:grHSllrr}.
Moreover, let $Q$ denotes the orthogonal projection onto the (finite dimensional) subspace $\ker(A)\subseteq\Gamma^\infty(E)$ with respect to the inner product \eqref{E:grHSllrr}.
Then $A+Q$ is invertible with inverse $(A+Q)^{-1}\in\tilde\Psi^{-k}(E)$.
Consequently, we have topological isomorphisms and Hodge type decompositions:
\begin{align*}
A+Q\colon\Gamma^\infty(E)&\xrightarrow\cong\Gamma^\infty(E)&\Gamma^\infty(E)&=\ker(A)\oplus A(\Gamma^\infty(E))\\
A+Q\colon\tilde H^r(E)&\xrightarrow\cong\tilde H^{r-\Re(k)}(E)&\tilde H^{r-\Re(k)}(E)&=\ker(A)\oplus A(\tilde H^r(E))\\
A+Q\colon\Gamma^{-\infty}(E)&\xrightarrow\cong\Gamma^{-\infty}(E)&\Gamma^{-\infty}(E)&=\ker(A)\oplus A(\Gamma^{-\infty}(E))
\end{align*}
\end{corollary}

\begin{proof}
Let $\mathbf E$ denote the vector bundle $E$ equipped with the trivial filtration, $\mathbf E=\mathbf E^0\supseteq\mathbf E^1=0$.
Recall that $A^\sharp=\tilde\Lambda_s^{-1}(\tilde\Lambda_sA\tilde\Lambda_s^{-1})^*\tilde\Lambda_s$, see~\eqref{E:Asharp}.
Hence, the assumption $A^\sharp=A$ implies that $\mathbf A:=\tilde\Lambda_sA\tilde\Lambda_s^{-1}\in\Psi^k(\mathbf E)$ is formally selfadjoint with respect to the $L^2$ inner product~\eqref{E:llrr}, that is, $\mathbf A^*=\mathbf A$.
Moreover, $\mathbf A$ satisfies the (ungraded) Rockland condition for $\tilde\sigma^s_x(\tilde\Lambda_s)$ is invertible with inverse $\tilde\sigma^{-s}_x(\tilde\Lambda_s^{-1})$, see~\eqref{E:grsigmaAB}.
Hence, according to Corollary~\ref{C:PsiinvA}, $\mathbf A+\mathbf Q\in\Psi^k(\mathbf E)$ is invertible with inverse $(\mathbf A+\mathbf Q)^{-1}\in\Psi^{-k}(\mathbf E)$, where $\mathbf Q\in\mathcal O^{-\infty}(\mathbf E)$ denotes the orthogonal projection onto $\ker(\mathbf A)$, a finite dimensional subspace of $\Gamma^\infty(\mathbf E)$.
Note that $Q:=\tilde\Lambda_s^{-1}\mathbf Q\tilde\Lambda_s\in\mathcal O^{-\infty}(E)$ is the orthogonal projection onto $\ker(A)$ with respect to the inner product \eqref{E:grHSllrr}.
Conjugating with $\tilde\Lambda_s$, we conclude that $A+Q\in\tilde\Psi^k(E)$ is invertible with inverse $(A+Q)^{-1}=\tilde\Lambda_s^{-1}(\mathbf A+\mathbf Q)^{-1}\tilde\Lambda_s\in\tilde\Psi^{-k}(E)$.
The remaining assertions follow at once.
\end{proof}

We have the following generalization of Corollary~\ref{C:Fredholm}:

\begin{corollary}[Fredholm operators and index]\label{C:graded-Fredholm}
Let $E$ and $F$ be a filtered vector bundles over a closed filtered manifold $M$.
Suppose $A\in\tilde\Psi^k(E,F)$ is such that $A$ and $A^t$ both satisfy the graded Rockland condition.
Then, for every real number $r$, we have an induced Fredholm operator $A\colon\tilde H^r(E)\to\tilde H^{r-\Re(k)}(F)$ whose index is independent of $r$ and can be expressed as
$$
\ind(A)=\dim\ker A-\dim\ker A^t.
$$
This index depends only on the graded Heisenberg principal cosymbol $\tilde\sigma^k(A)\in\tilde\Sigma^k(E,F)$.
\end{corollary}

\begin{proof}
Using Corollary~\ref{C:graded-regularity} we obtain a parametrix $B\in\tilde\Psi^{-k}(F,E)$ such that $BA-\id$ and $AB-\id$ are both smoothing operators, cf.\ Remark~\ref{R:AAtRock}.
We may now proceed exactly as in the proof of Corollary~\ref{C:Fredholm}.
\end{proof}

\subsection{Graded Rockland sequences}\label{SS:gRs}

Throughout this section we assume $M$ to be a closed filtered manifold.
Suppose $E_i$ are filtered vector bundles over $M$, and consider a sequence 
\begin{equation}\label{E:grRseq}
\cdots\to\Gamma^\infty(E_{i-1})\xrightarrow{A_{i-1}}\Gamma^\infty(E_i)\xrightarrow{A_i}\Gamma^\infty(E_{i+1})\to\cdots
\end{equation}
where $A_i\in\tilde\Psi^{k_i}(E_i,E_{i+1})$ for some complex numbers $k_i$.
Generalizing Definition~\ref{D:graded_hypoelliptic_seq} for sequences of differential operators, we make the following

\begin{definition}[Graded Rockland sequence]\label{D:gRs}
A sequence of operators as above is said to be a \emph{graded Rockland sequence} if, for every $x\in M$ and every non-trivial irreducible unitary representation $\pi\colon\mathcal T_xM\to U(\mathcal H)$, the sequence
$$
\cdots\to\mathcal H_\infty\otimes\gr(E_{{i-1,x}})\xrightarrow{\bar\pi(\tilde\sigma^{k_{i-1}}_x(A_{i-1}))}
\mathcal H_\infty\otimes\gr(E_{i,x})\xrightarrow{\bar\pi(\tilde\sigma^{k_i}_x(A_i))}
\mathcal H_\infty\otimes\gr(E_{i+1,x})\to\cdots
$$
is weakly exact, that is, the image of each arrow is contained and dense in the kernel of the subsequent arrow.
Here $\mathcal H_\infty$ denotes the subspace of smooth vectors in $\mathcal H$.
\end{definition}

Suppose the sequence in \eqref{E:grRseq} is a Rockland sequence.
Fix real numbers $s_i$ such that $\Re(k_i)+s_{i+1}-s_i$ is independent of $i$, and put
\begin{equation}\label{E:kappa}
\kappa:=\Re(k_i)+s_{i+1}-s_i.
\end{equation}
Let $\mathbf E_i$ denote the vector bundle $E_i$ considered as a trivially filtered bundle, that is, equipped with the filtration $\mathbf E_i=\mathbf E_i^0\supseteq\mathbf E^1_i=0$.
Fix a smooth volume density on $M$ as well as smooth fiber-wise Hermitian inner products $h_i$ on $\mathbf E_i$ and let $\llangle-,-\rrangle_{L^2(\mathbf E_i)}$ denote the associated $L^2$ inner product on sections of $\mathbf E_i$, see \eqref{E:llrr}.
Moreover, choose invertible $\tilde\Lambda_i\in\tilde\Psi^{s_i}(E_i,\mathbf E_i)$ with $\tilde\Lambda_i^{-1}\in\tilde\Psi^{-s_i}(\mathbf E_i,E_i)$, see Lemma~\ref{L:grLambda}, and consider the associated graded Sobolev inner product on sections of $E_i$,
\begin{equation}\label{E:grHSllrri}
\llangle\psi_1,\psi_2\rrangle_{\tilde H^{s_i}(E_i)}
:=\llangle\tilde\Lambda_i\psi_1,\tilde\Lambda_i\psi\rrangle_{L^2(\mathbf E_i)}
=\langle\tilde\Lambda_i^t(h_i\otimes dx)\tilde\Lambda_i\psi_1,\psi_2\rangle
\end{equation}
where $\psi_1,\psi_2\in\Gamma^\infty(E_i)$.
Let $A_i^\sharp\in\tilde\Psi^{2\kappa-k_i}(E_{i+1},E_i)$, 
$$
A_i^\sharp
=\tilde\Lambda_i^{-1}(\tilde\Lambda_{i+1}A_i\tilde\Lambda_i^{-1})^*\tilde\Lambda_{i+1}
=(\tilde\Lambda_i^t(h_i\otimes dx)\tilde\Lambda_i)^{-1}\,A^t_i\,\tilde\Lambda_{i+1}^t(h_{i+1}\otimes dx)\tilde\Lambda_{i+1}
$$
denote the formal adjoint of $A_i$ with respect to these inner products, that is, 
$$
\llangle A_i^\sharp\phi,\psi\rrangle_{\tilde H^{s_i}(E_i)}=\llangle\phi,A_i\psi\rrangle_{\tilde H^{s_{i+1}}(E_{i+1})}
$$ 
for all $\psi\in\Gamma^\infty(E_i)$ and $\phi\in\Gamma^\infty(E_{i+1})$.
Let us finally consider the Laplace type operators 
$$
B_i:=A_{i-1}A_{i-1}^\sharp+A_i^\sharp A_i.
$$
Note that $B_i=B_i^\sharp\in\tilde\Psi^{2\kappa}(E_i)$, see \eqref{E:kappa} and \eqref{E:BAsharp}.

\begin{lemma}\label{L:Delta}
The operator $B_i$ satisfies the graded Rockland condition.
\end{lemma}

\begin{proof}
Note that $\mathbf B_i:=\tilde\Lambda_iB_i\tilde\Lambda_i^{-1}\in\Psi^{2\kappa}(\mathbf E_i)$ is of the form $\mathbf B_i=\mathbf A_i^*\mathbf A_i+\mathbf A_{i-1}\mathbf A_{i-1}^*$ where $\mathbf A_i:=\tilde\Lambda_{i+1}A_i\tilde\Lambda_i^{-1}\in\Psi^{k_i+s_{i+1}-s_i}(\mathbf E_i,\mathbf E_{i+1})$.
Using \eqref{E:piab}, \eqref{E:pia*}, and Remark~\ref{R:Psi:adjoint}, we obtain
\begin{multline*}
\bar\pi(\sigma^{2\kappa}_x(\mathbf B_i))
=\bar\pi(\sigma^{k_{i-1}+s_i-s_{i-1}}_x(\mathbf A_{i-1}))\bar\pi(\sigma^{k_{i-1}+s_i-s_{i-1}}_x(\mathbf A_{i-1}))^*
\\+\bar\pi(\sigma^{k_i+s_{i+1}-s_i}_x(\mathbf A_i))^*\bar\pi(\sigma^{k_i+s_{i+1}-s_i}_x(\mathbf A_i)).
\end{multline*}
Since the operators $\mathbf A_i$ form an (ungraded) Rockland sequence, one readily concludes that $\mathbf B_i$ satisfies the (ungraded) Rockland condition, cf.\ the proof of Lemma~\ref{L:rockseq}.
Clearly, this implies that $B_i$ satisfies the graded Rockland condition, see \eqref{E:grpiab}.
\end{proof}

In view of Lemma~\ref{L:Delta}, Corollary~\ref{C:graded-Hodge} applies to each of the operators $B_i$ and we obtain the following generalization of Corollaries~\ref{C:RShypo}, \ref{C:Hodge}, \ref{C:regrockseq}, and \ref{C:HsHodge-seq}:

\begin{corollary}\label{C:grHodge}
The operator $(A_{i-1}^\sharp,A_i)$ is hypoelliptic. 
More precisely, if $\psi\in\Gamma^{-\infty}(E_i)$ is such that $A_{i-1}^\sharp\psi\in\tilde H^{r-2\kappa+\Re(k_{i-1})}(E_{i-1})$ and $A_i\psi\in\tilde H^{r-\Re(k_i)}(E_{i+1})$, then $\psi\in\tilde H^r(E_i)$.
Moreover, there exists a constant $C=C_{A_i,r}\geq0$ such that the maximal graded hypoelliptic estimate
\begin{equation*}
\|\psi\|_{\tilde H^r(E_i)}\leq C\left(\|A^\sharp_{i-1}\psi\|_{\tilde H^{r-2\kappa+\Re(k_{i-1})}(E_{i-1})}+\|Q_i\psi\|_{\ker(B_i)}+\|A_i\psi\|_{\tilde H^{r-\Re(k_i)}(E_{i+1})}\right)
\end{equation*}
holds for all $\psi\in\tilde H^r(E_i)$.
Here $r$ is any real number, $Q_i$ denotes the orthogonal projection onto the (finite dimensional) subspace $\ker(B_i)\subseteq\Gamma^\infty(E_i)$ with respect to the inner product~\eqref{E:grHSllrri}, $\|-\|_{\ker(B_i)}$ denotes any norm on $\ker(B_i)$, and $\|-\|_{\tilde H^r(E_i)}$ is any norm generating the Hilbert space topology on the graded Sobolev space $\tilde H^r(E_i)$. 
Furthermore,
$$
\ker(B_i)=\ker(A_{i-1}^\sharp|_{\Gamma^{-\infty}(E_i)})\cap\ker(A_i|_{\Gamma^{-\infty}(E_i)})
=\ker(A_{i-1}^\sharp|_{\Gamma^\infty(E_i)})\cap\ker(A_i|_{\Gamma^\infty(E_i)}).
$$
If, moreover, $A_iA_{i-1}=0$, then we have Hodge type decompositions
\begin{align*}
\Gamma^\infty(E_i)&=A_{i-1}(\Gamma^\infty(E_{i-1}))\oplus\ker(B_i)\oplus A_i^\sharp(\Gamma^\infty(E_{i+1}))\\
\tilde H^r(E_i)&=A_{i-1}\bigl(\tilde H^{r+\Re(k_{i-1})}(E_{i-1})\bigr)\oplus\ker(B_i)\oplus A_i^\sharp\bigl(\tilde H^{2\kappa-\Re(k_i)}(E_{i+1})\bigr)\\
\Gamma^{-\infty}(E_i)&=A_{i-1}(\Gamma^{-\infty}(E_{i-1}))\oplus\ker(B_i)\oplus A_i^\sharp(\Gamma^{-\infty}(E_{i+1}))
\end{align*}
as well as:
\begin{align*}
\ker(A_i|_{\Gamma^\infty(E_i)})&=A_{i-1}(\Gamma^\infty(E_{i-1}))\oplus\ker(B_i)\\
\ker(A_i|_{\tilde H^r(E_i)})&=A_{i-1}(\tilde H^{r+\Re(k_{i-1})}(E_{i-1}))\oplus\ker(B_i)\\
\ker(A_i|_{\Gamma^{-\infty}(E_i)})&=A_{i-1}(\Gamma^{-\infty}(E_{i-1}))\oplus\ker(B_i)
\end{align*}
In particular, every cohomology class admits a unique harmonic representative:
$$
\frac{\ker(A_i|_{\Gamma^{-\infty}(E_i)})}{\img(A_{i-1}|_{\Gamma^{-\infty}(E_{i-1})})}
=\frac{\ker(A_i|_{\Gamma^\infty(E_i)})}{\img(A_{i-1}|_{\Gamma^\infty(E_{i-1})})}
=\ker(B_i)
=\ker(A_{i-1}^\sharp)\cap\ker(A_i).
$$
\end{corollary}

\subsection*{Competing interests declaration}
The authors have no relevant financial or non-financial interests to disclose.

\subsection*{Data availability statement}
Data sharing not applicable to this article as no datasets were generated or analysed during the current study.

\subsection*{Acknowledgments}

S.~D.\ was supported by the Austrian Science Fund (FWF) grants P24420 and P28770.
The second author would like to express his gratitude to the Max Planck Institute for Mathematics in Bonn for the hospitality and financial support during an extended visit which enabled him to pursue this project.
Moreover, he gratefully acknowledges the support of the Austrian Science Fund (FWF) grant Y963, the START-Program of Michael Eichmair.
During the final stages of preparation the second author was supported by the Austrian Science Fund (FWF) grant P31663.
We thank anonymous referees for a careful reading of the manuscript and helpful suggestions regarding the presentation.
We thank Robert Yuncken for providing us with a copy of Melin's otherwise unavailable manuscript.
This version of the article has been accepted for publication, after peer review but is not the Version of Record and does not reflect post-acceptance improvements, or any corrections. 
The Version of Record is available online at: \url{http://dx.doi.org/10.1007/s10455-022-09870-0}

\bibliographystyle{abbrv}
  \bibliography{hypo}

\end{document}